\documentclass[11pt, reqno]{amsart}
\usepackage{amsfonts}
\usepackage{hyperref}
\usepackage{graphicx}
\usepackage{tabularx}
\usepackage{array}
\usepackage[usenames,dvipsnames]{color}
\usepackage[T1]{fontenc}
\usepackage[utf8]{inputenc}
\usepackage[russian, french, english]{babel}
\usepackage{comment}
\usepackage{amsmath}
\usepackage{amsthm}
\usepackage{amssymb}
\usepackage{fullpage}
\usepackage{xcolor}
\usepackage{tikz}
\usepackage{listings}
\usepackage{soul}
\usepackage[noindentafter]{titlesec}
\usepackage{mathtools}
\usepackage{enumerate, enumitem}

\tikzstyle{legend_general}=[rectangle, rounded corners, thin,
                          top color= white,bottom color=lavander!25,
                          minimum width=2.5cm, minimum height=0.8cm,
                          violet]
\hypersetup{
    colorlinks=true,
    urlcolor=blue,
    linkcolor=blue,
    citecolor=OliveGreen,
    linktoc=page,
}

\usepackage[backend=biber, style=numeric, sorting=nyt, maxnames=100,backref=true,sortcites = true, firstinits=true]{biblatex}
\newcommand{\E}{\mathbb{E}}

\DeclareMathOperator{\keep}{Keep}

\DeclareMathOperator{\dist}{dist}

\DeclareMathOperator{\Eq}{Eq}
\renewcommand{\epsilon}{\varepsilon}

\newtheorem{theorem}{Theorem}[section]
\newtheorem*{theorem*}{Theorem}
\newtheorem{proposition}[theorem]{Proposition}

\newtheorem{conjecture}[theorem]{Conjecture}

\newtheorem{lemma}[theorem]{Lemma}
\newtheorem*{lemma*}{Lemma}
\newtheorem{corollary}[theorem]{Corollary}
\theoremstyle{definition}
\newtheorem{definition}[theorem]{Definition}

\theoremstyle{claim}

\newcounter{ForClaims}[section]
\newtheorem{claim}{Claim}[ForClaims]

\newtheorem*{claim*}{Claim}

\newcommand*{\myproofname}{Proof of claim}
\newenvironment{claimproof}[1][\myproofname]{\begin{proof}[#1]}{\end{proof}}

\titleformat{\section}[block]{\scshape\filcenter}{\thesection.}{1ex}{}
\titleformat{\subsection}[block]{\bfseries}{\thesubsection.}{1ex}{}
\titleformat{\subsubsection}[runin]{\itshape}{\bfseries\upshape\thesubsubsection.}{1ex}{}[.---]
\titleformat{\paragraph}[runin]{\normalfont\normalsize\bfseries}{}{1em}{}
\titlespacing*{\paragraph}{0pt}{3.25ex plus 1ex minus .2ex}{1em}

\usepackage{algorithm}
\usepackage{algpseudocode}
\counterwithin{algorithm}{section}



\title{The strong chromatic index of $K_{t,t}$-free graphs}

\thanks{Peter Bradshaw and Abhishek Dhawan received funding from NSF RTG grant DMS-1937241.}

\author{Richard Bi}
\address{Department of Mathematics, University of Illinois Urbana-Champaign}
\email{rbi3@illinois.edu}

\author{Peter Bradshaw}
\address{Department of Mathematics, University of Illinois Urbana-Champaign}
\email{pb38@illinois.edu}

\author{Abhishek Dhawan}
\address{Department of Mathematics, University of Illinois Urbana-Champaign}
\email{adhawan2@illinois.edu}

\author{Jingwei Xu}
\address{Department of Mathematics, University of Illinois Urbana-Champaign}
\email{jx6@illinois.edu}

\begin{document}
\begin{abstract}
    A \emph{strong edge coloring} of a graph $G$ is an edge coloring $\phi:E(G) \rightarrow \mathbb N$ such that each color class forms an induced matching in $G$. The \emph{strong chromatic index} of $G$, written $\chi'_s(G)$, is the minimum number of colors needed for a strong edge coloring of $G$. Erd\H{o}s and Ne\v{s}et\v{r}il conjectured in 1985 that if $G$ has maximum degree $d$, then $\chi'_s(G) \leq \frac 54 d^2$.

    Mahdian showed in 2000 that if $G$ is $C_4$-free, then $\chi'_s(G) \leq (2+o(1)) \frac{d^2}{\log d}$, and he conjectured that the same upper bound holds for $K_{t,t}$-free graphs. In this paper, we prove this conjecture and improve upon it to show the following: every $K_{t,t}$-free graph $G$ of maximum degree $d$ satisfies  $\chi'_s(G) \leq (1+o(1)) \frac{d^2}{\log d}$. We employ a variant of the R\"odl nibble method to prove this result. The key new ingredient in our adaptation of the method is an application of the K\H{o}v\'ari-S\'os-Tur\'an theorem to show that $H \coloneqq L(G)^2$ satisfies certain structural properties. These properties, in conjunction with a variant of Talagrand's inequality to handle exceptional outcomes, allow us to concentrate the sizes of certain vertex sets through the nibble, even when these vertex sets have order smaller than the maximum codegree of $H$. We encapsulate these structural properties into a more general statement on list coloring that we believe to be of independent interest. In light of the conjectured computational threshold for coloring random graphs arising in average-case complexity theory, we suspect that our result is best possible using this approach.
\end{abstract}
\maketitle

\section{Introduction}\label{section: intro}

\subsection{Background and main result}\label{subsection: background}

A \emph{strong edge coloring} of a graph $G$ is a coloring $\phi\,:\,E(G) \rightarrow \mathbb N$ with the property that if $\phi(e) = \phi(e')$, then no endpoint of $e$ is adjacent to an endpoint of $e'$. Equivalently, a strong edge coloring of $G$ is a proper coloring of $L(G)^2$, the square of the line graph of $G$. 
In a strong edge coloring of $G$, each color class is an induced matching, that is, a matching $M \subseteq E(G)$ with the property that $E(G[V(M)]) = M$.
The \emph{strong chromatic index} of $G$, denoted $\chi_s'(G)$, is the minimum number of colors needed for a strong edge coloring of $G$.

If $G$ is a graph of maximum degree $d$, then the square of the line graph of $G$ has maximum degree at most $2d(d-1)$, so Brooks' Theorem implies that $\chi'_s(G) \leq 2 d (d - 1)$. In 1985, Erd\H{o}s and Ne\v set\v ril proposed the following conjecture:

\begin{conjecture}\label{conj: erdos nesetril}
    Every graph $G$ of maximum degree $d$ satisfies $\chi'_s(G) \leq \frac 54 d^2$.
\end{conjecture}

If true, the above is best possible, as this value is attained by a \emph{$C_5$-blowup}, defined as a graph obtained from $C_5$ by replacing each vertex with an independent set of size $t$ and replacing each edge with the biclique $K_{t,t}$. 
Molloy and Reed \cite{MolloyReed} made the first asymptotic improvement to the trivial upper bound, proving that when $d$ is sufficiently large, $\chi_s'(G) \leq 1.998 d^2$.
This bound was subsequently improved by Bruhn and Joos \cite{BruhnJoos}, and again by Bonamy, Perrett, and Postle \cite{Bonamy}.
Currently, the best known upper bound is due to Hurley, de Verclos, and Kang \cite{Kang}, who showed that when $d$ is sufficiently large, $\chi_s'(G) \leq 1.772 d^2$ using a more general argument that applies to \textit{locally sparse graphs} $H$, i.e., where $H[N(v)]$ contains few edges for each vertex $v$.

When considering upper bounds for a graph coloring parameter on graphs of bounded degree, it is natural to narrow one's focus to graphs with specific structural properties. 
For example,  Molloy~\cite{MolloyTF} improved a result of Johansson \cite{Johansson} to show that triangle-free graphs of maximum degree $d$ have chromatic number at most $(1 + o(1))d/\log d$, which is an asymptotic improvement to the greedy bound of $d+1$.
Moving beyond triangles, the study of $F$-free graphs has seen a lot of interest in the last 30 years or so; we discuss highlights in Section~\ref{subsection: prior work}.

Following this trend, it is reasonable to ask whether such structural constraints lead to improved bounds on $\chi_s'(\cdot)$.
When considering the strong chromatic index of graphs with maximum degree $d$, no asymptotic improvement can be made by forbidding triangles, since the extremal example of Conjecture~\ref{conj: erdos nesetril} is itself $K_3$-free.
Instead, one aims to study graph classes forbidding certain subgraphs of a $C_5$-blowup.
For instance, Mahdian \cite{Mahdian} studied $C_4$-free graphs, proving the following result:

\begin{theorem}[{\cite{Mahdian}}]\label{thm:Mahdian}
    If $G$ is a $C_4$-free graph of maximum degree $d$, then $\chi_s'(G) \leq (2 +o(1)) d^2 / \log d$.
\end{theorem}

We discuss Mahdian's proof strategy in detail in Section~\ref{subsection: proof overview} as our approach builds upon his ideas.
Roughly speaking, he employs a variant of the R\"odl nibble method, a semirandom graph coloring procedure, on $H \coloneqq L(G)^2$, taking advantage of various codegree constraints satisfied by $H$ (the codegree of a pair of vertices $u, v \in V(H)$ is the size of the set $N(u) \cap N(v)$).
In fact, Theorem~\ref{thm:Mahdian} holds for $C_4$ replaced by $K_{2, t}$ (see the discussion at the end of \cite[Section 8]{Mahdian}; the proof is identical, \textit{mutatis mutandis}).
Replacing $K_{2, t}$ by $K_{t, t}$, however, leads to challenges.
Nevertheless, one can prove an asymptotic improvement to the greedy bound by showing that $H$ is locally sparse.
Indeed, if $G$ has no subgraph isomorphic to $K_{t,t}$, then it follows from the K\H{o}v\'ari-S\'os-Tur\'an theorem (Lemma~\ref{lem:KST}) that for each $v \in V(H)$, the number of edges in $N(v)$ is at most $O(d^{4-\frac{1}{t-1}})$.
Building upon the seminal works of Alon, Krivelevich, and Sudakov \cite{AKSConjecture} and Vu \cite{vu2002general}, Hurley and Pirot \cite{HurleyPirot} showed in 2021 that a graph of maximum degree at most $D$ in which the neighborhood of each vertex contains at most $D^2 /f$ edges (for $1 \leq f \leq D^2+1$) has chromatic number at most $(1+o(1)) D / \log f$ (see also \cite{davies2020graph}).
Given a $K_{t,t}$-free graph $G$ of maximum degree $d$, if we set $D = 2d^2$ and $f = \Omega( d^{\frac{1}{t-1}})$, then we obtain the following result:

\begin{theorem}\label{thm:Mahdian Ktt}
    If $G$ is a $K_{t,t}$-free graph of maximum degree $d$, then
    $\chi_s'(G) \leq (2t-2 +o(1)) \dfrac{d^2}{\log d}$.
\end{theorem}

Note that the above recovers Theorem~\ref{thm:Mahdian} for $t = 2$.
Mahdian conjectured that his result holds for all $t \geq 2$:

\begin{conjecture}[\cite{Mahdian}]\label{conj:Mahdian}
    Let $t \geq 2$ be fixed and let $G$ be a $K_{t, t}$-free graph of maximum degree $d$.
    Then $\chi'_s(G) \leq (2 + o(1))\dfrac{d^2}{\log d}$.
\end{conjecture}

Our main result improves the constant factor in Theorem~\ref{thm:Mahdian Ktt}, resolving Conjecture~\ref{conj:Mahdian} in a stronger form.

\begin{theorem}\label{thm:main}
    Let $\epsilon > 0$ and $t \geq 2$ be fixed.
    Let $G$ be a $K_{t, t}$-free graph of maximum degree $d$ sufficiently large in terms of $\epsilon$ and $t$.
    Then $\chi'_s(G) \leq (1+\epsilon) \dfrac{d^2}{\log d}$.
\end{theorem}

The proof of the above result follows the strategy of Mahdian, i.e., employing the nibble method on $L(G)^2$.
However, we require several modifications to incorporate the $K_{t, t}$-freeness constraint as well as to improve the constant factor from $2$ to $1$; see Section~\ref{subsection: proof overview} for an informal description of these modifications.

\subsection{Relation to prior work}\label{subsection: prior work}

When considering Conjecture~\ref{conj: erdos nesetril}, there are a number of results for special graph classes.
We highlight a few in this section and direct the reader to the survey by Cranston \cite{cranston2023coloring} for a more thorough history.
The conjecture was verified for graphs of maximum degree $3$ by Andersen \cite{andersen1992strong}, and
independently by Horek, Qing, and Trotter \cite{horak1993induced}. 
For $d = 4$, Cranston \cite{cranston2006strong} achieved a bound of $\chi_s'(G) \leq 22$, which is off by $2$.
Huang, Santana, and Yu \cite{huang2018strong} improved the bound to $21$ nearly answering the question in this setting.
Furthermore, a number of works address degenerate graphs, the earliest of which is by Faudree, Schelp, Gy\'arf\'as, and Tuza \cite{faudree1990strong}, who established the bound $\chi_s'(G) \leq 4d(G) + 4$ for planar graphs.
Kaiser and Kang \cite{kaiser2014distance} study distance-$t$ edge colorings, a generalization of strong edge colorings where alike colored edges have to be even farther apart; we discuss this further in Section~\ref{subsection: general statement}.

Mahdian's result on $C_4$-free graphs was the first paper to consider the $F$-freeness constraint in relation to Conjecture~\ref{conj: erdos nesetril}.
As mentioned earlier, his proof translates over almost verbatim to $K_{2, t}$-free graphs.
In the same paper, he conjectures that $\chi_s'(G) \leq d^2$ for $C_5$-free graphs $G$ of maximum degree $d$, which would be tight since $\chi_s'(K_{d,d}) = d^2$.
D{\k{e}}bski, Junosza-Szaniawski, and {\'S}leszy{\'n}ska-Nowak showed that $\chi_s'(G) \leq (2 - 1/(t-2))d^2$ when $G$ contains no \textit{induced} copy of $K_{1, t}$ for $t \geq 4$.
To the best of our knowledge, no other works have explored the (induced) $F$-freeness constraint with respect to this problem.

With regards to our proof strategy, we build upon a number of works from the past $\sim30$ years which employ the R\"odl nibble method to graph coloring problems.
Kim \cite{kim1995brooks} first applied this strategy in his proof of the bound $\chi(G) \leq (1 + o(1))d/\log d$ for graphs of girth at least $5$.
Johansson built upon Kim's work to prove $\chi(G) = O(d/\log d)$ for $K_3$-free graphs \cite{Johansson} and, more generally, $\chi(G) = O(d\log\log d/\log d)$ for $K_r$-free graphs \cite{JohanssonKr} (see \cite{MolloyTF, Bernshteyn} for simpler proofs of this result; see also \cite{davies2020graph, dhawan2025bounds} for more recent improvements in terms of the leading constant factor).
For $K_{t, t}$-free graphs, Anderson, Bernshteyn, and the third named author of this manuscript \cite{anderson2023colouring} employed a variant of Kim's nibble approach to prove that $\chi(G) \leq (1 + o(1)) d/\log d$, matching the best-known $K_3$-free upper bound.
In follow-up work \cite{anderson2025coloring}, they prove the bound $\chi(G) \leq (4 + o(1)) d/\log d$ for $K_{1, t, t}$-free graphs by adapting a strategy of Pettie and Su \cite{PS15} who consider the case where $t = 1$.

We conclude this section with a discussion of the optimality of our result with respect to our proof technique.
With regards to ordinary vertex coloring, it is known that there exist $d$-regular graphs of arbitrarily high girth with chromatic number at least $d/(2\log d)$ \cite{BollobasBound}.
However, proving that a pseudorandom graph class $\mathcal{G}$ satisfies $\chi(G) < \Delta(G)/\log \Delta(G)$ for all $G \in \mathcal{G}$ remains a tantalizing open problem.
It turns out that the value $\Delta(G)/\log \Delta(G)$ is a natural threshold for upper bounds on $\chi(G)$, coinciding with the so-called \textit{shattering threshold} for colorings of random graphs of average degree $\Delta$ \cite{Zdeborova,Achlioptas}.
It is conjectured that no efficient algorithms can beat the bound $\Delta/\log \Delta$ for random regular graphs, a heavily studied problem in average-case complexity theory (the study of NP-hard problems on random instances).
A folklore result states that proving this conjecture unconditionally is equivalent to proving a statement stronger than $P \neq NP$ and so researchers have focused on providing evidence toward it instead: for a number of restricted classes of algorithms, one cannot beat this bound \cite{RV, wein2020optimal}.
It turns out that such factor-$2$ gaps are ubiquitous in random optimization problems.
For a broad overview of such gaps both in the context of random graphs and beyond, we direct the reader to the surveys~\cite{bandeira2018notes,gamarnik2021overlap,gamarnik2022disordered,gamarnik2025turing}.

When considering strong edge colorings, Mahdian showed that by deleting a small number of edges from the sparse random graph $G(n,d/n)$, one can obtain a high-girth graph with maximum degree at most $(1+o_d(1))d$ and strong chromatic index at least $(1 + o(1)) d^2/(2\log d)$.
While this problem has not been studied from the lens of average-case complexity, one expects a similar gap to exist.
As we will see in the following subsection, our proof follows iterative applications of the Lov\'asz Local Lemma.
In particular, thanks to algorithmic variants of the local lemma, our proof yields an efficient algorithm to compute such a coloring.
Recalling the conjectured hardness of surpassing this bound, our result is optimal with respect to the approach.
Indeed, the nibble method falls under the framework of \textit{local algorithms}, which have been proven unable to surpass the shattering threshold for graph coloring \cite{RV} (the referenced result is stated in terms of independent sets, however, one simply obtains the coloring analogue as a corollary).

\subsection{Proof overview}\label{subsection: proof overview}

In this section, we will provide an overview of our proof techniques.
It should be understood that the presentation in this section deliberately ignores certain minor technical issues, and so the actual arguments and formal definitions given in the rest of the paper may be slightly different from how they are described here. 
However, the differences do not affect the general conceptual framework underlying our approach.

As mentioned earlier, we follow the strategy of Mahdian in his proof of Theorem~\ref{thm:Mahdian}.
For the remainder of this section, let us fix a graph $G$ and let $H \coloneqq L(G)^2$.
The goal is to find a proper vertex coloring of $H$ using few colors.
To do so, we employ a variant of the so-called ``R\"odl Nibble'' method, in which we randomly color a small portion of $V(H)$ and then iteratively repeat the same procedure with the vertices that remain uncolored; see \cite{Nibble} for a survey of applications of this method to (hyper)graph coloring.

Before we provide the details of our adaptation of the method, we make a few definitions.
Introduced independently by Vizing \cite{vizing1976coloring} and Erd\H{o}s, Rubin, and Taylor \cite{erdos1979choosability}, \textit{list coloring} is a generalization of graph coloring in which each vertex is assigned a color from its own predetermined list of colors. 
Formally, $L\,:\,V(H) \to 2^{\mathbb{N}}$
\footnote{For a set $S$, we let $2^S$ denote the powerset of $S$.} is a \textit{list assignment} for $H$, and an \textit{$L$-coloring} of $H$ is a proper coloring of $H$ such that each vertex $v \in V(H)$ receives a color from its list $L(v)$.
The \textit{list chromatic number} of $H$, denoted $\chi_\ell(H)$, is the minimum value $q$ such that $H$ admits a proper $L$-coloring for any list assignment $L$ satisfying $|L(v)|\geq q$ for all $v$.
Given a vertex $v \in V(H)$ and a color $c \in L(v)$, the \textit{color degree} of $c$ at $v$ is $d_L(v, c) = |N_L(v, c)|$ where $N_L(v, c)$ is the set of neighbors $u$ of $v$ such that $c \in L(u)$.

At the heart of our argument is a proposition (see Proposition~\ref{prop: nibble}) which roughly says the following:
Let $H$ be a graph and let $L\,:\, V(H) \to 2^{\mathbb{N}}$ be a list assignment for the vertices of $H$ satisfying certain properties (we are being deliberately vague at this point and will elaborate more on these properties shortly).
Then there exists a proper partial $L$-coloring $\phi$ of $H$ and a list assignment $L_\phi$ for $H_\phi \coloneqq H[V(H) \setminus \mathrm{dom}(\phi)]$ such that the following hold:
\begin{enumerate}
    \item\label{item: list condition} for each $v \in V(H_\phi)$, $L_{\phi}(v) \subseteq L(v) \setminus \phi(N(v))$, and
    \item\label{item: ratio condition} $\dfrac{\max_{v, c}d_{L_{\phi}}(v, c)}{\min_v|L_{\phi}(v)|} \leq \zeta \dfrac{\max_{v, c}d_{L}(v, c)}{\min_v|L(v)|}$, for some $\zeta < 1$.
\end{enumerate}
In particular, (1) $\phi \cup \phi'$ is a proper $L$-coloring of $H$ for any proper $L_\phi$-coloring $\phi'$ of $H_{\phi}$, and (2)~the ratio of the maximum color degree to the minimum list size decreases.
We prove the above result by employing Kim's variant of the ``Wasteful Coloring Procedure'' \cite{kim1995brooks} (see \cite[Chapter 12]{MolloyReed} for a textbook treatment of the argument).
Before we discuss the procedure in detail, let us show how this proposition will be used to prove our main result.
To prove Theorem~\ref{thm:main}, we aim to repeatedly apply the above result until we reach a stage where we can complete the coloring.
More formally, let $H_1 = H$ and $L_1(v) = [q]$ for each $v$, where $q = (1+\epsilon) \Delta(H) / \log \Delta(H)$.
Recursively define $H_{i+1}$ and $L_{i+1}$ by applying the above result to $H_i$ and $L_i$ until we have the following for some $i^* \in \mathbb{N}$:
\[\dfrac{\max_{v, c}d_{L_{i^*}}(v, c)}{\min_v|L_{i^*}(v)|} < \frac{1}{8}, \qquad \text{and} \qquad \min_v|L_{i^*}(v)| \gg 1.\]
The second condition above is necessary as we require nonempty lists.\footnote{We note that this condition may not be met if we were to use $q' = (1-\epsilon) \Delta(H) / \log \Delta(H)$ colors instead.
Indeed, for $q$ as defined, we can ensure $\min_v|L_{i^*}(v)| \geq \Delta(H)^{\Theta(\epsilon)} = \omega(1)$; while for $q'$, we cannot guarantee $L_{i^*}(v) \neq \emptyset$.} 
At this point, we may complete the coloring via the following proposition:

\begin{proposition}\label{prop: final blow}
    Let $H$ be a graph with a list assignment $L$ such that $|L(v)| \geq 8d$ for every $v \in V(H)$, where $d = \max_{v, c}d_L(v, c)$. 
    Then there exists a proper $L$-coloring of $H$.
\end{proposition}

This proposition is standard and proved using the Lov\'asz Local Lemma (see \cite[Chapter 4.1]{MolloyReed} for a textbook treatment).
We note that a result of Haxell \cite{haxell2001note} allows one to replace the constant $8$ with $2$.
We now describe the procedure we employ to prove the existence of the coloring $\phi$ satisfying conditions~\ref{item: list condition}~and~\ref{item: ratio condition} mentioned earlier (this constitutes a single nibble iteration).

\paragraph{The wasteful coloring procedure}
The procedure is a randomized algorithm for computing a proper partial coloring.
Informally, the algorithm is as follows:
\begin{enumerate}\label{pageref}
    \item Each vertex in $H$ is \emph{activated} with some small probability.
    \item Each activated vertex is assigned a color $\phi(v)$ uniformly at random from $L(v)$.
    \item For each vertex $v$, let $L_{\phi}(v) = L(v) \setminus \{\phi(u)\,:\, u \in N(v)\}$.
    \item If an activated vertex's color survives the previous step, it is assigned the color permanently.
\end{enumerate}
Note that the above procedure is wasteful in the sense that we may be deleting a color $c$ from $L(v)$ when no vertex in $N(v)$ is permanently assigned $c$.
It turns out that this ``wastefulness'' greatly simplifies the analysis of this procedure.
Specifically, quantities such as $\Pr(c\in L_{\phi}(v))$ have a much cleaner formulation. This streamlining simplifies the recursive analysis of parameters, making the computations considerably less cumbersome. 
(See \cite[Chapter 12.2]{MolloyReed} for a more in-depth discussion of the utility of such wastefulness).

Define the following parameters:
\[\ell \coloneqq \min_v|L(v)|, \qquad \ell' \coloneqq \min_v|L_\phi(v)|, \qquad d \coloneqq \max_{v, c}d_L(v, c), \qquad d' \coloneqq \max_{v, c}d_{L_{\phi}}(v, c).\]
The goal is to show that, with positive probability, the ratio $d'/\ell'$ is noticeably smaller than $d/\ell$.
It turns out that this holds locally in expectation, i.e.,
\[\frac{\mathbb{E}[\deg_{L_{\phi}}(v,c)]}{\mathbb{E}[|L_{\phi}(v)|]} \,\leq\, \zeta \, \frac{d}{\ell},\]
for some $\zeta < 1$.
To see this we note the following:
\begin{itemize}
    \item $L_\phi(v)$ consists of colors $c \in L(v)$ such that (1) no neighbor of $v$ is activated and assigned $c$.
    \item $N_{L_{\phi}}(v, c)$ consists of those vertices $u \in N_L(v, c)$ such that (2) no neighbor of $u$ is activated and assigned $c$, and (3) $u$ is not colored.
\end{itemize}
Conditions (1) and (2) above have the same probability of occurring.
Therefore, condition (3) is what determines $\zeta < 1$.

Recall the goal of showing that, with positive probability, we have $d'/\ell' \leq \zeta\,d/\ell$.
To do so, it is enough to show that the random variables $|L_{\phi}(v)|$ and $d_{L_{\phi}}(v,c)$ are concentrated about their expected values.
The claim then follows by an application of the Lov\'asz Local Lemma.

\paragraph{Key challenge and our main contribution: concentration of $d'$}
The concentration of $\ell'$ follows a straightforward application of Talagrand's inequality, a powerful concentration tool for random variables satisfying certain Lipschitz-like conditions.
When considering $d'$, it turns out the relevant Lipschitz parameter is the maximum codegree of $H$, denoted $\Delta_2(H)$.
Recall the original application of this procedure by Kim: coloring graphs of girth at least $5$.
Here, $\Delta_2(H) = 1$ and so concentration is achieved by Talagrand's inequality.
The situation is a lot more delicate in other settings.

We begin by discussing Mahdian's approach toward proving Theorem~\ref{thm:Mahdian}.
Note that while $G$ is $C_4$-free, $H$ need not be!
In fact, $H$ can contain many copies of $C_4$.
The key structural property of $H$ that Mahdian takes advantage of is the following: for every vertex $v \in V(H)$, every vertex $u \in V(H) - v$ apart from $O(\sqrt{\Delta(H)})$ neighbors of $v$ (which we refer to as \textit{friends} of $v$) satisfy $|N(u) \cap N(v)| = O(\sqrt{\Delta(H)})$.
With this in hand, he is able to effectively ignore the contribution of the friends of $v$ to show that
\[d' \leq (1+o(1))\mathbb{E}[d'] + O(\sqrt{\Delta(H)}),\]
with high probability.
Unfortunately, this structural property is not hereditary, i.e, for $H' \subseteq H$ it does not hold with $\Delta(H)$ replaced by $\Delta(H')$.
For this reason, Mahdian needs to ensure that $\Delta(H_i) = \omega(\sqrt{\Delta(H_1)})$ for all $i$ when recursively applying the result.
This is where the constant factor $2$ is necessary in his argument.\footnote{We remark that in \cite{anderson2023colouring}, the authors employ a similar argument for $K_{t, t}$-free graphs $H$. As $K_{t, t}$-freeness is hereditary, the desired condition is met, allowing them to achieve the optimal constant $1$.}

In our proof of Theorem~\ref{thm:main}, we prove a similar structural property for $H$: for every vertex $v \in V(H)$, every vertex $u \in V(H) - v$ apart from $O(\Delta(H)^{1-1/t})$ friends of $v$ satisfy $|N(u) \cap N(v)| = O(\Delta(H)^{1-1/t})$.
With this in hand, an identical approach to Mahdian's would prove Theorem~\ref{thm:Mahdian Ktt}.
To improve the constant factor from $2t-2$ to $1$, we need to additionally control two parameters for each vertex $v$:
\begin{itemize}
    \item the size of the set $Q(u, v) = N(u) \cap N(v)$ for \textit{stranger} vertices $u$, and
    \item the size of the set $X_v$ denoting the friends of $v$.
\end{itemize}
Our proof of the first bullet is inspired by recent work of Methuku, Wigal, and the second and third named authors of this manuscript \cite{BDMW}.
The proof of the second bullet represents the key technical novelty of this work.

Let $v \in V(H)$ be fixed and consider the set $X_v$.
A curious feature of the codegree concentration in \cite{BDMW} is that it relies on the interaction between vertices inside the set of interest as opposed to the influence of those outside as has been standard in earlier such approaches.
It turns out that there exists a partition $X_{1, v} \sqcup\cdots \sqcup X_{s, v}$ of $X_v$ and a collection of sets $B_{1, v}, \ldots, B_{s, v}$ such that the following hold for each $i \in [s]$:
\begin{enumerate}
    \item $X_{i, v}$ is not ``too large,''
    \item ``most'' of the neighbors of the vertices in $X_{i, v}$ lie in $B_{i, v}$, and
    \item the codegrees of pairs of vertices in $X_{i,v}$ with respect to the graph $H[X_{i,v},B_{i,v}]$ is ``small.''
\end{enumerate}
Hopefully, we can restrict our attention to $B_{i,v}$ to concentrate $|X_{i,v}|$, which implies concentration of $|X_v|$.
As it happens, initially, $|B_{i,v}|$ is ``large.''
Furthermore, we are able to show that $|B_{i,v}|$ does not shrink ``too fast'' through the applications of the Wasteful Coloring Procedure and hence, concentrate $|X_{i,v}|$. (We remark that the actual argument for $B_{i,v}$ is a little more subtle, but unimportant for this overview.)

In summary: to concentrate $d'$, we must show that $|Q(u, v)|$ is small for all pairs of strangers $u$ and $v$, and $X_v$ is small for all $v$; to assist with the latter, we must also show that $|B_{i,v}|$ does not shrink too fast for each $i \in [s]$.
It turns out one can encapsulate all of these conditions into a more general theorem.
We state this general result in the following subsection and certain consequences of it, which we prove later in the paper.
We remark that while the statement of this general theorem is tedious, we believe it to be of independent interest.

\subsection{A general theorem on list-coloring}\label{subsection: general statement}
In the following theorem, we consider a graph in which certain unordered pairs of vertices are \emph{friends}.
We assume that a vertex is always its own friend.
Formally, a friend relation is a symmetric reflexive binary relation $\mathcal F \subseteq V(G) \times V(G)$.
Given a graph $G$ with a friend relation $\mathcal F$, we say that two vertices that are not friends are \emph{strangers}.

\begin{theorem}
\label{thm:general}
Let $\gamma, \epsilon \in (0,1)$ and $N \in \mathbb N$.
Then, there exists $\Delta_0 \in \mathbb N$ for which the following holds.
Suppose $G$ is a graph of maximum degree at most $\Delta \geq \Delta_0$
with a symmetric reflexive friend relation
$\mathcal F \subseteq V(G) \times V(G)$, satisfying the following:
\begin{enumerate}
    \item 
    \label{item:main-1}
    Any two strangers have at most $\frac{\Delta}{\log^{20} \Delta}$ common neighbors.
    \item \label{item:main-2} There is a family $\mathcal X$ of subsets $X \subseteq V(G)$ satisfying the following:
    \begin{enumerate}
        \item \label{item:main-2a} Each $X \in \mathcal X$ satisfies $|X| \leq \Delta^{\gamma}$,
        \item \label{item:main-2b} For each $v \in V(G)$, there exists a family $\mathcal X_v \subseteq \mathcal X$ of size at most $\Delta^{(1 - \gamma )\epsilon/12}$
        such that all friends of $v$ belong to $\bigcup_{X \in \mathcal X_v} X$,
        \item \label{item:main-2c} Each $v \in V(G)$ belongs to at most $\Delta^{N}$ sets $X \in \mathcal X$.
    \end{enumerate}
    \item \label{item:main-3}
    For each $X \in \mathcal X$ with at least two vertices, there is a set $B^X \subseteq V(G)$ satisfying the following:
    \begin{enumerate}
        \item \label{item:main-3a} Each $v \in X$ has at least $d(v) - \Delta(1 - \gamma + \log^{-10}\Delta)$
            neighbors in $B^X$,
        \item \label{item:main-3b} For each distinct friend pair $u,v \in X$, $u$ and $v$ have at most $\frac{\Delta}{\log^{20} \Delta}$ common neighbors in $B^X$. 
    \end{enumerate}
\end{enumerate}
Then, $\chi_{\ell} (G) \leq (1 + \epsilon) \frac{\Delta}{\log \Delta}$.
\end{theorem}

Let us provide some intuition for the above result.
Condition~\eqref{item:main-1} states that stranger pairs of vertices have low codegree.
Condition~\eqref{item:main-2} states that there is a collection $\mathcal{X}$ of ``small'' subsets of $V(G)$ such that for each vertex $v$, (a)~the friends of $v$ are covered by a ``small'' subset of $\mathcal{X}$, and (b)~$v$ is not contained in ``too many'' subsets in $\mathcal{X}$.
Condition~\eqref{item:main-3} states that for each $X \in \mathcal{X}$, there is a set $B^X \subseteq V(G)$ such that the subgraph $G_X\coloneqq G[X, B^X]$ satisfies (a)~$\deg_{G_X}(x)$ is ``large'' for each $x \in X$, and (b)~the codegree of $u, v \in X$ in $G_X$ is ``small'' for distinct friend pairs $u, v$. (These properties will be key to concentrating $|X|$ through the nibble procedure.)
We derive Theorem~\ref{thm:main} from Theorem~\ref{thm:general} in Section~\ref{subsection: main proof from general}.
The argument roughly applies Theorem~\ref{thm:general} with $\mathcal{X}$ consisting of the sets $X_{i, v}$ and $B^{X_{i, v}} = B_{i, v}$, where $X_{i, v}$ and $B_{i, v}$ are as described in Section~\ref{subsection: proof overview}.

As mentioned in the previous section, we believe Theorem~\ref{thm:general} is of independent interest.
Indeed, one recovers the classical result of Kim on graphs of girth at least $5$ \cite{kim1995brooks} by setting
\begin{itemize}
    \item $\mathcal{F} = \{(v,v): v \in V(G)\}$,
    \item $\mathcal X = \{\{v\} : v\in V(G)\}$,
    \item $\mathcal X_v = \{\{v\}\}$ for each $v \in V(G)$, and
    \item $\gamma, N$ arbitrary.
\end{itemize}
With the same parameterization, one also recovers (and slightly improves upon) a recent result of Methuku, Wigal, and the second and third named authors of this manuscript.

\begin{theorem}[\cite{BDMW}] \label{thm:codeg}
    If $G$ is a graph of maximum degree $d$ and with codegree at most $\frac{d}{\log^{20} d}$,
    then $\chi_{\ell}(G) \leq (1 + o(1)) \frac{d}{\log d}$.
\end{theorem}

Using a similar setup with $\gamma = 1 - \frac 1c$ and $\Delta = cd$,
one obtains the same result with a larger coefficient when 
we allow a small number of exceptions to the codegree bound for each vertex.

\begin{theorem}
\label{thm:bdd-friends}
    Let $c > 1$ be fixed. Suppose that $G$ is a graph of maximum degree $d$, and suppose that for each $v \in V(G)$, at most $d^{1 - 1/c}$ vertices $w \in V(G)$ satisfy $|N(v) \cap N(w)| > \frac{d}{\log^{20} d}$. Then, $\chi_{\ell}(G) \leq (c + o(1)) \frac{d}{\log d}$.
\end{theorem}
Mahdian showed that the assumption of Theorem~\ref{thm:bdd-friends} holds for $c = 2$, so Theorem~\ref{thm:bdd-friends} immediately implies
Theorem~\ref{thm:Mahdian}.

Moving beyond strong edge colorings, our general theorem also implies a result for distance-$t$ edge colorings of high girth graphs.
A distance-$t$ edge coloring of a graph $G$ is a coloring $\phi\,:\,E(G) \rightarrow \mathbb N$ with the property that if $\phi(e) = \phi(e')$, then the distance between any endpoint of $e$ and one of $e'$ is at least $t$.
(For $t = 2$, this is just a strong edge coloring.)
Kaiser and Kang showed that $d$-regular graphs with girth at least $2t+1$ admit a distance-$t$ edge coloring with $O(d^t / \log d)$ colors \cite[Theorem~1.2]{kaiser2014distance}.
In their proof, they show that $L(G)^t$ is locally-sparse and employ a result of Alon, Krivelevich, and Sudakov \cite{AKSConjecture} on the chromatic number of such graphs.
Employing the improved bound of Hurley and Pirot \cite{HurleyPirot} instead yields a hidden constant of $1+o(1)$ in their result.
We improve this constant to $(1+o(1))2/t$ which, by another result of Kaiser and Kang \cite[Proposition 1.3]{kaiser2014distance}, is optimal up to a factor of $2$.

\begin{theorem}\label{thm: girth}
    For each value $\epsilon > 0$ and integer $t \geq 2$,  there exists a value $d_0 \in \mathbb N$ such that the following holds for every $d \geq d_0$. 
    Let $G$ be a graph of maximum degree at most $d$ with girth at least $2t+1$. Then, $G$ admits a distance-$t$ edge coloring with at most $(1 +\epsilon) \frac{2 d^t}{t \log d}$ colors.
\end{theorem}

We also prove a similar result for graphs that are good spectral expanders.

\begin{theorem}\label{thm: spectral}
    For each $\epsilon > 0$ and $C \geq 2$, there exists a value $d_0 \in \mathbb N$ such that the following holds for every $d \geq d_0$. 
    If $G$ is a $d$-regular graph with $n \geq d^{2t}$ vertices such that the second largest eigenvalue of the adjacency matrix $A(G)$ satisfies $\lambda(G) \leq C\sqrt {d}$,
    then $G$ admits a distance-$t$ edge coloring with at most $(1+\epsilon) \frac{2d^t}{t \log d}$ colors.
\end{theorem}

The derivations of Theorems~\ref{thm:main},~\ref{thm: girth},~and~\ref{thm: spectral} from Theorem~\ref{thm:general} are in Section~\ref{section: corollaries}.
In light of these results, we believe Theorem~\ref{thm:general} will find applications to other graph coloring problems.

\subsubsection*{Structure of the paper}
The remainder of the paper is structured as follows: in Section~\ref{section: prelim}, we collect some of the basic notation and facts used in our proofs; in Section~\ref{section: corollaries}, we detail the derivations of our main results from Theorem~\ref{thm:general}; in Section~\ref{sec: general-proof}, we prove Theorem~\ref{thm:general} modulo a technical lemma, which we prove in Section~\ref{sec:great-expectations}.

\section{Preliminaries}\label{section: prelim}

We further split this section into two subsections: in the first, we introduce the basic notation and definitions we will use in the rest of the paper; in the next, we state certain lemmas that will be employed in our proofs.

\subsection{Notation}

For $n \in \mathbb N$, we let $[n] \coloneqq \{1, \ldots, n\}$.
Given a graph $G$,
we write $L(G)$ for the line graph of $G$, which is the graph on the vertex set $E(G)$ in which two vertices $e,e' \in E(G)$ are adjacent in $L(G)$ if and only if $e$ and $e'$ are incident in $G$. We write $G^t$ for the $t$-th power of $G$, 
defined as the graph on $V(G)$
in which $u,v \in V(G)$ are adjacent if and only if $\dist_G(u,v) \leq t$.

Given a vertex $v \in V(G)$, we write $E_1(v)$ for the set of edges incident to $v$. 
For $i \geq 2$, we write $E_i(v)$ for the edges in $E(G) \setminus (E_1(v) \cup \dots \cup E_{i-1}(v))$
that are incident to some edge of $E_{i-1}(v)$. 
For each $i \geq 1$, we write $N_i(v)$ for the set of vertices at distance exactly $i$ from $v$. In particular, $N(v) = N_1(v)$.

For the remainder of the paper, we ignore ceilings and floors wherever necessary with the understanding that it does not affect our overall arguments.

\subsection{Tools}
We establish some results that will be useful later. The inequalities below can be verified using basic calculus and Taylor expansions.

\begin{lemma}
\label{lem:1-x LB}
    For each constant $\alpha > 0$, if  $x > 0$ is sufficiently small, then 
    \[1 - x > \exp(-x (1+x)) > \exp(-(1+\alpha)x).\]
\end{lemma}

\begin{lemma}
\label{lem:1/(1+x)-LB}
    For all $x > 0$, $\frac{1}{1+x} > 1-x$.
\end{lemma}

The next lemma is often called Bernoulli's inequality.

\begin{lemma}
\label{lem:bernoulli}
    Let $x \in (0,1)$, and let $r$ be a positive integer. Then, 
    $(1 - x)^r \geq 1 - rx$.
    Furthermore, if $rx < \frac 12$, then 
    $(1-x)^{-r} \leq \frac 1{1-rx} < 1 + 2rx$.
\end{lemma}

The following is a well-known lemma of K\H{o}v\'ari, S\'os, and Tur\'an.
\begin{lemma}[\cite{KST}]
\label{lem:KST}
    Let $H$ be a bipartite graph with partite sets $A$ and $B$ of respective sizes $m$ and $n$. Suppose that $H$ has no subset $X \subseteq A$ of size $t$ and $Y \subseteq B$ of size $s$ inducing a $K_{s,t}$. Then,
    \[|E(H)| < (t-1)^{1/s} nm^{1-1/s} + (s-1) m.\]
\end{lemma}

For $C_{2k}$-free graphs, Bondy and Simmonovits proved the following improvement on the above (this result will be essential in our proof of Theorem~\ref{thm: girth}):

\begin{lemma}[\cite{bondy1974cycles}]\label{BondySimonovits}
    Let $H$ be an $n$-vertex $C_{2k}$-free graph for $k \geq 2$.
    Then,
     \[|E(H)| < 100kn^{1+1/k}.\]
\end{lemma}

The following result, often called the Hajnal-Szemer\'edi theorem,
establishes conditions for the existence of an \emph{equitable $k$-coloring}, defined as a $k$-coloring on a graph in which the sizes of any two color classes differ by at most $1$.

\begin{lemma}[\cite{Hajnal}]
\label{lem:HS}
    If $G$ is a graph with maximum degree at most $\Delta$, then $G$ 
    has an equitable $(\Delta+1)$-coloring.
\end{lemma}

We also utilize the Lov\'asz Local Lemma, which is stated as follows in
\cite[Theorem 6.11]{Mitzenmacher}:
\begin{lemma}\label{lem:LLL}
    Let $A_1, A_2, \ldots, A_n$ be events in a probability space. Suppose there exists $p \in [0,1)$ such that
    $\Pr(A_i) \leq p$
    for all $1 \leq i \leq n$. Further suppose each $A_i$ is mutually independent from all but fewer than $d$ other events $A_j$. If $4pd \le 1$, then with positive probability none of the events $A_1, \ldots, A_n$ occur.  
\end{lemma}

For concentrating random variables, we will employ two concentration inequalities.
First, we use the following form of the Chernoff bound, which can be found, for instance, in \cite[Chapter 4]{Mitzenmacher}.

\begin{lemma}
\label{lem:chernoff}
Let $Y$ be a random variable that is the sum of pairwise independent indicator variables, and let $\mu = \E[Y]$. Then for any value $\delta > 0$,
$$\Pr(Y > (1 + \delta) \mu) \leq \exp \left( - \frac{ \delta^2 \mu }{2 + \delta }\right).$$
\end{lemma}

Mitzenmacher \cite{Mitzenmacher} points out that 
it is enough to let $\mu \geq \E[Y]$.
We also use the concentration inequality \cite[Theorem~3.4]{li2022chromatic}, 
which is a special case of the linear Talagrand's inequality from \cite{delcourt2022finding}.
Before we state the inequality, we need to make a few definitions.

\begin{definition}\label{defn: r verifiable}
    Let $\{ (\Omega_i, \Sigma_i, \Pr_i) \}_{i=1}^n$ be a set of probability spaces, and let $(\Omega, \Sigma, \Pr)$ be their product space. Let $\Omega^* \subseteq \Omega$ be a set of \emph{exceptional outcomes}, and let $Y:\Omega \rightarrow \{0,1\}$ be a random variable. We say that $Y$ is \emph{$r$-verifiable} with verifier $R:\Omega \setminus \Omega^* \rightarrow 2^{\{1, \dots, n\}}$ if:
    \begin{itemize}
        \item $R(\omega) = \emptyset$ whenever $Y(\omega) = 0$; 
        \item $|R(\omega)| \leq r$ for every $\omega \in \Omega \setminus \Omega^*$ for which $Y(\omega) = 1$;
        \item If $Y(\omega) = 1$, then 
        $Y(\omega') = 1$ for every $\omega' \in \Omega \setminus \Omega^*$ satisfying $\omega_i = \omega_i'$ for all $i \in R(\omega)$.
    \end{itemize}
\end{definition}
We often say that a set $R(\omega)$ is a \emph{witness set} as 
the outcomes from the spaces $\Omega_i$
indexed by 
$R(\omega)$ \textit{witness} the event $Y(\omega) = 1$.

\begin{definition}\label{defn: rd observable}
    Let $\{ (\Omega_i, \Sigma_i, \Pr_i) \}_{i=1}^n$ be a set of probability spaces, and let $(\Omega, \Sigma, \Pr)$ be their product space. Let $\Omega^* \subseteq \Omega$ be a set of \emph{exceptional outcomes}, and let $r,d \geq 0$. We say that a random variable $Z$ is \emph{$(r,d)$-observable} with respect to $\Omega^*$ if:
    \begin{itemize}
        \item $Z = \sum_{j=1}^m Y_j$, where for every $j \in \{1, \dots, m\}$, $Y_j$ is a binary random variable in $\Omega$ that is $r$-verifiable with verifier $R_j$,
        \item For every $\omega \in \Omega \setminus \Omega^*$ and $i \in \{1, \dots, n\}$, 
        $i \in R_j(\omega )$ for at most $d$ values $j$.
    \end{itemize}
\end{definition}

The concentration inequality \cite[Theorem~3.4]{li2022chromatic} is stated as follows.

\begin{lemma}
\label{lem:conc-ineq}
     Let $\{ (\Omega_i, \Sigma_i, \Pr_i) \}_{i=1}^n$ be a set of probability spaces, and let $(\Omega, \Sigma, \Pr)$ be their product space. Let $\Omega^* \subseteq \Omega$ be a set of \emph{exceptional outcomes}, and let $r,d \geq 0$. Let $Z: \Omega \rightarrow \mathbb R_{\geq 0}$ be a nonnegative random variable. If $Z$ is $(r,d)$-observable with respect to $\Omega^*$, then for any $\tau$ satisfying
    \begin{equation}
    \label{eqn:conc-tau-sqrt}
     \tau > 96 \sqrt{rd \E[Z]} + 128 rd + 8 \Pr(\Omega^*)\sup(Z),
     \end{equation}
     the following concentration inequality holds:
     \[\Pr(|Z - \E[Z]| > \tau ) \leq 4 \exp \left ( - \frac{\tau^2}{8rd (4\E[Z] + \tau )} \right ) + 4 \Pr(\Omega^*).\]
\end{lemma}

We will apply Lemma~\ref{lem:conc-ineq} with the following 
probability space $(\Omega_i,\Sigma_i,\Pr_i)$ during the $i$-th iteration of the Wasteful Coloring Procedure (see Section~\ref{subsection: proof overview} for an informal description of the procedure).
For each vertex $v \in V(G)$, we define a set $\Omega_{i,v} = L_i(v) \cup \{0\}$. 
The element $0$ of $\Omega_{i,v}$ satisfies $\Pr_{i,v}(0) = 1 - \eta$ and corresponds to an outcome of the experiment in which $v$ is 
not activated.
Each element $c \in L_i(v)$ satisfies $\Pr_{i,v}(c) = \frac{\eta}{|L_i(v)|}$;
the element $c$
corresponds to an outcome of the experiment in which $v$ is activated and assigned $c$.
Then, $\Sigma_{i,v} = 2^{\Omega_{i,v}}$ is the set of probability-measurable subsets of $\Omega_{i,v}$, and $\Pr_{i,v}\,:\,\Sigma_{i,v} \rightarrow [0,1]$ gives the probabilities defined by the experiment above.
Hence, $(\Omega_{i,v},\Sigma_{i,v},\Pr_{i,v})$ gives a probability space describing the random activation and color assignment at an uncolored vertex $v$.

The product of the probability spaces $(\Omega_{i,v},\Sigma_{i,v},\Pr_{i,v})$ (denoted $(\Omega_i, \Sigma_i, \Pr_i)$) consists of all outcomes of the $i$-th iteration of our Wasteful Coloring Procedure.
When applying Lemma~\ref{lem:conc-ineq}, the set $\Omega_i^* \subseteq \Omega_i$ of exceptional events that we use depends on the random variable in question.
In general, we either choose $\Omega_i^* = \emptyset$, or we let $\Omega_i^*$ consist of unlikely outcomes of the procedure in which many vertices in some local part of $G$ are activated and assigned the same color.

\section{Corollaries of Theorem~\ref{thm:general}}\label{section: corollaries}

In this section, we will derive Theorems~\ref{thm:main},~\ref{thm: girth},~and~\ref{thm: spectral} from Theorem~\ref{thm:general}.
We further split this section into subsections to deal with each result separately.

\subsection{Proof of Theorem~\ref{thm:main}}\label{subsection: main proof from general}

Let $\epsilon > 0$ and $t \geq 2$ be fixed.
Let $G$ be a $K_{t,t}$-free graph of maximum degree $d$
(we reserve the symbol $\Delta$ as an upper bound for the maximum degree of $L(G)^2$, the square of the line graph of $G$).
We aim to show that $\chi'_s(G) \leq (1 + \epsilon) \frac{d^2}{\log d}$, provided that $d$ is sufficiently large.
For this, we will apply Theorem~\ref{thm:general} 
to $L(G)^2$
with our value of $\epsilon$, and with $\gamma = \frac 12$ and $\Delta = 2d^2$.

As $G$ can be embedded in a $d$-regular $K_{t,t}$-free graph, we assume without loss of generality that $G$ is $d$-regular.
We also assume that $d$ is sufficiently large in terms of $t$ and $\epsilon$.
In particular, we assume that 
$t = o \left ( \frac{\log d}{\log \log d} \right )$ so that
\begin{equation}
\label{eqn:t-bounds}
d^{1/t} = (\log d)^{\omega(1)}, \qquad \text{ and } \qquad (\log d)^t = d^{o(1)}.
\end{equation}

We will show that $L(G)^2$ satisfies the conditions of Theorem~\ref{thm:general}.
Before we define the friend relation for $V(L(G)^2) = E(G)$, we define one for $V(G)$.
We say that two vertices $u,v \in V(G)$ are \emph{friends}
if $|N(u) \cap N(v)| \geq \frac{d}{\log^{40} d}$.
We let every vertex be its own friend.
Then, given two edges $e, e' \in E(G)$, $e$ and $e'$ are friends if some endpoint of $e$ is a friend of some endpoint of $e'$ in $G$. 
In particular, if $e$ and $e'$ are incident, then $e$ and $e'$ are friends.
This gives a friend relation for the vertex set of $L(G)^2$.
We make some observations about friend pairs of vertices and edges.

\begin{lemma}
\label{lem:few-friends}
    Each $v \in V(G)$ has at most $4 \log^{40t} d $ friends.
\end{lemma}
\begin{proof}
    Let $B \coloneqq N(v)$ and
    let 
    $A \subseteq N(N(v))$ be the set of friends of $v$.
    Let $H$ be a bipartite graph with disjoint partite sets $A'$ and $B'$.
    For each $a \in A$, add a vertex $a'$ to $A'$,
    and for each $b \in B$, add a vertex $b'$ to $B'$.
    For each adjacent pair $(a,b) \in A \times B$,
    add an edge $a'b'$ to $H$.
    If a subset $X' \subseteq A'$ of size $2t$ and a subset $Y' \subseteq B'$ of size $t$ induce a $K_{t,2t}$ in 
    $H$, 
    then by taking the corresponding set $Y \subseteq B$ and a $t$-subset of the corresponding $X \subseteq A$ disjoint from $Y$,
    we find a $K_{t,t}$ in $G$, a contradiction.
    Therefore,
    no $X' \subseteq A'$ of size $2t$ and $Y' \subseteq B'$ of size $t$ induce a $K_{t,2t}$ in $H$.

    Now, 
    since each $a \in A$ is a friend of $v$,
    each $a' \in A'$ has degree at least $\frac{d}{\log^{40} d}$
    in $H$.
    Therefore, $|E(H)| \geq \frac{|A| d}{\log^{40} d}$.
    On the other hand, 
    by
    applying 
    Lemma~\ref{lem:KST} with $t$ replacing $s$ and $2t$ replacing $t$, and with $n = |B| = |N(v)| = d$ and $m = |A|$, we have
    \[
    \frac{|A| d}{\log^{40} d} \leq e(H) < (2t-1)^{1/t} d |A|^{1-1/t} + (t-1) |A|.
    \]
    If $(t-1)|A| > (2t-1)^{1/t}d |A|^{1-1/t}$, 
    then
    $
    \frac{|A| d}{\log^{40} d} <  2(t-1) |A|,
    $
    contradicting our upper bound on $t$.
    Therefore,
    \[
    \frac{|A| d}{\log^{40} d} < 2(2t - 1)^{1/t}  d |A|^{1-1/t} < 4 d |A|^{1-1/t}.
    \]
    Rearranging, we have $|A| < 4 \log^{40t} d$, completing the proof.
\end{proof}

For each edge $e \in E(G)$, we write $N(e) = N_{L(G)^2}(e)$ for the neighbors of the vertex $e$ in $L(G)^2$.

\begin{lemma}
\label{lem:Q1-non-incident}
    If $e,e' \in E(G)$ are strangers, then $|N(e) \cap N(e')| < \frac{\Delta}{\log^{20} \Delta}$.
\end{lemma}
\begin{proof}
    By~\eqref{eqn:t-bounds}, we have $d^{\frac{1}{t}} \geq \log^{40}d$.
    Let $e = uv, e' = u'v'$. 
    We count the edges
     $xy\in N(e)\cap N(e')$.
    There are at most $4 d$
    edges $xy$ for which $x$ or $y$ belongs to $\{u,v,u',v'\}$.
    Since $e$ and $e'$ are strangers,
    $|N(u) \cap N(u')| \leq \frac{d}{\log^{40} d}$.
    Therefore, there are at most 
    $\frac{d^2}{\log^{40} d}$
    edges $xy$ for which $x \in N(u) \cap N(u')$.
    By repeating the argument for all pairs in $\{u,v\} \times \{u',v'\}$, there are at most $\frac{4d^2}{\log^{40} d}$ 
    edges $xy$
    with an endpoint that is adjacent to an endpoint of both $e$ and $e'$.
    Finally, we count the edges $xy$ with  $x \in N(u)$ and $y \in N(u')$.
    As we already counted edges with an endpoint that is adjacent to an endpoint of both $e$ and $e'$,
    may assume that $x \in N(u) \setminus N(u')$ and $y \in N(u') \setminus N(u)$.
    We note that $A \coloneqq N(u) \setminus N(u')$ and $B \coloneqq N(u') \setminus N(u)$
    are disjoint and $G[A, B]$ does not contain a copy of $K_{t,t}$.
    Therefore, by applying Lemma~\ref{lem:KST} with $m, n \leq d$ and 
    with our value $t = o \left (\frac{\log d}{\log\log d} \right )$
    representing both $s$ and $t$,
    the number of such edges
    $xy$ is at most
    $$
    (t-1)^\frac{1}{t}d^{2-\frac{1}{t}} + (t-1)d = (2+o(1))d^{2-\frac{1}{t}}\leq \frac{(2+o(1))d^2}{\log^{40}d} < \frac{d^2}{\log^{30} d}.
    $$
    By repeating this argument for all pairs in $\{u,v\} \times \{u',v'\}$,
    we find at most
    $\frac{4d^2}{\log^{30}d}$
    edges of this last type.
    Putting all of these cases together,
    \[|N(e) \cap N(e')| < \frac{4d^2}{\log^{40} d} +  \frac{d^2}{ \log^{25} d} < \frac{\Delta}{\log^{20} \Delta},\]
    as desired.
\end{proof}

The following lemma will play a key role in our proof of Condition~\eqref{item:main-3a} of Theorem~\ref{thm:general} with our choice of the sets $B^X$.

\begin{lemma}
\label{lem:B1}
    Let $v \in V(G)$, and 
    let
    $vw \in E_1(v)$.
    If $v$ and $w$ are not friends, then 
    $vw$ is adjacent to at least  $\frac 12 \Delta - \frac{\Delta}{\log^{10} \Delta} $ edges in $E_3(v)$.
\end{lemma}
\begin{proof}
    Let $N^*=N(w)\setminus N[v]$, and note that $N^* \subseteq N_2(v)$. Since $v$ and $w$ are not friends, $|N^*|\geq d \left (1 - \frac{1}{\log^{40}d} \right )$. 
    By Lemma~\ref{lem:few-friends}, there are at least $|N^*|-4\log^{40t}d$ vertices $v' \in N^*$ that are not friends of $v$.
    Each such vertex $v'$ belongs to $N_2(v)$ and has 
    at least $d \left (1-\frac{1}{\log^{40}d}\right )$ neighbors $w'\not \in N[v]$.
    Each such pair $v'w'$ is an edge in $E_3(v)$.
    Therefore, the number of desired edges is at least
    \[\left(d\left(1 - \frac{1}{\log^{40}d}\right) -4\log^{40t}d \right)\cdot d\left(1 - \frac{1}{\log^{40}d}\right)
    >
    d^2 \left(1 - \frac{1}{\log^{30}d}\right)^2 > \frac 12\Delta - \frac{\Delta}{\log^{10} \Delta},\]
    as desired.
\end{proof}

The following lemma will play a key role in our proof of Condition~\eqref{item:main-3b} of Theorem~\ref{thm:general} with our choice of the sets $B^X$.

\begin{lemma}
\label{lem:Q1-incident}
    If $e=vw$ and $e' = vw'$ are two edges in $G$ for which $w$ and $w'$ are not friends, then $|N(e) \cap N(e') \cap E_3(v)| \leq \frac{\Delta}{\log^{20} \Delta} $.
\end{lemma}
\begin{proof}
Fix $xy \in N(e) \cap N(e') \cap E_3(v)$. Since $xy \in E_3(v)$, we have that $\{x, y\} \cap \{w, w', v\} = \emptyset$. Additionally, since $xy \in N(e) \cap N(e')$, we have that $\{x, y\} \cap N(w) \neq \emptyset$ and $\{x, y\} \cap N(w') \neq \emptyset$. There are two cases to consider: $\{x, y\} \cap N(w) \cap N(w') \neq \emptyset$ and $\{x, y\} \cap N(w) \cap N(w') = \emptyset$. 

If $\{x, y\} \cap N(w) \cap N(w') \neq \emptyset$, then either $x$ or $y$ is a common neighbor of $w$ and $w'$. Since $w$ and $w'$ are strangers, $|N(w) \cap N(w')| \leq \frac{d}{\log^{40} d}$. Then, we have 
\begin{align*}
    |\{xy \colon \{x, y\} \cap N(w) \cap N(w') \neq \emptyset\}| \leq \sum_{z \in N(w) \cap N(w')} \deg_G(z) \leq \frac{d^2}{\log^{40}d}.
\end{align*}

If $\{x, y\} \cap N(w) \cap N(w') = \emptyset$, then without loss of generality, we have that $x \in N(w) \setminus N(w')$ and $y \in N(w') \setminus N(w)$. We have that $|N(w) \setminus N(w')| \leq d$ and $|N(w') \setminus N(w)| \leq d$, so applying Lemma~\ref{lem:KST} with a forbidden $K_{t, t}$ along with~\eqref{eqn:t-bounds}, the number of edges with one endpoint in $N(w) \setminus N(w')$ and one endpoint in $N(w') \setminus N(w)$ is less than 
\begin{align*}
    (t - 1)^{1/t} d^{2 - 1/t} + (t - 1)d &\leq (2 + o(1)) d^{2 - 1/t} < \frac{ d^2}{\log^{40} d}.
\end{align*}
Therefore,
\[
    |N(e) \cap N(e') \cap E_3(v)| \leq \frac{d^2}{\log^{40}d} + \frac{ d^2}{\log^{40} d} < \frac{ \Delta}{ \log^{20}\Delta },
\]
as claimed.
\end{proof}

With these lemmas in hand, we  now show that $L(G)^2$ satisfies the conditions of 
Theorem~\ref{thm:general} with $\gamma = \frac 12$, $N = 1$, and our value $\epsilon >0$.
As $\Delta = 2d^2$, $L(G)^2$ has maximum degree at most $\Delta$.
We verify each condition separately:
\begin{enumerate}
    \item By Lemma~\ref{lem:Q1-non-incident}, $|N_{L(G)^2}(e) \cap N_{L(G)^2}(e')|< \frac{\Delta}{\log^{20} \Delta}$ for each stranger pair $e,e' \in E(G)$, so Condition~\eqref{item:main-1} is satisfied.

    \item Next, 
    we check Condition~\eqref{item:main-2}.
    We define a family $\mathcal X$ as follows.
    First, for each vertex $v \in V(G)$, we define a family $\mathcal Y_v$.
    For each $w \in N(v)$,
    if $v$ and $w$ are friends, then add a set $X_{vw} = \{vw\}$ to $\mathcal Y_v$.
    Next, let $N'(v) \subseteq N(v)$ be the set of all neighbors of $v$ which are not friends with $v$.
    Using Lemma~\ref{lem:few-friends}, 
    we partition $N'(v)$ 
    into subsets $S_{v,1}, \dots, S_{v,r}$
    of mutual strangers
    for some $r < 5 \log^{40 t} d$.
    For each $S_{v,i}$, we add a set $X_{v,i} = \{v w: w \in S_{v,i}\}$ to $\mathcal Y_v$.
    As Lemma~\ref{lem:few-friends} implies that $v$ has at most $4 \log^{40 t}d$ friends,
    there are at most $4 \log^{40 t} d$ sets $X_{vw}$ in $\mathcal Y_v$, and there are at at most $5 \log^{40 t} d$ sets $S_{v,i}$ in $\mathcal Y_v$.
    Altogether, $\mathcal Y_v$ has at most $9 \log^{40t} d $ sets.
    We let $\mathcal X = \bigcup_{v \in V(G)} \mathcal Y_v$.
    (For technical reasons, if a set $X$ occurs multiple times in this union, we let $X$ appear with multiplicity in $\mathcal X$. 
    It is straightforward to see that if the conditions of Theorem~\ref{thm:general} hold for a multiset $\mathcal X$, then the conditions still hold after the multiplicity of each $X \in \mathcal X$ is set to $1$.)

    We now verify that Condition~\eqref{item:main-2} holds for our choice of $\mathcal X$.
    Since each set $X \in \mathcal X$ consists of edges incident to a common vertex $v \in V(G)$, $|X| \leq d < \Delta^{\gamma}$, so Condition~\eqref{item:main-2a} holds.
    Next, for each $e \in E(G)$,
    we write $F(e)$ for the set of vertices $w \in V(G)$ that are friends of an endpoint of $e$, and 
    we define $\mathcal X_e = \bigcup_{w \in F(e)} \mathcal Y_w$.
    Lemma~\ref{lem:few-friends} implies that
    $|F(e)| \leq 8 \log^{40t} d$.
    Furthermore, as observed earlier, we have $|\mathcal Y_w| \leq 9 \log^{40 t} d$ for each $w \in F(e)$.
    Therefore,
    $|\mathcal X_e | \leq 72 \log^{80 t} d < \Delta^{(1-\gamma) \epsilon/12}$. As all friends of $e$ belong to a set in $\mathcal X_e$, Condition~\eqref{item:main-2b} holds.
    Finally, for each $e \in E(G)$ with endpoints $u$ and $v$, $e$ only belongs to sets in $\mathcal Y_u \cup \mathcal Y_v$. Therefore, 
    as $d$ is sufficiently large,
    $e$ belongs to at most $18\log^{40 t} d < \Delta$ sets $X \in \mathcal X$.
    Thus, Condition~\eqref{item:main-2c} holds.

    \item Finally, 
    we check that Condition~\eqref{item:main-3} holds.
    Consider a set $X \in \mathcal X$ for which $|X| \geq 2$, and let $v \in V(G)$ be the unique vertex for which $X \in \mathcal Y_v$. (This is precisely where we must allow $\mathcal X$ to be a multiset, as the families $\mathcal Y_v$ may not be disjoint.)
    We let $B^X = E_3(v)$.
    Note that as $|X| \geq 2$, it must be the case that $w$ and $w'$ are strangers for any distinct $vw, vw' \in X$.
    Therefore, by Lemma~\ref{lem:B1}, 
    each edge $e \in X$ has at least 
    \[d^2 \left (1- \frac{1}{\log^{400}d} \right )^2 > \Delta \left ( \frac 12 - \log^{-10} \Delta \right ) \geq d_{L(G)^2}(e) - \left (1 - \gamma + \log^{-10} \Delta \right ) \Delta\]
    neighbors in $B^X$, so Condition~\eqref{item:main-3a} holds.
    Furthermore, by Lemma~\ref{lem:Q1-incident},
    for any two edges $e,e' \in X$, we have
    \[|N(e) \cap N(e') \cap B^X| < \frac{\Delta}{\log^{20} \Delta},\]
    so Condition~\eqref{item:main-3b} holds.

\end{enumerate}
Therefore, by Theorem~\ref{thm:general}, $\chi_{\ell}(L(G)^2) \leq (1 + \epsilon) \frac{\Delta}{\log \Delta}$.
As a result, we have
\[\chi'_s(G) \leq (1 + \epsilon) \frac{2d^2}{\log (2d^2)} < (1 + \epsilon) \frac{d^2}{\log d}.\]
completing the proof of Theorem~\ref{thm:main}.

\subsection{Proof of Theorem~\ref{thm: girth}}

Let $\epsilon > 0$ and $t \geq 2$ be fixed.
Let $G$ be a graph of maximum degree $d$ with girth at least $2t+1$
(we reserve the symbol $\Delta$ as an upper bound for the maximum degree of $L(G)^t$).
We assume without loss of generality that $G$ is $d$-regular.
We aim to show that $G$ admits a distance-$t$ edge coloring with $(1 + \epsilon) \frac{d^2}{\log d}$ colors, provided that $d$ is sufficiently large in terms of $\epsilon$ and $t$.
For this, we will apply Theorem~\ref{thm:general} 
to $L(G)^t$
with our value of $\epsilon$, and with $\gamma = \frac 12$ and $\Delta = 2d^t$.
To this end, we say that two edges $e,e' \in E(G)$ are \emph{friends} if they have at least $\frac{\Delta}{\log^{20}\Delta}$ common neighbors in $L(G)^t$.

\begin{lemma}
\label{claim:t-1 girth}
    For every distinct pair $u,v \in V(G)$,
    $|N_{t-1}(u) \cap N_{t-1}(v)| < td^{t-2}$.
\end{lemma}
\begin{proof}
    We follow the approach of Kaiser and Kang \cite{kaiser2014distance}.
    Let $T$ be the tree induced by vertices $w$ for which a path of length at most $t-1$ joins $w$ to $u$ in $G-v$. 
    The girth condition on $G$ guarantees that $T$ is a tree.
    Note that any two vertices in $T$ have a distance of at most $2t-2$ in $T$. Therefore, each vertex of $V(G) \setminus V(T)$ has at most one neighbor in $T$.

    Now, we count the number of paths of length at most $t-1$ that begin at $v$ and end in $V(T)$. The number of paths of length at most $t-2$ is at most $1+(d-1)+(d-1)^2 + \cdots + (d-1)^{t-2} < d^{t-2}$.
    Now, consider a path $P$ of length $t-1$ that begins at $v$ and ends in $V(T)$. 
    Write $V(P) = (v,v_1,v_2,\dots, v_{t-1})$.
    Since $v \not \in V(T)$, one of the vertices $v_i$ is the first in $V(T)$. Since each vertex in $V(G) \setminus V(T)$ has at most one neighbor in $V(T)$, the number of  paths $P$ for which $v_i$ is the first in $V(T)$ is at most $d^{t-2}$. Taking a union, the number of paths of length at most $t-1$ beginning at $v$ and ending in $V(T)$ is at most $(t-1)d^{t-2} + d^{t-2} = td^{t-2}$.
\end{proof}

The following lemma will play a key role in our proof of Condition~\eqref{item:main-3a} of Theorem~\ref{thm:general} with our choice of the sets $B^X$.

\begin{lemma}
\label{claim:B-girth}
Let $v,w \in V(G)$ be neighbors. Then, at least $d|N_{t-1}(w)| - d^t \log^{-12} (2d^t)$
edges
of $E_{t+1}(v)$ are adjacent to $vw$ in $L(G)^t$.
\end{lemma}
\begin{proof}
    We write $S = N_{t-1}(w) \setminus N_{t-1}(v)$.
    By Lemma~\ref{claim:t-1 girth}, we have
    \[|S| > |N_{t-1}(w)| - td^{t-2}.\] 
    For each $x \in S$, if an edge $xy$ is incident to $x$ and $y \not \in N_{t-1}(v)$,
    then $xy \in E_{t+1}(v) \cap N_{L(G)^t}(vw)$.
    Thus, we aim to count the edges $xy$ with $x \in S$ and $y \not \in N_{t-1}(v)$.
    To this end, we let $T = N_{t-1}(v)$, and we estimate $e(S,T)$.
    By Lemma~\ref{BondySimonovits}, $e(S,T) = O(t(|S|+|T|)^{1+\frac 1t}) = O(d^{t-\frac 1t})$. 
    Therefore, 
    \[
        |E_{t+1}(v) \cap N_{L(G)^t}(vw)| > d \left (|N_{t-1}(w)| - t d^{t-2}\right ) - O( d^{t-\frac 1t} ) > d|N_{t-1}(w)| - d^t \log^{-12} (2d^t),
    \]
    as desired.
\end{proof}

\begin{lemma}
\label{claim:A123-girth}
    Let $uv, u'v' \in E(G)$ and let $A \subseteq N_{L(G)^t}(uv) \cap N_{L(G)^t}(u'v')$ be arbitrary.
    Define 
    \[
        A_1 \coloneqq \bigcup_{(w,w') \in \{u,v\} \times \{u',v'\} } \left \{ xy \in A: x \in N_{t-1}(w) \cap N_{t-1} (w') \right \}.
    \]
    Then,
    $|A| < |A_1| + \frac{d^t}{\log^{24}d }$.
\end{lemma}
\begin{proof}
    We define two additional edge sets $A_2, A_3 \subseteq A$ so that $A = A_1 \cup A_2 \cup A_3$. Let
    \[
        A_2 \coloneqq \bigcup_{ \substack{w \in \{u,v,u',v'\} \\ 1 \leq i \leq t-2}} \{xy \in A: x \in N_i(w) \}
    \]
    denote the set of edges in $A$ with an endpoint within distance $t-2$ of one of $u,v,u',v'$. We define 
    \[
    A_3 \coloneqq \bigcup_{(w,w') \in \{u,v\} \times \{u',v'\} } \{xy \in A: x \in N_{t-1}(w) \text{ and } y \in N_{t-1}(w') \}.
    \]
    It is a simple observation that $A = A_1 \cup A_2 \cup A_3$.

    As $G$ is $d$-regular, the number of vertices in $\bigcup_{1 \leq i \leq t-2} N_i(w)$ for a given vertex $w \in V(G)$ is at most $d^{t-2}$. Therefore, the number of edges with an endpoint in $\bigcup_{1 \leq i \leq t-2} N_i(w)$ is at most $d^{t-1}$, implying $|A_2| \leq 4d^{t-1}$.

    To
    estimate $|A_3|$,
    observe that 
    every edge in $A_3$ has both endpoints in 
    $ N_{t-1}(w) \cup N_{t-1}(w')$.
    As $| N_{t-1}(w) \cup N_{t-1}(w')|<2d^{t-1}$, 
    Lemma~\ref{BondySimonovits} implies that $|A_3|= O(td^{t-1/t})$.
    Therefore, $|A| < |A_1| + |A_2| + |A_3| < |A_1| + \frac{d^t}{\log^{24} d}$ for $d$ sufficiently large.
\end{proof}

\begin{lemma}
    If $e$ and $e'$ are not incident, then they are strangers.
\end{lemma}
\begin{proof}
    Suppose that $e$ and $e'$ are friends. Then, 
    \[|N_{L(G)^t} (e) \cap N_{L(G)^t} (e') | > \frac{\Delta}{\log^{20} \Delta} > \frac{d^t}{\log^{21} d}.\]
    Letting $A = N_{L(G)^t} (e) \cap N_{L(G)^t} (e')$, the above 
    implies that  for the set $A_1$ defined in Lemma~\ref{claim:A123-girth} with respect to $A$,
    $|A_1| > \frac{d^t}{2 \log^{21 }d } $. 
    However, $|A_1| = O(td^{t-1})$ by Lemma~\ref{claim:t-1 girth}, a contradiction.
\end{proof}

The following lemma will play a key role in our proof of Condition~\eqref{item:main-3b} of Theorem~\ref{thm:general} with our choice of the sets $B^X$.

\begin{lemma}
    Let $v \in V(G)$, and let $w,w'\in N(v)$.
    Then, \[|N_{L(G)^t}(vw) \cap N_{L(G)^t}(vw') \cap E_{t+1}(v)| < \frac{2d^t}{\log^{24} d}.\]
\end{lemma}
\begin{proof}
For $A = N_{L(G)^t}(vw) \cap N_{L(G)^t}(vw') \cap E_{t+1}(v)$, let
\[
A_1 \coloneqq  \left \{ xy \in A: x \in N_{t-1}(w) \cap N_{t-1} (w') \right \}.
\]
By Lemma~\ref{claim:t-1 girth}, $|A_1| < td^{t-1}$.
Furthermore, by the definition of distance, 
no edge of $E_{t+1}(v)$ has an endpoint in $N_{t-1}(v)$.
Therefore, as $A \subseteq E_{t+1}(v)$,
the set $A_1$ defined here is the same as the set $A_1$ defined in Lemma~\ref{claim:A123-girth}, and hence 
\[|A| < |A_1| + \frac{d^t}{\log^{24} d} < \frac{2d^t}{\log^{24} d},\]
as desired.
\end{proof}

With these lemmas in hand, we  now show that $L(G)^t$ satisfies the conditions of 
Theorem~\ref{thm:general} with $\gamma = \frac 12$, $N = 1$, and our value $\epsilon >0$.
As $\Delta = 2d^2$, $L(G)^2$ has maximum degree at most $\Delta$.
We verify each condition separately:
\begin{enumerate}
    \item We have defined a symmetric friend relation $\mathcal F \subseteq E(G) \times E(G)$
    so that two edges are strangers if and only if they have at most $\Delta/\log^{20} \Delta$ common neighbors in $L(G)^t$, so Condition~\eqref{item:main-1} holds.

    \item Next, 
    we check Condition~\eqref{item:main-2}.
    We define a family $\mathcal X$ as follows.
    For each $v \in V(G)$,
    we define a set $X_v$ consisting of all edges incident to $v$. Observe that $|X_v| = d < \Delta^{\gamma}$, so Condition~\eqref{item:main-2a} holds.
    Next, for each $e \in E(G)$ with endpoints $u$ and $v$,
    all friends of $e$ belong to the size-$2$ family $\{X_u,X_v\}$;
    therefore,
    Condition~\eqref{item:main-2b} holds.
    Finally, for each $e \in E(G)$ with endpoints $u$ and $v$, $e$ does not belong to any set in $\mathcal X$ except for $X_u$ and $X_v$;
    therefore, Condition~\eqref{item:main-2c} holds as well.

    \item Finally, we check that Condition~\eqref{item:main-3} holds.
    For each set $X_v \in \mathcal X$,
    let $B^{X_v} = E_{t+1}(v)$.
    By Lemma~\ref{claim:B-girth},
    each edge $e \in X_v$ of the form $uw$ has at least $d|N_{t-1}(w)| - d^t \log^{-12}(2d^t)$ 
    neighbors in $B^X$.
    As $d_{L(G)^t}(e) < d|N_{t-1}(w)| + d^t$, we have 
    \begin{align*}
        d_{L(G)^t}(e) -  \Delta (1 - \gamma + \log^{-10} \Delta) &< d|N_{t-1}(w)|
    - \frac{\Delta}{\log^{10} \Delta} \\
    &< d|N_{t-1}(w)| - d^t \log^{-12}(2d^t),
    \end{align*}
    implying Condition~\eqref{item:main-3a} holds.
    Furthermore, for any pair $w,w' \in N(v)$,
    Lemma~\ref{claim:B-girth} implies that $vw$ and $vw'$ have at most $\frac{\Delta}{\log^{24} d} < \frac{\Delta}{\log^{20} \Delta} $ common neighbors in $B^X$. 
    Therefore Condition~\eqref{item:main-3b} holds.
\end{enumerate}
Therefore, by Theorem~\ref{thm:general}, $\chi_{\ell}(L(G)^t) \leq (1 + \epsilon) \frac{\Delta}{\log \Delta}$.
As a result, $G$ admits a proper distance-$t$ edge coloring with at most
\[(1 + \epsilon) \frac{2d^t}{\log (2d^t)} < (1 + \epsilon) \frac{2d^t}{t\log d}\]
colors, completing the proof of Theorem~\ref{thm:main}.

\subsection{Proof of Theorem~\ref{thm: spectral}}

We consider a $d$-regular $n$-vertex graph with $n \geq d^{2t}$ and with second eigenvalue $\lambda \leq C\sqrt d$, where $C \geq 2$ is some constant.
We  aim to show that the graph $L(G)^t$
has a proper coloring with $(1 + o(1)) \frac{2d^t}{t \log d}$ colors using Theorem~\ref{thm:general}.
Similar to the proof of Theorem~\ref{thm: girth},
we let $\Delta = 2d^t$, and we observe that the maximum degree of $L(G)^t$ is at most $\Delta$.
We aim to apply Theorem~\ref{thm:general} 
to $L(G)^t$
with an arbitrary value $\epsilon > 0$,
$\gamma = \frac 12$, and 
$\Delta = 2d^t$.

Before considering $L(G)^t$, we first observe some properties of $G$.
As $G$ is a spectral expander, many important properties of $G$ can be derived from the expander mixing lemma:
\begin{lemma}[{\cite[Lemma 2.5]{Hoory}}]
\label{lem:expander-mixing}
 Let $S,T \subseteq V(G)$, let $e(S,T)$ be the number of ordered pairs $(u,v)$ for which $u \in S$, $v \in T$, and $uv \in E(G)$.
 Then, $ \left | e(S,T) - \frac dn |S| |T| \right | \leq \sqrt{\lambda |S| |T|}$.
\end{lemma}

We often use the following weaker corollary of the expander mixing lemma.

\begin{lemma}
\label{lem:bdd-cut}
Let $S, T \subseteq V(G)$, and suppose that $|S||T| \leq d^{4t-2}$.
    Then, $e(S,T) < 2 \lambda\sqrt{|S||T|}$.
\end{lemma}
\begin{proof}
    By the expander mixing lemma (Lemma~\ref{lem:expander-mixing}),
    \[
 e(S,T) \leq \sqrt{|S| |T|}  \left ( \frac dn  \sqrt{|S| |T|}+ \lambda \right ).
    \]
    Since $n \geq d^{2t}$ and  $|S||T| \leq d^{4t-2}$, we have $\frac dn \sqrt{|S||T|} \leq 1 < \lambda$. Therefore, $e(S,T) < 2\lambda \sqrt{|S||T|}$.
\end{proof}

We say that two vertices in $V(G)$ are friends if their codegree in $G^t$ is at least $\frac{d^t}{\log^{25} d}$.
The following lemma shows that each $v \in V(G)$ has few friends.

\begin{lemma}
\label{lem:few-friends-dist-t}
    Each $v \in V(G)$ has at most $4C^{2t} \log^{60} (d^t)$ friends.
\end{lemma}
\begin{proof}
Consider a friend $w$ of $v$. Since at most $2d^{t-1}$ vertices in $G$ have a distance of at most $t-1$ from $v$ or from $w$,
it follows that at least $\frac{d^t}{\log^{25} d} - 2d^{t-1} > \frac{d^t}{\log^{30} d} \eqqcolon h$ 
vertices of $G$ have a distance of exactly $t$ from both $v$ and $w$.
Now, let $A$ be the adjacency matrix of $G$.
Let $G(A^t)$
be the multigraph (possibly with loops and parallel edges) whose adjacency matrix is equal to $A^t$.
In other words, 
the number of edges joining
two vertices $u,u' \in V(G(A^t)) = V(G)$
is equal to the number of walks
in $G$ of length exactly $t$ from $u$ to $u'$.
We write $N$ for the neighborhood of $v$ in the graph $G(A^t)$,
and we note that $|N| \leq d^t$.
We observe that for each friend $w$ of $v$, $w$ has at least $h$ neighbors in $N$
in the graph $G(A^t)$.

Note that $G(A^t)$ is $d^t$-regular and has a spectrum equal to the spectrum of $A$ raised to the power of $t$.
Furthermore, by the Perron-Frobenius theorem,
if $G$ is not bipartite or if $t$ is odd, then the largest eigenvalue of $G(A^t)$ is $d^t$, and the second largest eigenvalue of $G(A^t)$ is $\lambda^t \leq C^t d^{t/2}$.
If $G$ is bipartite and $t$ is even, then $G^t$ contains two components, each with a largest eigenvalue of $d^t$ and a second largest eigenvalue of $\lambda^t \leq C^t d^{t/2}$.

Now,  let $F \subseteq V(G)$
be the set of vertices that are friends with
$v$.
Since each $w \in F$ has at least $h$ neighbors in $F$ via $G(A^t)$,
the number of edges of $G(A^t)$ with one endpoint in  $F$ and the other $N$, counted with multiplicity, 
is at least $h |F|$.
By the expander mixing lemma,
\[
h |F| \leq e(F,N) \leq \frac{d^t |N| |F| }{n} + \lambda^t \sqrt{|N| |F|} \leq \frac{d^{2t} |F| }{n} + C^t d^{t} \sqrt{ |F|} .
\]
As each $w \in F$ has distance at most $2$ from $v$ in $G(A^t)$,
$\sqrt{|F|} \leq d^t$, so that 
\[
h |F| \leq \frac{ d^{3t} \sqrt{|F|} }{n} + C^t d^{t} \sqrt{|F|}  < 2C^t d^{t} \sqrt{|F|},
\]
where we use the fact that $n \geq d^{2t}$.
Rearranging, we have
\[
|F| \leq \frac{4C^{2t} d^{2t}}{h^2} = 4C^{2t} \log^{60} d,
\]
as desired.
\end{proof}

\begin{lemma}
\label{lem:t-1 intersection}
    If $v,w \in V(G)$ are strangers, then $|N_{t-1}(v) \cap N_{t-1}(w)| < \frac{d^{t-1}}{\log^{24} d}$.
\end{lemma}
\begin{proof}
    Write $S = N_{t-1}(v) \cap N_{t-1}(w)$.
    Let $T = N_G(S)$, so that $e(S,T) = d |S|$.
    By Lemma~\ref{lem:bdd-cut},
    \[
    d|S| = e(S,T)  < 2 \lambda \sqrt{|S||T|}.
    \]
    Rearranging, we have
    \[
    |S| < \frac{4 \lambda^2 |T|}{d^2} = \frac{4C^2 |T|}{d}.
    \]
    Since $T \subseteq N_{G^t}(v) \cap N_{G^t}(w)$ 
    and $v$ and $w$ are strangers, $|T| \leq \frac{d^t}{\log^{25} d}$.
    Therefore,
    \[
    |S| < \frac{4C^2 d^{t-1} }{\log^{25} d} < \frac{d^{t-1}}{\log^{24} d},
    \]
    for $d$ sufficiently large.
\end{proof}

\begin{lemma}
    For each $v \in V(G)$ and $i \in \{1, \dots, t\}$, $|N_{i}(v)| \geq \frac{d^i}{(5C^2)^{i-1}}$.
\end{lemma}
\begin{proof}
    We use induction to prove the stronger statement
    $|N_i(v)| \geq \frac{d}{5C^2} |N_{i-1}(v)|$ for $i \in [t]$, where $N_0(v) = \{v\}$.
    The base case $i = 1$ is trivial as $G$ is $d$-regular.
    For $2 \leq i \leq t$,
    letting $S = N_{i-1}(v)$ and $T = \{v\} \cup N_1(v) \cup \dots \cup N_i(v)$,
    we have $e(S,T) = |S| d$.
    Therefore, by Lemma~\ref{lem:bdd-cut}, we have
    \[
        |S| d = e(S,T)  < 2  \lambda  \sqrt{  |S| |T|} \implies |T| \geq \frac{|S| d }{4C^2}.
    \]
    By the induction hypothesis,
    $|N_j(v)| > 2|N_{j-1}(v)|$ for $2 \leq j \leq i-1$ and $d$ sufficiently large,
    implying
    \[|T| < |N_i(v)| + 2|N_{i-1}(v)| = |N_i(v)| + 2|S|.\] 
    Therefore,
    \[
    |N_i(v)| > |T| - 2|S| \geq \left ( \frac{d}{4C^2} - 2 \right ) |S| > \frac{d}{5C^2} |S|,
    \]
    as desired.
\end{proof}

With the above results in hand, let us turn our attention to the main goal of this section: applying Theorem~\ref{thm:general} to $L(G)^t$ in order to prove Theorem~\ref{thm: spectral}.
First, we estimate the degree of an edge $vw \in E(G)$ in the graph $L(G)^t$.
\begin{lemma}
\label{lem:spec-d-bd}
    For each $vw \in E(G)$, $d_{L(G)^t}(vw) <  d|N_{t-1}(v)| + d^t + O(d^{t-1})$.
\end{lemma}
\begin{proof}
Every edge in $N_{L(G)^t}(vw)$ is incident either to $N_i(v)$ or $N_i(w)$ for some $1 \leq i \leq t-1$.
    As $\left |\bigcup_{i=1}^{t-2} N_i(v) \cup N_i(w) \right | = O(d^{t-2})$, the number of edges incident to $N_i(v)$ or $N_i(w)$ for $i \leq t-2$ is $O(d^{t-1})$.
    Therefore, 
    \[d(vw) \leq d |N_{t-1}(v)| + d|N_{t-1}(w)| + O(d^{t-1}) < d|N_{t-1}(v)| + d^t + O(d^{t-1}),\]
    completing the proof.
\end{proof}

The following lemma will play a key role in our proof of Condition~\eqref{item:main-3a} of Theorem~\ref{thm:general} with our choice of the sets $B^X$.

\begin{lemma}
\label{lem:spec-N-BX}
Let $vw \in E(G)$. If $v$ and $w$ are strangers,
then at least $d|N_{t-1}(w)| - d^t \log^{-12} (2d^t)$
edges
of $E_{t+1}(v)$ are adjacent to $vw$ in $L(G)^t$.
\end{lemma}
\begin{proof}
    We write $S = N_{t-1}(w) \setminus N_{t-1}(v)$.
    By Lemma~\ref{lem:t-1 intersection},
    $|S| > |N_{t-1}(w)| - \frac{d^{t-1}}{\log^{24} d}$. 
    For each $x \in S$, if an edge $xy$ satisfies $y \not \in N_{t-1}(v)$,
    then $xy \in E_{t+1}(v) \cap N_{L(G)^t}(vw)$.
    Next, we bound the number of edges $xy$ with $x \in S$ and $y \not \in N_{t-1}(v)$.
    To this end, we let $T = N_{t-1}(v)$, and apply Lemma~\ref{lem:bdd-cut} to obtain
    \[
    e(S,T) < 2 \lambda \sqrt{|S| |T|} \leq 2C \sqrt{d}  \cdot d^{t-1} < d^{t - \frac 13}.
    \]
    From here, it follows that
    \[
|E_{t+1}(v) \cap N_{L(G)^t}(vw)| > d \left (|N_{t-1}(w)| - \frac{d^{t-1}}{\log^{24} d}\right ) - d^{t-\frac 13}  > d|N_{t-1}(w)| - \frac{d^t}{\log^{12} (2d^t)},
    \]
    as desired.
\end{proof}

\begin{lemma}
\label{lem:spec123}
    Let $uv, u'v' \in E(G)$ and let $A \subseteq N_{L(G)^t}(uv) \cap N_{L(G)^t}(u'v')$ be arbitrary.
    Define 
    \[
        A_1 \coloneqq \bigcup_{(w,w') \in \{u,v\} \times \{u',v'\} } \left \{ xy \in A: x \in N_{t-1}(w) \cap N_{t-1} (w') \right \}.
    \]
    Then,
    $|A| < |A_1| + \frac{d^t}{\log^{24}d }$.
\end{lemma}
\begin{proof}
    We define two additional edge sets $A_2, A_3 \subseteq A$ so that $A = A_1 \cup A_2 \cup A_3$. Let
    \[
        A_2 \coloneqq \bigcup_{ \substack{w \in \{u,v,u',v'\} \\ 1 \leq i \leq t-2}} \{xy \in A: x \in N_i(w) \}
    \]
    denote the set of edges in $A$ with an endpoint within distance $t-2$ of one of $u,v,u',v'$. We define 
    \[
    A_3 \coloneqq \bigcup_{(w,w') \in \{u,v\} \times \{u',v'\} } \{xy \in A: x \in N_{t-1}(w) \text{ and } y \in N_{t-1}(w') \}.
    \]
    It is a simple observation that $A = A_1 \cup A_2 \cup A_3$.

    As $G$ is $d$-regular, the number of vertices in $\bigcup_{1 \leq i \leq t-2} N_i(w)$ for a given vertex $w \in V(G)$ is at most $d^{t-2}$. Therefore, the number of edges with an endpoint in $\bigcup_{1 \leq i \leq t-2} N_i(w)$ is at most $d^{t-1}$, implying $|A_2| \leq 4d^{t-1}$.

    To
    estimate $|A_3|$, we let $S = N_{t-1}(w)$, $T = N_{t-1}(w')$, 
    and observe that $|S|, |T| < d^{t-1}$.
    By Lemma~\ref{lem:bdd-cut}, we have
    \[|A_3| \leq e(S,T) \leq 2 \lambda \sqrt{ |S| |T|} < 2 C \sqrt d\cdot  d^{t-1} < d^{t-1/3}.\]
    Therefore, $|A| \leq |A_1| + |A_2| + |A_3| < |A_1| + \frac{d^t}{\log^{24} d}$ for $d$ sufficiently large.
\end{proof}

The following lemma will play a key role in our proof of Condition~\eqref{item:main-3b} of Theorem~\ref{thm:general} with our choice of the sets $B^X$.

\begin{lemma}
    \label{lem:spec-eeBX}
    Let $v \in V(G)$, and let $w,w'\in N(v)$ be strangers.
    Then, \[|N_{L(G)^t}(vw) \cap N_{L(G)^t}(vw') \cap E_{t+1}(v)| < \frac{2d^t}{\log^{24} d}.\]
\end{lemma}
\begin{proof}
For $A = N_{L(G)^t}(vw) \cap N_{L(G)^t}(vw') \cap E_{t+1}(v)$, let
\[
A_1 =  \left \{ xy \in A: x \in N_{t-1}(w) \cap N_{t-1} (w') \right \}.
\]
As $w$ and $w'$ are strangers, Lemma~\ref{lem:t-1 intersection} implies that $|A_1| < \frac{d^t}{\log^{24} d}$.
Furthermore, by the definition of distance, 
no edge of $E_{t+1}(v)$ has an endpoint in $N_{t-1}(v)$.
Therefore, as $A \subseteq E_{t+1}(v)$,
the set $A_1$ defined here is the same as the set $A_1$ defined in Lemma~\ref{lem:spec123}, and hence 
\[|A| < |A_1| + \frac{d^t}{\log^{24} d} < \frac{2d^t}{\log^{24} d},\]
as desired.
\end{proof}

With these lemmas in hand, we check the conditions of Theorem~\ref{thm:main} for the graph $L(G)^t$ with $N=1$, $\Delta = 2d^t$, and $\gamma = \frac 12$.

\begin{enumerate}
    \item Recall that two vertices of $G$ are friends if they have at least $\frac{d^t}{\log^{25} d}$ 
common neighbors in $G^t$.
We define a symmetric friend relation $\mathcal F \subseteq E(G) \times E(G)$ so that two edges $e,e'$ are friends if and only if
an endpoint $v$ of $e$ is a friend in $G$ 
of an endpoint $v'$ of $e'$.
Now, consider two stranger edges $e = uv$ and $e' = u'v'$ in $G$.
By Lemma~\ref{lem:spec123}, the number of common neighbors of $e$ and $e'$ in $L(G)^t$ is less than $|A_1| + \frac{d^t}{\log^{24} d}$, where
\[
A_1 \coloneqq \bigcup_{(w,w') \in \{u,v\} \times \{u',v'\} } \left \{ xy \in A: x \in N_{t-1}(w) \cap N_{t-1} (w') \right \}.
\]
As $e$ and $e'$ are strangers, Lemma~\ref{lem:t-1 intersection} implies that $|A_1| < \frac{4d^{t}}{\log^{24} d}$.
Therefore, the number of common neighbors of $e$ and $e'$ in $L(G)^t$ is less than $\frac{5d^t}{\log^{24} d} < \frac{\Delta}{\log^{20} \Delta}$, and Condition~\eqref{item:main-1} holds.
\item Next, we check Condition~\eqref{item:main-2}.
We define a family $\mathcal X$
in a manner similar to the proof of Theorem~\ref{thm:main}.
For each $v \in V(G)$, we define a family $\mathcal Y_v$. For each friend $w \in N(v)$, we add a set $X_{vw} = \{vw\}$ to $\mathcal Y_v$.
Next, we use Lemma~\ref{lem:few-friends-dist-t} to partition the set $N'(v) \subseteq N(v)$ of stranger neighbors of $v$ into subsets $S_{v,1}, \dots, S_{v,r}$ of mutual strangers for some value $r < 5C^{2t} \log^{60} (d^t)$, and we add the sets $\{vw\,:\, w \in S_{v, i}\}_{i = 1}^r$ to $\mathcal Y_v$.
As Lemma~\ref{lem:few-friends-dist-t} implies that $v$ has at most $4C^{2t} \log^{60} (d^t)$ friends, $|\mathcal Y_v| < 9C^{2t} \log{60}(d^t)$.
We let $\mathcal X = \bigcup_{v \in V(G) }\mathcal Y_v$ (and as in the proof of Theorem~\ref{thm:main}, we consider $\mathcal X$ to be a multiset).

We now  verify that Condition~\eqref{item:main-2} holds for our choice of $\mathcal X$. Since each $X \in \mathcal X$ consists of edges  incident to a common vertex $v \in V(G)$, $|X| \leq d < \Delta^{\gamma}$, so Condition~\eqref{item:main-2a} holds.
Next, for each $e \in E(G)$, we write $F(e)$  for the set of vertices $w \in V(G)$  that are friends of an endpoint  of $e$, and we define  $\mathcal X_e = \bigcup_{w \in F(e)} \mathcal Y_w$. Lemma~\ref{lem:few-friends-dist-t}
implies that $|F(e)| \leq 8C^{2t} \log^{60} (d^t)$. Furthermore, as observed earlier, we have 
$|\mathcal Y_w| < 9C^{2t} \log{60}(d^t)$ for each $w \in F(e)$.
Therefore,
$|\mathcal X_e| \leq 72C^{4t} \log^{120}(d^t) < \Delta^{(1-\gamma) \epsilon / 12}$ for $d$ sufficiently large.
As all friends of $e$ belong to  a set in $\mathcal X_e$, Condition~\eqref{item:main-2b} holds.
Finally, for each $e \in E(G)$ with endpoints $u$ and $v$, $e$ only belongs to sets in $\mathcal Y_u \cup \mathcal Y_v$. Therefore, as $d$ is sufficiently large, $e$  belongs to at most $18C^{2t} \log^{60}(d^t)$ sets $X \in \mathcal X$. Thus, Condition~\eqref{item:main-2c} holds.
\item Finally, we check that Condition~\eqref{item:main-3} holds.
Consider a set $X \in \mathcal  X$ for which $|X| \geq 2$, and let $v \in V(G)$  be the unique vertex for which $X \in \mathcal Y_v$.
We let $B^X = E_{t+1}(v)$.
By Lemmas~\ref{lem:spec-d-bd}~and~\ref{lem:spec-N-BX},
each edge $vw \in X$ has at least
\[
d|N_{t-1}(w)| - d^t \log^{-12}(2d^t) > d_{L(G)^t}(vw) - (1 - \gamma + \log^{-10} \Delta) \Delta
\]
neighbors in $B^X$ via $L(G)^t$, so Condition~\eqref{item:main-3a} holds.
Furthermore,
 as $|X| \geq 2$, it must be the case  that $w$ and $w'$ are strangers for any distinct $vw,vw' \in X$. Therefore,
for any two edges $e,e' \in X$, Lemma~\ref{lem:spec-eeBX}
implies that 
\[
|N_{L(G)^t}(e) \cap N_{L(G)^t}(e') \cap B^X| < \frac{2d^t}{\log^{24} d} < \frac{\Delta}{\log^{20} \Delta},
\]
verifying Condition~\eqref{item:main-3b}.
\end{enumerate}

Therefore, by Theorem~\ref{thm:main}, $\chi_{\ell}(L(G)^t) \leq (1+\epsilon) \frac{\Delta}{\log \Delta}$. As a result, $G$ admits a proper distance-$t$ edge coloring with at most 
\[
(1+\epsilon) \frac{2d^t}{\log(2d^t)} < (1+\epsilon) \frac{2d^t}{t\log d}
\]
colors, completing the proof of Theorem~\ref{thm: spectral}.

\section{Proof of Theorem~\ref{thm:general}}\label{sec: general-proof}

We restate the result for the reader's convenience:

\begin{theorem*}
[{Restatement of Theorem~\ref{thm:general}}]
Let $\gamma, \epsilon \in (0,1)$ and $N \in \mathbb N$.
Then, there exists $\Delta_0 \in \mathbb N$ for which the following holds.
Suppose $G$ is a graph of maximum degree at most $\Delta \geq \Delta_0$
with a symmetric reflexive friend relation
$\mathcal F \subseteq V(G) \times V(G)$, satisfying the following:
\begin{enumerate}
    \item
    Any two strangers have at most $\frac{\Delta}{\log^{20} \Delta}$ common neighbors.
    \item There is a family $\mathcal X$ of subsets $X \subseteq V(G)$ satisfying the following:
    \begin{enumerate}
        \item Each $X \in \mathcal X$ satisfies $|X| \leq \Delta^{\gamma}$,
        \item For each $v \in V(G)$, there exists a family $\mathcal X_v \subseteq \mathcal X$ of size at most $\Delta^{(1 - \gamma )\epsilon/12}$
        such that all friends of $v$ belong to $\bigcup_{X \in \mathcal X_v} X$,
        \item Each $v \in V(G)$ belongs to at most $\Delta^{N}$ sets $X \in \mathcal X$.
    \end{enumerate}
    \item
    For each $X \in \mathcal X$ with at least two vertices, there is a set $B^X \subseteq V(G)$ satisfying the following:
    \begin{enumerate}
        \item Each $v \in X$ has at least $d(v) - \Delta(1 - \gamma + \log^{-10}\Delta)$
            neighbors in $B^X$,
        \item For each distinct friend pair $u,v \in X$, $u$ and $v$ have at most $\frac{\Delta}{\log^{20} \Delta}$ common neighbors in $B^X$. 
    \end{enumerate}
\end{enumerate}
Then, $\chi_{\ell} (G) \leq (1 + \epsilon) \frac{\Delta}{\log \Delta}$.
\end{theorem*}

We will prove Theorem~\ref{thm:general} by repeatedly applying Kim's variant of the Wasteful Coloring Procedure (see page~\pageref{pageref} for an informal description of the procedure).
Initially, each vertex $v \in V(G)$ has a list $L_1(v)$ of $k$ colors, where $k = (1+\epsilon) \frac{\Delta}{\log \Delta}$.
After each iteration, we update 
the lists $L_i(v)$
for each uncolored $v \in V(G)$, so that each color $c \in L_i(v)$ has the property that no $u \in N(v)$ is colored $c$.
We additionally track sets $T_i(v,c) \subseteq N(v)$ for each uncolored $v \in V(G)$ and $c \in L_i(v)$.
This set contains the uncolored neighbors $u \in N(v)$ for which $c \in L_i(u)$.
We let $t_i(v, c) = |T_i(v, c)|$.
We show that
after sufficiently many iterations, we reach a point at which there are integers $d, \ell \in \mathbb N$ satisfying $\ell \geq 8d$ such that $|L_i(v)| \geq \ell$ for each uncolored $v \in V(G)$, and $t_i(v,c) \leq d$ for each uncolored $v \in V(G)$ and $c \in L_i(v)$.
At this point, we may apply Proposition~\ref{prop: final blow} to the subgraph of $G$ induced by the uncolored vertices and the list assignment $L_i(\cdot)$.

Before we formally describe our coloring procedure, we introduce a few relevant parameters.
First, we fix $K \coloneqq \frac{\epsilon}{1000}$
and $\eta \coloneqq \frac{K}{\log\Delta}$. Then, for each $i \geq 1$,
we define:
\begin{eqnarray*}
    L_1 &=& (1+\epsilon) \frac{\Delta}{\log \Delta} = k, \\
    T_1 &=& \Delta, \\
    \keep_i &=& \left (1 - \eta L_i^{-1}  \right )^{T_i}, \\
    L_{i+1} &=& L_i \keep_i \left (1  - \frac{1}{\log^{2} \Delta} \right ), \\
    T_{i+1} &=& T_i \keep_i \left (1 - \eta \keep_i \right )\left (1  + \frac{1}{\log^{2} \Delta} \right ).
\end{eqnarray*}
We note that these values are well defined, as $T_1$ and $L_1$ define $\keep_1$, which in turn define $T_2$ and $L_2$, which in turn define $\keep_2$, and so on.

Let us now formally state our coloring procedure.
For each $i \geq 1$, we execute the following two-part procedure: 

\begin{algorithm}[htb!]
    \caption{\textbf{Nibble (iteration i)}}\label{algorithm: nibble}
    \begin{enumerate}[itemsep = .2cm, label=(N\arabic*)]
        
        \item \label{step:activate} \emph{Activate} each uncolored $v \in V(G)$ independently with probability $\eta$.

        \item \label{step:copy} For each uncolored $v \in V(G)$, make a copy $L_{i+1}(v) = L_i(v)$.
        
        \item \label{step:assignment} For each activated vertex $v$, assign a color $c \in L_i(v)$ to $v$ uniformly at random, and delete $c$ from $L_{i+1}(u)$ for all $u \in T_i(v,c)$.

        \item \label{step: color} For each activated vertex $v$ with assigned color $c$, if $c$ is not deleted from $L_{i+1}(v)$, then color $v$ with $c$. 
        
        \item \label{step:coin_flip} For each uncolored $v \in V(G)$ and $c \in L_i(v)$,  perform an independent \emph{equalizing coin flip}, by which $(v,c)$ \emph{survives} with probability 
        \[\Eq_i(v,c) \coloneqq \min\left\{1,\, \left (1 - \eta L_i^{-1} \right )^{T_i - t_i(v,c)}\right\}.\]
        If $(v,c)$ does not survive, then delete $c$ from $L_{i+1}(v)$.

    \end{enumerate}
\end{algorithm}

\begin{algorithm}[htb!]
    \caption{\textbf{Trim (iteration i)}}\label{algorithm: trim}
    \begin{enumerate}[itemsep = .2cm, label=(T\arabic*)]
        \item \label{step:trim} For each uncolored $v \in V(G)$, delete $\max\{|L_{i+1}(v)| - L_{i+1},0\}$ colors from $L_{i+1}(v)$  arbitrarily. 

        \item \label{step:prune-edges}
        For each edge $uv \in E(G)$ for which $L_{i+1}(u) \cap L_{i+1}(v) = \emptyset$, delete the edge $uv$ from $E(G)$.

    \end{enumerate}
\end{algorithm}

Note that we make a distinction between \emph{assigning} a color $c$ to a vertex $v$ and \emph{coloring} a vertex $v$ with the color $c$ (Steps~\ref{step:assignment} and~\ref{step: color} of Algorithm~\ref{algorithm: nibble} respectively). During an iteration,
a vertex $v$ that is colored with the color $c$ is necessarily assigned the color $c$, but 
a vertex $v$ that is assigned the color $c$ may not be colored with the color $c$ (if some other vertex $u \in N(v)$ is also assigned the color $c$).
We let $\phi_i$ be the partial coloring computed during the $i$-th iteration.

The quantities $T_i$ and $L_i$
defined earlier approximate the color degrees and list sizes, respectively,
throughout the iterations of our procedure.
We will see that
before each iteration $i$, it will hold for each uncolored $v \in V(G)$ and $c \in L_i(v)$ that $|L_i(v)| \geq L_i$ and $t_i(v,c) \leq T_i$.
Furthermore, we will see shortly that $ 1\ll T_{i^*} \leq \frac 18 L_{i^*}$ for some integer $i^* \geq 1$. Therefore, if the inequalities $|L_i(v)| \geq L_i$ and $t_i(v,c) \leq T_i$
hold throughout each iteration of our procedure,
then after $i^*$ iterations, we reach a situation in which Proposition~\ref{prop: final blow} can be applied to extend our partial coloring of $V(G)$ to a 
proper $L$-coloring of $G$. 
As mentioned in Section~\ref{subsection: proof overview}, while showing that the list sizes $|L_i(v)|$ are concentrated close to their expected values is fairly straightforward, showing that the color degrees $t_i(v,c)$ are concentrated from above is difficult. 
To do so, we need to introduce several additional parameters:
\begin{itemize}
     \item For each $X \in \mathcal X$ and $c \in \bigcup_{v \in X} L(v)$, let 
     \[X_i(c) = \{v \in X \,:\, \text{$v$ is uncolored and $c \in L_i(v)$}\} .\] 
    \item 
    For each $X \in \mathcal X$ of size at least $2$, $v \in X$, and $c \in L(v)$,
    let 
    \[
        B_i^X(v,c) = T_i(v,c) \cap B^X.
    \]
        \item 
        For each stranger pair $u,v \in V(G)$ and $c \in L(u) \cap L(v)$,
        define 
        \[
            Q_i(u,v,c) = T_i(u,c) \cap T_i(v,c).
        \]
    \item For each $X \in \mathcal X$ and distinct friend pair $u,v \in X$,
    define 
    \[
        Q^X_i(u,v,c) = B_i^X(u,c) \cap B_i^X(v,c) = T_i(u,c) \cap T_i(v,c) \cap B^X.
    \]
\end{itemize}
We write 
$q_i(u,v,c) = |Q_i(u,v,c)|$, $q_i^X(u,v,c) = |Q_i^X(u,v,c)|$, and $b_i^X(v,c) = |B_i^X(v,c)|$.

We also define quantities $B_i$, $Q_i$, and $X_i$ for $i \geq 1$
(the purposes of these quantities will be made clear shortly):
\begin{eqnarray*}
Q_1& = & \frac{\gamma \Delta}{10\log^{18} \Delta},   \\
Q_{i+1} & = &Q_i  \keep_i  \left (1 - \eta \keep_i \right )\left (1  + \frac{1}{\log^{2} \Delta} \right ),  \\
X_{i} &=& T_i^{\gamma}, \\
B_1 &=&  \Delta (\gamma  - \log^{-2} \Delta), \\
B_{i+1} & =&  B_i \keep_i \left (1 - \eta \keep_i  \right ) \left (1+\frac{1}{\log^2 \Delta} \right ).
\end{eqnarray*}

The goal is to show the following property holds throughout our nibble procedure.

\begin{definition}[Property $P(i)$]
Consider the partial coloring $\cup_{j < i}\phi_j$.
The following list contains the desired structural properties of the uncolored subgraph and the lists $L_i(\cdot)$.
\begin{enumerate}[label=(P\arabic*)]
    \item For each uncolored $v \in V(G)$, $ |L_i(v)| = L_i$.
    \item For each uncolored $v \in V(G)$ and $c \in L_i(v)$, $t_i(v,c) \leq T_i$.
    \item \label{item:stranger-Q} For each stranger pair $u,v \in V(G)$ and color $c \in L_i(u) \cap L_i(v)$,
    $q_i(u,v,c)  \leq Q_i$.
    \item \label{item:friend-QX} For each $X \in \mathcal X$, each distinct friend pair $u,v \in X$, and color $c \in L_i(u) \cap L_i(v)$, $q^X_i(u,v,c) \leq Q_i$.
    \item For each $X \in \mathcal X$ and color $c$, $|X_i(c)| \leq X_i$.
    \item For each $X \in \mathcal X$ of size at least $2$, $v \in X$, and $c \in L_i(v)$, $t_i(v,c) - b^X_i(v,c) \leq T_i - B_i$.
\end{enumerate}

\end{definition}

Note that if Property $P(i)$ holds, then we have $\Eq_i(v, c) = (1 - \eta L_i^{-1})^{T_i - t_i(v, c)}$.
Let us begin by showing that $P(1)$ holds.
\begin{lemma}\label{lemma: P(1)}
    $P(1)$ holds. 
\end{lemma}
\begin{proof}
    We check the six conditions of $P(1)$ separately.
\begin{itemize}
    \item For all $v \in V(G)$, $|L_1(v)| = k = (1+\epsilon) \frac{\Delta}{\log \Delta} = L_1$.
   \item  For each $v \in V(G)$ and $c \in L_1(v)$,
    $t_1(v,c) \leq |N(v)| = \Delta = T_1$.
    \item For each stranger pair $u,v$ and color $c \in L_1(u) \cap L_1(v)$,
    it follows from the definition of strangers that 
    $q_1(u,v,c) \leq |N(u) \cap N(v)| \leq \frac{\Delta}{\log^{20} \Delta} < Q_1$.
    \item For each $X \in \mathcal X$, distinct pair $u,v \in X$, and $c \in L_1(u) \cap L_1(v)$,
    it follows from our assumption that $q^X_1(u,v,c) \leq |N(u) \cap N(v) \cap B^X| \leq \frac{\Delta}{\log^{20} \Delta} < Q_1$.
    \item For each $X \in \mathcal X$ and color $c$,
    $|X_1(c)| \leq |X| \leq \Delta^{\gamma} = T_1^{\gamma} = X_1$.
    \item
    Consider $X \in \mathcal X$ of size at least $2$, $v \in X$, and $c \in L_1(v)$.
    By assumption,
    $v$ has at least $B_1 + d(v) - \Delta$
    neighbors in $B^X$.
    Therefore, 
    \[|T_1(v,c) \setminus B_1^X(v,c)| \leq d(v) - (B_1 + d(v) - \Delta) = \Delta - B_1=  T_1 - B_1.\]
    
\end{itemize}
Therefore, $P(1)$ holds.
\end{proof}

The goal is to show that $P(i)$ holds for all $i \leq i^*$, where $i^*$ is the minimum integer such that $L_{i^*} \geq 8T_{i^*}$.
Before we state the main result regarding our coloring procedure, we state certain properties of our parameters which will be useful in the proof of the result.

\begin{lemma}
\label{lemma: properties}
    Let $r_i = T_i/L_i$.
    For $i \leq \min\{\log^{3/2} \Delta,\,i^*\}$, we have
    \begin{enumerate}[label=(\roman*)]
        \item \label{item:keep} $\keep_i \geq \keep_1 > e^{-K} $.
        \item \label{item:r} $r_{i} \leq r_1 \left ( 1 - \eta \keep_1 \right)^{i-1} \exp \left ( 4 (i-1) \log^{-2} \Delta \right ) \leq r_1 < \log \Delta$.
        \item \label{item:L} $L_{i} > 10 \Delta^{\epsilon/2}$.
        \item \label{item:T} $T_i > \Delta^{\epsilon/2}$.
        \item \label{item:B-and-T} $ \frac{B_i}{T_i}= \gamma - \log^{-2} \Delta $. 
        \item \label{item:T-and-Q} $\frac{T_i}{Q_i} = \frac{10}{\gamma} \log^{18} \Delta$. 
        \item \label{item:B-and-Q} $\frac{B_i}{Q_i} > 5\log^{18} \Delta$.
        \item \label{item:Q-and-X} $\frac{Q_i}{X_i} > \Delta^{(1 - \gamma) \epsilon / 3}$.
    \end{enumerate}
\end{lemma}

See Section~\ref{sec: proof of properties} for the proof of Lemma~\ref{lemma: properties}.
The main result regarding Algorithms~\ref{algorithm: nibble}~and~\ref{algorithm: trim} is the following:

\begin{proposition}\label{prop: nibble}
    The following holds for all $i < i^*$.
    Suppose that $P(i)$ holds.
    Then there exists an outcome $\phi_i$ of Algorithm~\ref{algorithm: nibble} such that $P(i+1)$ holds after performing the trimming steps of Algorithm~\ref{algorithm: trim}.
\end{proposition}

We defer the proof of Proposition~\ref{prop: nibble} to Section~\ref{sec: proof of prop nibble}.
Next, with Lemma~\ref{lemma: properties} in mind, we show that $i^*$ is small.

\begin{lemma}
\label{lem:r-rate}
$i^* \leq \log^{3/2}\Delta$.
\end{lemma}
\begin{proof}
Suppose that $i^* > \log^{3/2}\Delta$.
In particular, we have
$r_i > \frac 18$ for each $i \in [\log^{3/2} \Delta]$.
Setting $i = \log^{4/3}\Delta  + 1$ in Lemma~\ref{lemma: properties}~\ref{item:r},
we have 
\begin{eqnarray*}
r_{i} &\leq&   r_1 \left (1 - \eta \keep_1 \right)^{\log^{4/3} \Delta}  \exp ( 4 \log^{4/3} \Delta \log^{-2} \Delta) \\
&<&   \log \Delta \cdot \exp \left ( -K  \keep_1 \log^{1/3} \Delta \right ) .
\end{eqnarray*}
By Lemma~\ref{lemma: properties}~\ref{item:keep}, and since $K$ is a positive constant bounded away from zero, we have that 
\[r_{i}  = o(1) < \frac 18,\]
a contradiction.
\end{proof}

We note that, as a result of Lemma~\ref{lem:r-rate}, all conclusions of Lemma~\ref{lemma: properties} hold
for each $i \in [i^*]$.
Before we prove the auxiliary results in this section, let us show how they imply Theorem~\ref{thm:general}.

\begin{proof}[Proof of Theorem~\ref{thm:general}]
    By Lemma~\ref{lemma: P(1)} and Proposition~\ref{prop: nibble}, $P(i)$ holds for all $i \in [i^*]$.
    Furthermore, by definition of $i^*$, we have $L_{i^*} \geq 8T_{i^*}$.
    In particular, we may apply Proposition~\ref{prop: final blow} to the subgraph induced by the uncolored vertices with respect to the coloring $\phi\coloneqq \bigcup_{i = 1}^{i^*-1}\phi_i$ and the list assignment $L_{i^*}(\cdot)$.
    Let $\psi$ be the corresponding coloring.
    The desired coloring of $G$ is $\phi\cup \psi$.
\end{proof}

\subsection{Proof of Lemma~\ref{lemma: properties}}\label{sec: proof of properties}
\stepcounter{ForClaims} \renewcommand{\theForClaims}{\ref{lemma: properties}}

    We begin with a proof of the inequalities in~\ref{item:keep},~\ref{item:r},~\ref{item:B-and-T}, and~\ref{item:T-and-Q},
    which do not require induction on $i$.
    By
    Lemma~\ref{lem:1-x LB} and the fact that $\eta L_1^{-1}$ is sufficiently small,
    \[
        \keep_1 = (1 - \eta L_1^{-1} )^{T_1} > \exp \left ( -(1+\epsilon) \frac{\eta T_1}{L_1} \right )
        = e^{-K}.\] 
        
    For $i \geq 1$, we have 
    \[(1 - \eta \keep_1)^{i-1} \leq \exp ( - \eta (i-1) e^{-K} ) \leq \exp( - 4 (i-1) \log^{-2} \Delta ),\]
    where the last inequality follows for $\Delta$ sufficiently large.
    With this in hand, we have
    \[
         r_1 \left ( 1 - \eta \keep_1 \right)^{i-1} \exp \left ( 4 (i-1) \log^{-2} \Delta \right )  \leq r_1  = \frac{\log \Delta}{1+\epsilon} < \log \Delta.
    \]

    For~\ref{item:B-and-T}, note the following:
    \[\frac{B_i}{T_i} = \frac{B_1}{T_1} = \gamma - \log^{-2}\Delta.\]

    Finally, we have
    \[\frac{T_i}{Q_i} = \frac{T_1}{Q_1} = \frac{10}{\gamma} \log^{18}\Delta,\]
    completing the proof of~\ref{item:T-and-Q}.

    Now, we prove the rest of the inequalities by induction on $i$.
    When $i=1$, 
   ~\ref{item:L} and~\ref{item:T}
    hold as $L_i$ and $T_i$ are both of the order $\Delta^{1 - o(1)}$;
   ~\ref{item:B-and-Q} holds as $\frac{B_1}{Q_1} = (10 - o(1)) \log^{18} \Delta$;
   ~\ref{item:Q-and-X} holds as $\frac{Q_1}{X_1} = \Delta^{1 - \gamma - o(1)}$.
    Now, suppose that $1 \leq i < \min\{\log^{3/2} \Delta,\, i^*\}$ and that the statement holds for all values $1 \leq j \leq i$. 
    (Note that $r_j > \frac 18$ for every $1 \leq j \leq i$ since $i < i^*$.)
    We aim to show that the inequalities in the lemma hold for $i+1$.
    We do so through a sequence of claims.

    First, let us show that~\ref{item:keep} holds.
    \begin{claim}
        \label{claim:keep}
        The inequality $\keep_{i+1} > \keep_i \geq   \keep_1$ holds.
    \end{claim}
    \begin{claimproof}   
    We recall $\keep_{i+1} = (1 - \eta L_{i+1}^{-1})^{T_{i+1}}$. Using Lemma~\ref{lem:1-x LB},
    \[\keep_{i+1} > \exp \left ( - \eta  T_{i+1}  L_{i+1}^{-1} ( 1 + \eta L_{i+1}^{-1}) \right ).\]
    To prove the claim, it suffices to show that 
    \[
     \exp \left ( - \eta  T_{i+1}  L_{i+1}^{-1} ( 1 +  \eta L_{i+1}^{-1}) \right )
     \geq  \exp \left ( - \eta T_i L_i^{-1} \right )
    > \left ( 1 - \eta  L_i^{-1} \right )^{T_i} =  \keep_i.
    \]
    Equivalently, it suffices to show that $
    T_{i+1} L_{i+1}^{-1} (1 +
    \eta L_{i+1}^{-1}) < T_i L_i^{-1}$
    or that 
    \begin{equation}
    \label{eq:keep-calc}
        1 > ( 1 + \eta L_{i+1}^{-1}) \frac{T_{i+1}}{T_i}   \cdot  \frac{L_i }{L_{i+1}} = (1 - \eta \keep_i) ( 1+  \eta L_{i+1}^{-1}) \frac{1 + \log^{-2} \Delta }{1 - \log^{-2} \Delta }.
    \end{equation}

    By the induction hypothesis, $\keep_i \geq \keep_1 > \exp(-K) > \frac 23$.
    As $L_i > \Delta^{\epsilon/2}$ by the induction hypothesis, we have
    \[L_{i+1} > \frac 12 L_i, \qquad \text{and} \qquad 1 + \eta L_{i+1}^{-1} <  1 + L_i^{-1} < 1 + 2\log^{-2} \Delta. \]
    Putting all this together and applying Lemma~\ref{lem:bernoulli}, we have
    \[ ( 1 + \eta L_{i+1}^{-1})  \frac{1 + \log^{-2} \Delta }{1 - \log^{-2} \Delta } < 
     (1 + 2 \log^{-2} \Delta  )^3 < \exp (6 \log^{-2} \Delta  ).\]
    Therefore,  as $\keep_i > \frac 23$,  the expression on the RHS of~\eqref{eq:keep-calc} is at most 
    \[
         \exp\left (- \frac 12 \eta +  6 \log^{-2} \Delta  \right )  =\exp \left ( - \frac{K}{2 \log \Delta}  + 6 \log^{-2} \Delta  \right ) < 1,
    \]
    as desired.
    \end{claimproof}

The next claim shows that~\ref{item:r} holds.    
    \begin{claim} 
    \label{claim:ri}
    The inequality $r_{i+1} < r_1 \left ( 1 - \eta \keep_1 \right)^i \exp \left ( 4 i \log^{-2} \Delta \right )$ holds.
    \end{claim}
    \begin{claimproof}
    Using the recurrence relations for $T_i$ and $L_i$,
    we have
    \begin{eqnarray*}
    r_{i+1} = \frac{T_{i+1}}{L_{i+1}} & = &\frac{T_i \left ( 1 - \eta \keep_i \right )}{L_i} \cdot \frac{1 + \log^{-2} \Delta }{1 - \log^{-2} \Delta } <  r_i \left (1 - \eta \keep_i \right ) (1 + 2\log^{-2} \Delta )^2 .
    \end{eqnarray*}
This implies that
    \[r_{i+1} < r_1   \prod_{j=1}^i \left (1 - \eta \keep_i \right ) \prod_{j=1}^i (1 + 2\log^{-2} \Delta )^2 . \]
    By the induction hypothesis, $\keep_i \geq \keep_1$. Therefore,
    \begin{equation}
    \label{eq:r+1 UB}
        r_{i+1} < r_1 \left ( 1 - \eta \keep_1 \right)^i \exp \left ( 4 i \log^{-2} \Delta \right ),
    \end{equation}
    as desired.
\end{claimproof}

To prove~\ref{item:L} and~\ref{item:T}, we need the following claim.
\begin{claim}
\label{claim:keep-prod}
    $\prod_{j=1}^i \keep_j > \Delta^{\frac 23 \epsilon - 1}$.
\end{claim}
\begin{claimproof}
    By Lemma~\ref{lem:1-x LB},
    $\keep_j  = (1 - \eta L_j^{-1} )^{T_j} > \exp(-e^K \eta r_j)$ for each $1 \leq j \leq i$.
    Therefore,
    \[
           \prod_{j=1}^i \keep_j >  \prod_{j=1}^i \exp \left ( - e^K \eta  r_j\right).
    \]
    By the induction hypothesis and since $i < \log^{3/2}\Delta$, we have
    $r_{j} < r_1 \left ( 1 - \eta\keep_1 \right)^{j-1} e^K $ for each $1 \leq j \leq i$. Therefore,
        \begin{align*}
         \prod_{j=1}^i \keep_j  &>   \prod_{j=1}^i \exp \left ( - e^{2K} \eta r_1 \left ( 1 - \eta  \keep_1 \right)^{j-1}  \right ) \\
        &= \exp\left(-e^{2K} \eta r_1  \sum_{j = 1}^i \left(1 - \eta\keep_1\right)^{j - 1} \right) \\
        &>  \exp\left(- e^{2K} \eta 
        r_1\sum_{j = 1}^\infty \left(1 - \eta \keep_1\right)^{j - 1} \right) \\
        &= \exp\left(-e^{2K} \eta 
         \left( \eta \keep_1\right)^{-1} r_1 \right) \\
        &>  \exp\left(-e^{3K}   r_1 \right) & \text{(since $\keep_1 > e^{-K}$) } \\ 
        &= \exp\left(- \frac{e^{3K}  \log \Delta}{1 + \epsilon} \right) & \text{(since $r_1 = \frac{ \log \Delta}{1 + \epsilon}$)}.
    \end{align*}
    As $\frac{1}{1 + \epsilon} < 1 - \frac{4}{5} \epsilon < \exp\left(-\frac{4}{5} \epsilon\right)$
    and $K = \frac{1}{1000} \epsilon$,
    Lemma~\ref{lem:1-x LB} implies
    \[
        \frac{e^{3K} }{1 + \epsilon} < \exp\left(3K - \frac{4}{5}\epsilon \right ) =  \exp \left ( - \frac{199}{250} \epsilon \right ) < 1 - \frac 23 \epsilon .
    \]
    It follows that 
    \[
         \prod_{j=1}^i \keep_j  > \exp\left(-\left (1 - \frac{2}{3} \epsilon\right) \log \Delta\right) =  \Delta^{\frac 23 \epsilon  - 1 },
    \]
    as claimed.
\end{claimproof}

The following claim shows that~\ref{item:L} holds.

\begin{claim}
\label{claim:Li+1}
    The inequality $L_{i+1} > 10 \Delta^{\epsilon/2}$ holds.
\end{claim}
\begin{claimproof}
    Using the recurrence $L_{i+1} = L_i \keep_i( 1 - \log^{-2} \Delta )$,
    we have 
    \[L_{i+1} = L_1 \prod_{j=1}^i\left ( \keep_j ( 1 - \log^{-2} \Delta ) \right ) .\]
    Furthermore, by Lemma~\ref{lem:1-x LB}
    we have 
    \[\prod_{j = 1}^i \left(  1 - \log^{-2} \Delta \right) > \exp\left(-2 i \log^{-2} \Delta \right) > \frac 12 ,\] 
    since $ i \leq \log^{3/2} \Delta$ and $\Delta$ is sufficiently large.
    Therefore, by Claim~\ref{claim:keep-prod}, we have
    \[L_{i+1} > \frac 12 L_1 \prod_{j=1}^i \keep_j > \frac 12 L_1 \Delta^{\frac 23 \epsilon - 1} = \frac{ (1 + \epsilon ) \Delta^{\frac 23 \epsilon } }{2 \log \Delta} > 10 \Delta^{\epsilon/2} ,\]
    as desired.
    \end{claimproof}

The next claim shows that~\ref{item:T} holds:
\begin{claim}\label{claim: T i+1}
    The inequality $T_{i+1} > \Delta^{\epsilon / 2}$ holds.
\end{claim}
\begin{claimproof}
    Noting that $r_i > \frac 18$ and $\eta \keep_1 = o(1)$, we have
    \[
        T_{i+1} = L_{i+1}r_{i+1}       
        = L_{i+1}r_i (1 - \eta \keep_i) \cdot \frac{1 + \log^{-2} \Delta }{1 - \log^{-2} \Delta } > \frac{L_{i+1}}{10},
    \]
    for $\Delta$ sufficiently large.
    Thus, as $L_{i+1} > 10 \Delta^{\epsilon/2}$ by Claim~\ref{claim:Li+1},
    $T_{i+1} > \Delta^{\epsilon/2}$.
\end{claimproof}

To show~\ref{item:B-and-Q},
we note that the inequality $\frac{B_{i+1}}{Q_{i+1}} > 5 \log^{18} \Delta$ follows from~\ref{item:T-and-Q} and the fact that $\frac{B_i}{T_i} > \frac 12 \gamma$ by~\ref{item:B-and-T} for $\Delta$ sufficiently large.

Finally, using the inequality~\ref{item:T} for $i+1$ as proven in Claim~\ref{claim: T i+1}, we have
\[  
    \frac{Q_{i+1}}{X_{i+1}} = \frac{\gamma T_{i+1} / (10 \log^{18} \Delta)}{T_{i+1}^{\gamma} } = \frac{\gamma T_{i+1}^{1-\gamma} }{10 \log^{18} \Delta} > \Delta^{(1 - \gamma) \epsilon / 3},
\]
for $\Delta$ sufficiently large, completing the proof that~\ref{item:Q-and-X} holds for $i+1$. Lemma~\ref{lemma: properties} follows by induction.

We conclude this subsection with a lemma that gives a very rough estimate of how much each of our values changes from one iteration to the next.
By Lemmas~\ref{lemma: properties}, Lemma~\ref{lem:r-rate},
and the fact that $\Delta$ is sufficiently large, we see that for $i \in [i^* - 1]$, $T_i$, $L_i$, $X_i$, $B_i$, and $Q_i$ are on the order of $\Delta^{\Omega(1)}$. 
Furthermore, by Lemma~\ref{lemma: properties} and Lemma~\ref{lem:1-x LB}, we have
\[e^{-K} < \keep_i = (1 - \eta L_i^{-1})^{T_i} < \exp(-\frac 18 \eta) < 1 - \log^{-2} \Delta.\]
With this in hand, the following inequalities are easily observable:
\begin{lemma}
    \label{lem:rough-inequalities}
    For each $i \in [i^*-1]$ and symbol $A \in \{B,L,Q,T,X\}$, $A_{i+1} \leq A_i \leq 2A_{i+1}$.
\end{lemma}

\subsection{Proof of Proposition~\ref{prop: nibble}}\label{sec: proof of prop nibble}

In this section, we will prove Proposition~\ref{prop: nibble} which states that if $i < i^*$ and $P(i)$ holds then $P(i+1)$ holds as well.
Note that as a result of Lemma~\ref{lem:r-rate}, we may assume all conclusions of Lemma~\ref{lemma: properties} hold for $i < i^*$.
For the proof, let $G_i$ be the graph induced by the uncolored vertices in $G$ with respect to the partial coloring $\cup_{j < i}\phi_j$ and let $L_i\,:\,V(G_i) \to 2^{\mathbb N}$ be the associated list assignment.
A key step in our proof is to show that the desired conditions hold locally with high probability.

\begin{lemma}\label{lemma: locally conditions hold}
    The following hold:
    \begin{enumerate}[label=(\roman*)]
        \item\label{item: L conc} For each $v \in V(G_i)$, $\Pr(|L_{i+1}(v)| < L_{i+1}) < \exp(-\log^2 \Delta)$.
        \item\label{item: T conc} For each $v \in V(G_i)$ and $c \in L_i(v)$, $\Pr(t_{i+1}(v,c) > T_{i+1}) < \exp(-\log^2 \Delta)$.
        \item\label{item: Q conc 1} For each stranger pair $u,v \in V(G_i)$ and color $c \in L_i(u) \cap L_i(v)$,
        \[\Pr(q_{i+1}(u,v,c)  > Q_{i+1}) < \exp(-\log^2 \Delta).\]
        \item\label{item: Q conc 2} For each $X \in \mathcal X$, each distinct friend pair $u,v \in X$, and color $c \in L_i(u) \cap L_i(v)$, \[\Pr\left(q^X_{i+1}(u,v,c) > Q_{i+1}\right) < \exp(-\log^2 \Delta).\]
        \item\label{item: X conc} For each $X \in \mathcal X$ and color $c$, $\Pr(|X_{i+1}(c)| > X_{i+1}) < \exp(-\log^2 \Delta)$.
        \item\label{item: B conc} For each $X \in \mathcal X$ of size at least $2$, $v \in X$, and $c \in L_i(v)$, \[\Pr\left(t_{i+1}(v,c) - b^X_{i+1}(v,c) > T_{i+1} - B_{i+1}\right) < \exp(-\log^2 \Delta).\]
    \end{enumerate}
\end{lemma}

Before we prove the above result, let us show how it implies Proposition~\ref{prop: nibble}.

\begin{proof}[Proof of Proposition~\ref{prop: nibble}]
    We define the following bad events, which contain all undesirable outcomes that may occur after iteration $i$ of the nibble. For each bad event, we also specify a \emph{center}.
    \begin{itemize}
        \item For each $v \in V(G_i)$, let $\mathcal L(v)$ be the event that $|L_{i+1}(v)| < L_{i+1}$.
        Let $v$ be the center of $\mathcal L(v)$.
        \item For each $v \in V(G_i)$ and for each $c \in L_{i}(v)$, let $\mathcal T(v,c)$ be the event that $t_{i+1}(v,c) > T_{i+1}$. Let $v$ be the center of $\mathcal T(v,c)$.
        \item For each $X \in \mathcal X$ and $c \in \bigcup_{v \in X} L(v)$, let $\mathcal E_X(c)$ be the event that $|X_{i+1}(c)| > X_{i+1}$.
        Let some vertex $v \in X$ be the center of $\mathcal E_X(c)$.
        \item For each $X \in \mathcal X$, vertex $v \in X \cap V(G_i)$, and $c \in L_{i}(v)$, let $\mathcal B^X(v,c)$ be the event that $t_{i+1}(v,c) - b^X_{i+1}(v,c) > T_{i+1} - B_{i+1}$.
        Let $v$ be the center of $\mathcal B^X(v,c)$.
        \item For each stranger pair $u,v \in V(G_i)$, and for each $c \in L_{i}(u) \cap L_{i}(v)$, 
        if $Q_{i}(u,v,c)$ is nonempty, let $\mathcal Q(u,v,c)$ be the event that $q_{i+1} (u,v,c) > Q_{i+1}$.
        Similarly, for each $X \in \mathcal X$ and distinct friend pair $u,v \in X$,
        and for each $c \in L_{i}(u) \cap L_{i}(v)$ such that $Q_{i}(u,v,c)$ is nonempty,
         let $\mathcal Q(u,v,c)$ be the event that $q_{i+1} (u,v,c) > Q_{i+1}$.
         Let $v$ be the center of $\mathcal Q(u,v,c)$.
    \end{itemize}
    It is easy to see that if no bad event occurs, then $P(i+1)$ holds.
    In particular, since the sets $T_{i+1}(v,c)$, $T_{i+1}(v,c) \setminus B^X_{i+1}(v,c)$, $X_{i+1}(c)$, and $Q_{i+1}(u,v,c)$ considered in $P(i+1)$
    cannot increase in size during the $i$-th application of Algorithm~\ref{algorithm: trim},
    an outcome of the $i$-th application of Algorithm~\ref{algorithm: nibble} with no bad event listed above implies that each of these sets is small enough to satisfy $P(i+1)$.
    Furthermore, if $|L_{i+1}(v)| \geq L_{i+1}$ for each vertex $v$ that is uncolored after the $i$-th application of Algorithm~\ref{algorithm: nibble},
    then after the $i$-th application of Algorithm~\ref{algorithm: trim}, we will have $|L_{i+1}(v)| = L_{i+1}$.
    
    Now, we use the Lov\'asz Local Lemma (Lemma~\ref{lem:LLL}) to show that with positive probability, no bad event occurs.
    We note that by Lemma~\ref{lemma: locally conditions hold}
    the probability of each bad event is at most $\exp(-\log^2 \Delta)$.
    Next, we estimate the number of bad events for which each vertex $v \in V(G)$ is the center. For each $v \in V(G)$,
    $v$ is the center of at most $k$ events $\mathcal T(v,c)$,
    at most $\Delta^{O(1)}$ events $\mathcal B^X(v,c)$,
    at most $\Delta^{O(1)}$ events $\mathcal E_X(c)$,
    and one event $\mathcal L(v)$.
    Furthermore, since a bad event $\mathcal Q(u,v,c)$ is only defined if $Q_i(u,v,c)$ is nonempty (which requires that $u$ and $v$ have at least one common neighbor),
    $v$ is the center of at most $k \Delta^2$ bad events $\mathcal Q(u,v,c)$.
    Therefore, each $v \in V(G)$ is the center of at most $\Delta^{O(1)}$ bad events.
    Furthermore, it is easy to verify that for each bad event $\mathcal E$ with center $v$, $\mathcal E$
    depends entirely on
    the activations and coin flips of vertices $w$
    for which $\dist(v,w) \leq 5$; therefore, $\mathcal E$ is mutually independent with all but at most $\Delta^{O(1)}$ other bad events.
    Thus, by the Lov\'asz Local Lemma (Lemma~\ref{lem:LLL}),
    since $4 \Delta^{O(1)} \exp(-\log^2 \Delta) < 1$,
    it holds with positive probability that no bad event occurs, and hence that $P(i+1)$ holds.
\end{proof}

The following general lemma will be key to the proof of Lemma~\ref{lemma: locally conditions hold}.
As the proof is rather involved, we defer it to Section~\ref{sec:great-expectations}.

\begin{lemma}\label{lem:great-expectations}
    Let $R, S \subseteq V(G)$ and $i \in [i^*-1]$ be such that $P(i)$ holds.
    Let $\alpha \geq 4 \log^{18} \Delta$ be an integer.
    Fix a color $c$, and for each $v \in S$,
    write $N_R(v,c) = T_i(v,c) \cap R$.
    Suppose that the following hold:
    \begin{enumerate}[label=(C\arabic*)]
        \item For each $v \in S$, $c \in L_i(v)$;
        \item\label{item: bounds on |S|} $\log^{16} \Delta \leq |S| \leq T_i$;
        \item\label{item: bound on alpha} $T_i -t_i(v,c) + |N_R(v,c)| \geq \alpha$ for each $v \in S$;
        \item\label{item: codegree bound great expectations}
        For each pair $u,v \in S$, $|N_R(u,c) \cap N_R(v,c)  | \leq \frac{\alpha}{4 \log^{18} \Delta}$.
    \end{enumerate}
    Let $S' \subseteq S$ be the set 
    of vertices $v \in S$ for which the following hold throughout iteration $i$ of Algorithm~\ref{algorithm: nibble}:
    \begin{itemize}
        \item $v$ is not colored during Step~\ref{step: color};
        \item no $u \in N_R(v,c)$ is activated and assigned $c$; and 
        \item $(v,c)$ survives its equalizing coin flip.
    \end{itemize}
    Then, 
    \[
    \Pr \left (|S'| \geq \left ( 1 - \eta L_i^{-1} \right)^{\alpha}(1 - \eta \keep_i)|S| + \frac{|S|}{2 \log^{2} \Delta} \right  ) < \exp(-\log^3 \Delta).
    \]
\end{lemma}

In some of our concentration arguments, we aim to apply Lemma~\ref{lem:great-expectations} with $R = V(G)$, and $\alpha = T_i$.
However,
since some vertex pairs $w,w' \in S$ may be friends,
we cannot guarantee that $|N_R(w,c) \cap N_R(w',c)| \leq \frac{\alpha}{4 \log^{18} \Delta}$.
Therefore, we partition $S$ into sets of mutual strangers, and we apply Lemma~\ref{lem:great-expectations} to each part.
The following lemma shows that we can partition any set into ``large'' sets of mutual strangers.

\begin{lemma}
\label{lem:stranger-partition}
    Let $i \in [i^*-1]$ be such that $P(i)$ holds.    Let $A \subseteq V(G_i)$, and suppose that $c \in L_i(v)$ for each $v \in A$.
    Then, $A$ can be partitioned into sets of mutual strangers, with each part of size at least  $\frac{|A|}{X_i \Delta^{(1 - \gamma) \epsilon/10} }$.
\end{lemma}
\begin{proof}
    Let $H$ be an auxiliary graph with vertex set $V(H) = A$, where two vertices $w,w' \in A$ are adjacent in $H$ if and only if $w$ and $w'$ are friends.
    Given a vertex $w \in A$,
    Condition~\eqref{item:main-2b} of Theorem~\ref{thm:general}
    implies that
    there is a family $\mathcal X_w \subseteq \mathcal X$ of size at most $\Delta^{(1 - \gamma )\epsilon/12}$
    such that each friend of $w$ is contained in a set $X \in \mathcal X_w$.
    For each $X \in \mathcal X_w$,
    $|X_i(c)| \leq X_i$ by Property $P(i)$,
    so $w$ is friends with at most $X_i \Delta^{(1 - \gamma )\epsilon/12}$ vertices $w' \in A$.
    Therefore, the maximum degree of $H$ is at most 
    $X_i \Delta^{(1 - \gamma )\epsilon/12}$.
By the Hajnal-Szemer\'edi Theorem (Lemma~\ref{lem:HS}),
$H$ has an equitable coloring with at most $X_i \Delta^{(1 - \gamma) \epsilon/12}  + 1 < X_i \Delta^{(1 - \gamma) \epsilon/10} $ colors, so each color class has size at least $\frac{|A|}{X_i \Delta^{(1 - \gamma) \epsilon/10} }$.
Each color class of $H$ corresponds to a part in a partition of $A$ into sets of mutual strangers.
\end{proof}

As a corollary to Lemma~\ref{lem:great-expectations}~and~\ref{lem:stranger-partition}, we obtain the following result.
The proof follows by applying Lemma~\ref{lem:great-expectations} with $R = V(G)$ and $\alpha = T_i$ to each part from Lemma~\ref{lem:stranger-partition} and then taking a union bound.

\begin{corollary}\label{cor:great-expectations}
    Let $S \subseteq V(G)$ and $i \in [i^*-1]$ be such that $P(i)$ holds.
    Fix a color $c$.
    Suppose that the following hold:
    \begin{itemize}
        \item For each $v \in S$, $c \in L_i(v)$;
        \item $X_i\Delta^{(1-\gamma)\epsilon/8 } \leq |S| \leq T_i$; 
    \end{itemize}
    Let $S' \subseteq S$ be the set 
    of vertices $v \in S$ for which the following hold throughout iteration $i$ of Algorithm~\ref{algorithm: nibble}:
    \begin{itemize}
        \item $v$ is not colored during Step~\ref{step: color};
        \item no $u \in N(v,c)$ is activated and assigned $c$; and 
        \item $(v,c)$ survives its equalizing coin flip.
    \end{itemize}
    Then, 
    \[
    \Pr \left (|S'| \geq 
    \keep_i
    (1 - \eta \keep_i)|S| + \frac{|S|}{2 \log^{2} \Delta} \right  ) < \exp(-\log^2 \Delta).
    \]
\end{corollary}
\begin{proof}
    Using Lemma~\ref{lem:stranger-partition}, we partition $S$ into sets of mutual strangers, with each part of size at least $\frac{|S|}{X_i \Delta^{(1-\gamma)\epsilon/10}}$.
    Consider an arbitrary part $S_0$ in this partition.
    As $|S| \geq X_i \Delta^{(1-\gamma)\epsilon/10}$, it follows that $|S_0| \geq \Delta^{(1-\gamma)\epsilon/40} \gg \log^{16} \Delta$ vertices. We let $R = V(G)$ and $\alpha = T_i$, and we observe that 
    $T_i - t_i(v,c) + |N_R(v,c)| = T_i = \alpha$ for each $v \in S_0$.
    As $S_0$ is a set of mutual strangers, \[|N(u,c) \cap N(v,c)| = |Q_i(u,v,c)| \leq Q_i = \frac{\gamma T_i }{10 \log^{18} \Delta }< \frac{\alpha}{4 \log^{18} \Delta}.\]

    Now, by Lemma~\ref{lem:great-expectations}, it holds with probability at least $1 - \exp(-\log^3 \Delta)$ that
    \[
    |S' \cap S_0| < \keep_i (1-\eta \keep_i) \left ( 1 + \frac{1}{2\log^2 } \right ) |S_0|.
    \]
    The corollary follows by taking a union bound and summing over all parts $S_0$.
\end{proof}

For our concentration arguments in the remainder of this paper,
we
define the following probability spaces:
\begin{enumerate}
    \item For each $v \in V(G_i)$, we let $(\Omega_v,\Sigma_v,\Pr_v)$ denote the probability space corresponding to Steps~\ref{step:activate}~and~\ref{step:assignment} of Algorithm~\ref{algorithm: nibble}. More formally, let $\Omega_{v} = L_i(v) \cup \{0\}$. 
    The element $0$ of $\Omega_{v}$ satisfies $\Pr_{v}(0) = 1 - \eta$ and corresponds to an outcome of Step~\ref{step:activate} in which $v$ is not activated.
    Then, $\Sigma_{v} = 2^{\Omega_{v}}$ is the set of probability-measurable subsets of $\Omega_{v}$, and $\Pr_{v}\,:\,\Sigma_{v} \rightarrow [0,1]$ gives the associated probabilities.

    \item Furthermore, for each  $v \in V(G_i)$ and $c \in L_i(v)$, we let $(\Omega_{v,c}, \Sigma_{v,c}, \Pr_{v,c})$ denote the probability space describing Step~\ref{step:coin_flip} for $(v,c)$ during the Algorithm~\ref{algorithm: nibble}.
    In particular, $\Omega_{v,c} = \{0,1\}$, where we let the element $1$ correspond to the outcome in which $(v,c)$ survives its coin flip, and we let $0$ correspond to the outcome in which $(v,c)$ does not survive its coin flip.
    As such, $\Pr_{v,c}(0) = 1- \Eq_i(v,c)$,
    and $\Pr_{v,c}(1) = \Eq_i(v,c)$.
    We write $\Sigma_{v,c} = 2^{\Omega_{v,c}}$ for the set of probability-measurable subsets of $\Omega_{v,c}$, and $\Pr_{v,c}\,:\,\Sigma_{v, c} \rightarrow [0,1]$ gives the associated probabilities. 
\end{enumerate}
Let $(\Omega, \Sigma, \Pr)$ be the product of these spaces. \label{prob_space}

We are now ready to complete the proof of Lemma~\ref{lemma: locally conditions hold}.
We will consider each item separately.

\paragraph{Proof of Lemma~\ref{lemma: locally conditions hold}~\ref{item: L conc}: concentration of $|L_{i+1}(v)|$.}

We will prove this through a sequence of lemmas.
We begin by computing $\E[|L_{i+1}(v)|]$.

\begin{lemma}\label{lemma: exp L_i}
    $\E[|L_{i+1}(v)|] = \keep_i L_i$.
\end{lemma}

\begin{proof}
    Consider an arbitrary color $c \in L_i(v)$.
    Note that $c\in L_{i+1}(v)$ if and only if (1)~no vertex in $T_i(v, c)$ is activated and assigned the color $c$, and (2)~the pair $(v, c)$ survives its equalizing coin flip.
    As all activations, assignments, and coin flips are independent, we have
    \[\Pr(c\in L_{i+1}(v)) = (1 - \eta\,L_i^{-1})^{T_i - t_i(v, c)}\prod_{u \in T_i(v, c)}(1 - \eta\,|L_i(u)|^{-1}).\]
    By property $P(i)$ all lists have size $L_i$ and so the above simplifies to
    \[(1 - \eta\,L_i^{-1})^{T_i - t_i(v, c)}\prod_{u \in T_i(v, c)}(1 - \eta\,|L_i|^{-1}) = \keep_i.\]
    By linearity of expectation, we have $\E[|L_{i+1}(v)|] = \keep_i L_i$ as claimed.
\end{proof}

Next, we show that $|L_{i+1}(v)|$ is concentrated, completing the proof of Lemma~\ref{lemma: locally conditions hold}~\ref{item: L conc}.

\begin{lemma}\label{lemma: conc L_i}
    $\Pr(|L_{i+1}(v)| < L_{i+1}) \leq \exp(-\log^2 \Delta)$.
\end{lemma}

\begin{proof}
    Consider
    the random variable $Z = L_i - |L_{i+1}(v)|$. By Lemma~\ref{lemma: exp L_i}, we have
    \[\E[Z] = L_i - \E[ |L_{i+1}(v)| ] = L_i - \keep_i L_i = (1 - \keep_i)L_i.\]
    Note that if $|L_{i+1}(v)| < L_{i+1}$, then 
    \[Z - \E[Z] = \keep_i L_i - |L_{i+1}(v)| > \frac{\keep_iL_i}{\log^2\Delta}.\]
    Therefore, if $|L_{i+1}(v)| < L_{i+1}$, then $|Z - \E[Z]| > \keep_iL_i/\log^2\Delta$. 
    Hence, it suffices to show that 
    \[
    \Pr(|Z - \E[Z]| > \keep_i L_i/\log^2\Delta) \leq \exp(-\log^2 \Delta).
    \]

    To concentrate $Z$, we will employ Lemma~\ref{lem:conc-ineq}.
    To this end, recall the probability space $(\Omega, \Sigma, \Pr)$ defined on page~\pageref{prob_space}. We let $\Omega^* = \emptyset$.    
    First, we show that $Z$ is $(1, 1)$-observable (see Definition~\ref{defn: rd observable}).
    We have that $Z = \sum_{c \in L_i(v)} Y_c$,
    where $Y_c$ for $c \in L_i(v)$ is a binary random variable such that $Y_c = 1$ if and only if $c \notin L_{i + 1}(v)$ after Algorithm~\ref{algorithm: nibble}. For each $c \in L_i(v)$, we have that $Y_c$ is $1$-verifiable with verifier $R_c$ defined as follows:
    \begin{itemize}
        \item if $Y_c(\omega) = 0$, then $R_c(\omega) = \emptyset$;
        \item if some vertex $u \in T_i(v, c)$ is activated and assigned $c$, then we let $R_c(\omega) = \{u\}$;
        \item if $(v, c)$ loses its coin flip in Step~\ref{step:coin_flip} of Algorithm~\ref{algorithm: nibble}, we let $R_c(\omega) = \{(v,c)\}$.
    \end{itemize}    
    To finish showing that $Z$ is $(1, 1)$-observable, we must show the following:
    \begin{itemize}
        \item for every $\omega \in \Omega \setminus \Omega^*$ and $u \in V(G)$, $u \in R_c(\omega)$ for at most $1$ value of $c \in L_i(v)$, and
        \item for every $\omega \in \Omega \setminus \Omega^*$ and $c' \in L_i(v)$, $(v, c') \in R_c(\omega)$ for at most $1$ value of $c \in L_i(v)$. 
    \end{itemize}
    To this end, fix $\omega \in \Omega \setminus \Omega^*$. Now fix $u \in V(G)$. 
    Since $u$ is assigned at most one color under $\omega$, $u$ belongs to $R_c(\omega)$ for at most one color $c \in L_i(v)$.
    Now fix $c' \in L_i(v)$. 
    By definition of $R_c(\omega)$, clearly $(v, c') \in R_c(\omega)$ only if $c = c'$.
    Therefore,
    $(v, c') \in R_c(\omega)$ for at most $1$ value of $c \in L_i(v)$. This completes the proof that $Z$ is $(1, 1)$-observable.

    We aim to apply Lemma~\ref{lem:conc-ineq} with $r =d=1$ and $\tau = \frac{\keep_iL_i}{\log^{2} \Delta}$. We have that $\E[Z] \leq L_i$. Additionally, $\keep_i > \exp(-K)$ by Lemma~\ref{lemma: properties}~\ref{item:keep}, so $\tau  = \omega(\sqrt{L_i})$,
    implying~\eqref{eqn:conc-tau-sqrt} is satisfied.
    Therefore, as $\tau < L_i$ and $L_i > \Delta^{\epsilon/2}$ by Lemma~\ref{lemma: properties}~\ref{item:L}, we have
    \begin{align*}
         \Pr(|Z - \E[Z]| > \tau ) &\leq 4 \exp \left ( - \frac{\tau^2}{8(2)(1) (4\E[Z] + \tau )} \right ) \\
         &\leq  4 \exp \left(- \frac{L_i}{80 \log^{20} \Delta }\right) \\
         &< \exp \left(-\log^2\Delta  \right ),
    \end{align*}
    as desired.
\end{proof}

\paragraph{Proof of Lemma~\ref{lemma: locally conditions hold}~\ref{item: T conc}: concentration of $t_{i+1}(v, c)$.}

By Property $P(i)$, $t_i(v,c) \leq T_i$.
As $T_{i+1}(v, c) \subseteq T_i(v, c)$, 
there is nothing to prove in the case that $t_i(v,c) \leq T_{i+1}$. Therefore, we may assume that $t_i(v,c) > T_{i+1} > \frac 12 T_i$, where the last inequality follows by Lemma~\ref{lem:rough-inequalities}.

We apply Corollary~\ref{cor:great-expectations} with $S = T_i(v, c)$, so that $|S| \leq T_i$ and $t_{i+1}(v,c) = |S'|$.
By Lemma~\ref{lemma: properties} items~\ref{item:T-and-Q}~and~\ref{item:Q-and-X}, we have
\[\frac{|S|}{X_i} \geq \frac{T_i}{2X_i} = \frac 12 \cdot \frac{T_i}{Q_i} \cdot \frac{Q_i}{X_i}  
 > \Delta^{(1-\gamma)\epsilon / 3}.\]
Thus, $t_{i+1}(v,c)$ is concentrated from above by Corollary~\ref{cor:great-expectations}.

\paragraph{Proof of Lemma~\ref{lemma: locally conditions hold}~\ref{item: Q conc 1}~and~\ref{item: Q conc 2}: concentration of $q_{i+1}(u,v,c)$.}
By Property $P(i)$, $q_i(u,v,c) \leq Q_i$.
As $Q_{i+1}(u, v, c) \subseteq Q_i(u, v, c)$, 
there is nothing to prove in the case that $q_i(u,v,c) \leq Q_{i+1}$. Therefore, we may assume that $q_i(u,v,c) > Q_{i+1} > \frac 12 Q_i$, where the last inequality follows by Lemma~\ref{lem:rough-inequalities}.

We apply Corollary~\ref{cor:great-expectations} with $S = Q_i(u, v, c)$, so that $|S| \leq Q_i < T_i$ by Lemma~\ref{lemma: properties}~\ref{item:T-and-Q} and
$q_{i+1}(u,v,c) = |S'|$.
By Lemma~\ref{lemma: properties}~\ref{item:Q-and-X}, we have
\[\frac{|S|}{X_i} \geq \frac{Q_i}{2X_i} > \frac 12 \Delta^{(1-\gamma)\epsilon / 3}.\]
Thus, $q_{i+1}(u,v,c)$ is concentrated from above by Corollary~\ref{cor:great-expectations}.

The proof of Lemma~\ref{lemma: locally conditions hold}~\ref{item: Q conc 2} is identical to the above \textit{mutatis mutandis} (replacing $Q_i$ with $Q_i^X$ etc.) and so we omit the details.

\paragraph{Proof of Lemma~\ref{lemma: locally conditions hold}~\ref{item: X conc}: concentration of $|X_{i+1}(c)|$.}

As $X_{i+1}(c) \subseteq X_i(c)$, there is nothing to prove when $|X_{i}(c)| \leq X_{i+1}$. Therefore, we may assume that $|X_{i}(c)| > X_{i+1} > \frac 12 X_i$, where the last inequality follows by Lemma~\ref{lem:rough-inequalities}.

We aim to apply Lemma~\ref{lem:great-expectations} 
with $S = X_i(c)$ and $R = B^X_i(v,c)$.
Note that $|S| > \frac 12 X_i = \Delta^{\Omega(1)}$ by Lemma~\ref{lemma: properties}~\ref{item:T}.
Observe that by Property $P(i)$,
for each $v \in X_i(c)$, we have
\[T_i - t_i(v,c) + |N_R(v,c)| = T_i - t_i(v,c) + b^X_i(v,c) \geq B_i.\]
Additionally, by Lemma~\ref{lemma: properties}~items~\ref{item:T} and 
\ref{item:B-and-T}, we have $B_i \gg \log^{18}\Delta$.
Therefore, we may take $\alpha = B_i$. 
For each pair $u,v \in X_i(c)$, 
if $u$ and $v$ are friends, then Property $P(i)$~\ref{item:friend-QX} implies that 
\[
|N_R(u,c) \cap N_R(v,c)| = |B^X_i(u,c) \cap B^X_i(v,c)| = |Q_i^X(u,v,c)| \leq Q_i;
\]
otherwise, if $u$ and $v$ are strangers, then Property $P(i)$ implies that 
\[
|N_R(u,c) \cap N_R(v,c)| \leq |N(u,c) \cap N(v,c)| = |Q_i(u,v,c)| \leq Q_i.
\]
As $Q_i  \leq \frac{B_i}{5\log^{18} \Delta}$
by
Lemma~\ref{lemma: properties}~\ref{item:B-and-Q}, $|N_R(u,c) \cap N_R(v,c)| < \frac{\alpha}{4 \log^{18} \Delta}$.

Letting $S'$ be defined as in Lemma~\ref{lem:great-expectations}, with probability at least $1 - \exp(-\log^3 \Delta)$, we have
\[
|S'| < (1 - \eta L_i^{-1} )^{B_i} ( 1 - \eta \keep_i) |S| + \frac{|S|}{2 \log^{2} \Delta} .
\]
Since $X_{i+1}(c) \subseteq S'$,
and since $|X_i(c)| \leq X_i$ by Property $P(i)$,
this implies that 
with probability at least $1 - \exp(-\log^3 \Delta)$,  
\begin{equation}
\label{eqn:XiUB}
|X_{i+1}(c)| \leq (1 - \eta L_i^{-1} )^{B_i} ( 1 - \eta \keep_i) X_i + \frac{X_i}{2 \log^{2} \Delta} .
\end{equation}
It remains to show that the upper bound in~\eqref{eqn:XiUB} is at most $X_{i+1}$.
Since $B_i <\gamma T_i < T_i$ by Lemma~\ref{lemma: properties}~\ref{item:B-and-T},
it follows from Lemma~\ref{lem:bernoulli} and Lemma 
\ref{lemma: properties}~\ref{item:r} that 
\[(1 - \eta L_i^{-1} )^{B_i} > 1-\eta r_i
> 1-K > \frac 45.\]
Also, by Lemma~\ref{lemma: properties}~\ref{item:keep}, 
$1 - \eta \keep_i > \frac 45$.
It follows that
\[(1 - \eta L_i^{-1} )^{B_i} ( 1 - \eta \keep_i) \geq \frac{1}{2}.\]
In particular, the upper bound in~\eqref{eqn:XiUB} is at most
\begin{align*}
 &(1 - \eta L_i^{-1} )^{B_i}  ( 1 - \eta \keep_i)\left (1 + \frac{1}{\log^2 \Delta} \right ) X_i  
\\ &\qquad = (1 - \eta L_i^{-1} )^{(\gamma-\log^{-2}\Delta) T_i} ( 1 - \eta \keep_i)\left (1 + \frac{1}{\log^2 \Delta} \right ) T_i^{\gamma} .
\end{align*}
We write $Z = 
(1-\eta/L_i)^{-T_i \log^{-2}\Delta }\left (1+\log^{-2}\Delta \right )^{1-\gamma}
(1-\eta \keep_i)^{1-\gamma}
$
and observe that the above is at most
\begin{equation}
\label{eqn:XZ}
 (1 - \eta L_i^{-1} )^{\gamma T_i} ( 1 - \eta \keep_i)^{\gamma} \left (1 + \frac{1}{\log^2 \Delta} \right )^{\gamma} T_i^{\gamma} Z = T_{i+1}^{\gamma} Z =X_{i+1} Z.
\end{equation}
Finally, we note that 
\begin{align*}
    Z &< \exp\left ( \frac{2\eta T_i}{L_i \log^2 \Delta}  
+ \frac{1-\gamma}{\log^2 \Delta} - \eta (1-\gamma) \keep_i
\right )
\\ &=\exp\left ( \frac{2\eta r_i}{ \log^2 \Delta}  
+ \frac{1-\gamma}{\log^2 \Delta} - \eta (1-\gamma) \keep_i
\right ) <1,
\end{align*}
where the last inequality holds since $r_i < \log \Delta$ by Lemma~\ref{lemma: properties}~\ref{item:r}, so $\eta (1-\gamma) \keep_i$ dominates the other terms inside the exponential for $\Delta$ sufficiently large. 
Combining~\eqref{eqn:XiUB} and~\eqref{eqn:XZ},
it follows that  $|X_{i+1}(c)| < X_{i+1}$
with probability at least $1 - \exp(-\log^3 \Delta)$.

\paragraph{Proof of Lemma~\ref{lemma: locally conditions hold}~\ref{item: B conc}: concentration of $t_{i+1}(v, c) - b^X_{i+1}(v,c)$.}

We observe that for each $u \in B_i^X(v,c) \setminus B_{i+1}^X(v,c)$,
either $u$
was colored during iteration $i$ of Algorithm~\ref{algorithm: nibble}, or $c$ was deleted from $L_{i+1}(u)$;
in both cases, $u \not \in T_{i+1}(v,c)$.
Therefore,
$T_{i+1}(v,c) \setminus B_{i+1}^X(v,c) \subseteq  T_i(v,c) \setminus B_i^X(v,c)$, implying that $t_{i+1}(v,c) - b_{i+1}^X(v,c) \leq t_i(v,c) - b_i^X(v,c)$.
As a consequence, if $t_i(v,c) - b_i^X(v,c) \leq T_{i+1} - B_{i+1}$, then there is nothing to prove;
therefore, we assume that 
$t_i(v,c) - b_i(v,c) > T_{i+1} - B_{i+1}$.

We aim to apply Corollary~\ref{cor:great-expectations} with $S = T_i(v,c) \setminus B_i^X(v,c)$, so that
$|S| < T_i$ and
 $|S'| = t_{i+1}(v,c) -b_{i+1}^X(v,c)$.
By Lemma~\ref{lem:rough-inequalities} and Lemma~\ref{lemma: properties}~items~\ref{item:B-and-T},
\ref{item:T-and-Q},  
and~\ref{item:Q-and-X}, we have
\[
\frac{|S|}{X_i} > \frac{T_{i+1} - B_{i+1}}{X_i} > \frac{(1-\gamma)T_{i+1}}{X_i} > \frac{(1-\gamma)T_i}{2X_i} =\frac{(1-\gamma)}{2} \frac{T_i}{Q_i}\cdot \frac{Q_i}{X_i}
\gg \Delta^{(1-\gamma)\epsilon/3}. \]
Thus, $t_i(v,c) - b_{i+1}^X(v,c)$ is concentrated from above by Corollary~\ref{cor:great-expectations}.

\section{Proof of Lemma~\ref{lem:great-expectations}}\label{sec:great-expectations}

We restate the lemma for the reader's convenience.

\begin{lemma*}[Restatement of Lemma~\ref{lem:great-expectations}]
    Let $R, S \subseteq V(G)$ and $i \in [i^*-1]$ be such that $P(i)$ holds.
    Let $\alpha \geq 4 \log^{18} \Delta$ be an integer.
    Fix a color $c$, and for each $v \in S$,
    write $N_R(v,c) = T_i(v,c) \cap R$.
    Suppose that the following hold:
    \begin{enumerate}[label=(C\arabic*)]
        \item For each $v \in S$, $c \in L_i(v)$;
        \item $\log^{16} \Delta \leq |S| \leq T_i$;
        \item $T_i -t_i(v,c) + |N_R(v,c)| \geq \alpha$ for each $v \in S$;
        For each pair $u,v \in S$, $|N_R(u,c) \cap N_R(v,c)  | \leq \frac{\alpha}{4 \log^{18} \Delta}$.
    \end{enumerate}
    Let $S' \subseteq S$ be the set 
    of vertices $v \in S$ for which the following hold throughout iteration $i$ of Algorithm~\ref{algorithm: nibble}:
    \begin{itemize}
        \item $v$ is not colored during Step~\ref{step: color};
        \item no $u \in N_R(v,c)$ is activated and assigned $c$; and 
        \item $(v,c)$ survives its equalizing coin flip.
    \end{itemize}
    Then,
    \[
    \Pr \left (|S'| \geq \left ( 1 - \eta L_i^{-1} \right)^{\alpha}(1 - \eta \keep_i)|S| + \frac{|S|}{2 \log^{2} \Delta} \right  ) < \exp(-\log^3 \Delta).
    \]
\end{lemma*}

In our calculations, we use the fact that due to Lemma~\ref{lemma: properties}~\ref{item:L}
and the fact that $\Delta$ is sufficiently large,
$L_i$ is also sufficiently large.

For the rest of this section, we assume that $P(i)$ holds.
Rather than working with $S'$ directly, we define a new set $S''$ containing $S'$ as follows.
We let $R_{high} \subseteq R$ denote the set of vertices $r \in R$ with $c \in L_i(r)$ and at least $h \coloneqq \frac{|S|}{\log^{8} \Delta}$ neighbors in $S$.
Let $S'' \subseteq S$ be the set of vertices $v$ for which the following hold after Algorithm~\ref{algorithm: nibble}:
\begin{itemize}
    \item $v$ is not colored during Step~\ref{step: color};
    \item no $u \in N_R(v,c) \setminus R_{high}$ is activated and assigned $c$; and
    \item $(v,c)$ survives its equalizing coin flip in Step~\ref{step:coin_flip}.
\end{itemize}
Clearly $S' \subseteq S''$ and so $\Pr(S' \geq t) \leq \Pr(S'' \geq t)$.
In particular, it suffices to show that
\begin{align}\label{eq: S''}
    \Pr \left (|S''| \geq \left ( 1 - \eta L_i^{-1} \right)^{\alpha}(1 - \eta \keep_i)|S| + \frac{|S|}{2 \log^{2} \Delta} \right  ) < \exp(-\log^3 \Delta),
\end{align}
which is the focus of the remainder of this section.
Our proof follows a sequence of lemmas.
First, we show that $R_{high}$ does not contain too many vertices.

\begin{lemma}\label{lemma: R high not large}
    $|R_{high}| \leq \dfrac{\alpha}{\log^{2} \Delta}$.
\end{lemma}

\begin{proof}
    Define an auxiliary bipartite graph $H$
    with partite sets $A$ and $B$
    as follows.
    For each vertex $r \in R_{high}$, add a vertex $a_r$ to $A$.
    For each vertex
    $s \in S$, add an edge $b_s$ to $B$. 
    If $r \in R_{high}$ and $s \in S$ have the relationship $r \in T_i(s,c)$,
    add the edge $a_r b_s$ to $H$.
    Since each vertex $r \in R_{high}$ has at least $h$ neighbors in $S$, each corresponding $a_r \in A$ has at least $h$ neighbors in $B$.
    Therefore $|E(H)| \geq |A|h = |R_{high}|h$.
    Also, by condition~\ref{item: codegree bound great expectations} in Lemma~\ref{lem:great-expectations}, for each pair $b,b' \in B$,
    $b$ and $b'$ have at most $\frac{\alpha}{4 \log^{18} \Delta}$ common neighbors in $A$.
    Thus, by applying Lemma~\ref{lem:KST} with $s = 2$, $t = \frac{\alpha}{4 \log^{18} \Delta}$, $n = |B| = |S|$, and $m = |A| = |R_{high}|$, we have
    \[|R_{high}| h \leq |E(H)| < \sqrt t |S| |R_{high}|^{1/2} + |R_{high}|.\]
    If $|R_{high}| > \sqrt{t} |S| |R_{high}|$, then $\frac{|S|}{\log^8 \Delta} = h < 2$, a contradiction since $|S| \geq \log^{16} \Delta$. Therefore,
    \[
    |R_{high}|h < 2  \sqrt t |S| |R_{high}|^{1/2} \implies
    |R_{high}| < 4t (|S|^2/h^2) = 4t \log^{16}\Delta = \frac{\alpha}{\log^{2} \Delta},
    \]
    as desired.
\end{proof}

We write $\alpha' = \alpha\left ( 1  - \frac{1}{\log^{2} \Delta } \right )$, and observe that for each $v \in S$, we have
\begin{align}\label{eq: alpha'}
    T_i - t_i(v,c) + |N_R(v,c) \setminus R_{high}| \geq \alpha'
\end{align}
as a result of Lemma~\ref{lemma: R high not large} and condition~\ref{item: bound on alpha} of Lemma~\ref{lem:great-expectations}.
Next, we compute an upper bound on the expected value of $|S''|$.

\begin{lemma}\label{lemma: exp S''}
    $\E[|S''|] \leq |S| (1 - \eta L_i^{-1} )^{\alpha} (1 - \eta\keep_i) + \dfrac {|S|}{4 \log^{2} \Delta}$.
\end{lemma}

\begin{proof}
    For each $v \in S$, we estimate the probability that $v \in S''$. To this end, we fix $v \in S$. We condition on two cases based on the activation status of $v$ during Step~\ref{step:activate} of Algorithm~\ref{algorithm: nibble}.

    First, we condition on the event that $v$ is not activated. In this case, $v \in S''$ if and only if the following events occur:
    \begin{itemize}
        \item $\mathcal F_1$: for each vertex $u \in N_R(v,c) \setminus R_{high}$, $c$ is not assigned to $u$, and 
        \item $\mathcal F_2$: $(v, c)$ survives its equalizing coin flip. 
    \end{itemize}
    We observe that events $\mathcal F_1$ and $\mathcal F_2$ are independent.
    We have
    \[
        \Pr(\mathcal F_1) =
        \left ( 1 - \eta L_i^{-1} \right)^{ |N_R(v,c) \setminus R_{high}|} 
    \]
    and 
    \[\Pr(\mathcal F_2) = \Eq_i(v,c) = \left(1 - \eta L_i^{-1}\right)^{T_i - t_i(v,c)}.\]
    Hence, the conditional probability that $v \in S''$ is 
    \begin{equation*}
        \Pr(\mathcal F_1) \Pr(\mathcal F_2) =  \left(1 -  \eta L_i^{-1}\right)^{|N_R(v,c) \setminus R_{high}| + T_i - t_i(v,c) } .
    \end{equation*}
    By~\eqref{eq: alpha'}, we have
    \begin{equation}
        \label{eqn:Y,T_i}
         \left(1 -  \eta L_i^{-1}\right)^{|N_R(v,c) \setminus R_{high}| + T_i - t_i(v,c) } \leq \left(1 -  \eta L_i^{-1}\right)^{\alpha'}.
    \end{equation}
    Therefore, writing $Y = \left(1 -  \eta L_i^{-1}\right)^{\alpha'}$, we have
     \begin{equation}
    \label{eq:1st-case-UB}
        \Pr(\mathcal F_1)  \Pr(\mathcal F_2) \leq Y.
    \end{equation}

    Next, we condition on the event that $v$ is activated. In this case, $v \in S''$ 
    if and
    only if the following events occur during Algorithm~\ref{algorithm: nibble}:
    \begin{itemize}
        \item 
        $\mathcal E_1$: $v$ is assigned some color $d \in L_i(v)$;
        \item 
        $\mathcal E_2$: a vertex $w \in T_i(v, d)$ is activated and assigned the same color $d$;
        \item 
        $\mathcal E_3$: no vertex $u \in N_R(v,c) \setminus R_{high}$ is assigned color $c$;
        \item 
        $\mathcal E_4$: $(v,c)$ survives its equalizing coin flip. 
    \end{itemize}
    We aim to estimate $\Pr(\mathcal E_1 \cap \mathcal E_2 \cap \mathcal E_3 \cap \mathcal E_4)$.
    To this end, we partition $\mathcal E_1$ into the disjoint event set $\{\mathcal A_d : d \in L_i(v)\}$, where $\mathcal A_d$ is the event that $v$ is assigned the color $d$.
    From this partition, we simplify the probability to
    \begin{equation*}
        \Pr(\mathcal E_1 \cap \mathcal E_2 \cap \mathcal E_3 \cap \mathcal E_4) = \sum_{d \in L_i(v)} \Pr(\mathcal A_d \cap  \mathcal E_2 \cap \mathcal E_3 \cap \mathcal E_4).
    \end{equation*}
    Fix a color $d \in L_i(v)$.
    We aim to estimate 
    \begin{equation}
    \label{eqn:d234-original}
        \Pr(\mathcal A_d \cap  \mathcal E_2 \cap \mathcal E_3 \cap \mathcal E_4)= \Pr(\mathcal A_d) \Pr(\mathcal E_2 \cap \mathcal E_3 \mid \mathcal A_d )  \Pr(\mathcal E_4),
    \end{equation}
    where equality holds from the fact that $\mathcal E_4$ is independent of $\mathcal A_d$, $\mathcal E_2$, and $\mathcal E_3$.    
    Rather than estimating each of the factors in~\eqref{eqn:d234-original}, we take advantage of the fact that 
    \begin{eqnarray*}
    \Pr(\mathcal E_3) = \Pr(\mathcal E_3 \mid \mathcal A_d) &=& \Pr(\mathcal E_2 \cap \mathcal E_3 \mid \mathcal A_d ) + \Pr(\overline{\mathcal E_2} \cap \mathcal E_3  \mid \mathcal A_d) \\
     &=& \Pr(\mathcal E_2 \cap \mathcal E_3 \mid \mathcal A_d ) +  \Pr(\overline{\mathcal E_2 } \mid \mathcal A_d) \Pr(\mathcal E_3 \mid \mathcal A_d \cap \overline{\mathcal E_2}).
    \end{eqnarray*}
    Therefore, 
    \[\Pr(\mathcal E_2 \cap \mathcal E_3 \mid \mathcal A_d ) = \Pr(\mathcal E_3) - \Pr(\mathcal E_3 \mid \mathcal A_d \cap \overline{\mathcal E_2}) \Pr(\overline{\mathcal E_2 } \mid \mathcal A_d),\]
    and hence
    \begin{equation}
    \label{eq:rearranged-A234}
        \Pr(\mathcal A_d \cap  \mathcal E_2 \cap \mathcal E_3 \cap \mathcal E_4) = \Pr(\mathcal A_d) (\Pr(\mathcal E_3) - \Pr(\mathcal E_3 \mid \mathcal A_d \cap \overline{\mathcal E_2}) \Pr(\overline{\mathcal E_2 } \mid \mathcal A_d) )
        \Pr(\mathcal E_4).
    \end{equation}

    Now, we estimate the probabilities in ~\eqref{eq:rearranged-A234}.
    First, we have
    \begin{equation}
    \label{eq:Ad}
        \Pr(\mathcal A_d) = \frac{1}{|L_i(v)|} = \frac{1}{L_i}
    \end{equation}
    and
    \begin{equation}
    \label{eq:E3}
        \Pr(\mathcal E_3) = (1 - \eta L_i^{-1})^{|N_R(v,c) \setminus R_{high}|}.
    \end{equation}
    Next, we estimate $\Pr(\mathcal E_3 \mid \mathcal A_d \cap \overline{\mathcal E_2})$.
    We recall that $\overline{\mathcal E_2}$ is the event that ``no $w \in T_i(v,d)$ is assigned $d$.''
    We consider an uncolored neighbor $u \in N_R(v)$.
    Let $B_c$ be the event that $u$ is activated and assigned $c$, and let $B_d$ be the event that $u$ is activated and assigned $d$.
    Thus,
    \begin{eqnarray}
    \notag
        \Pr(B_c \mid \mathcal A_d \cap \overline{\mathcal E_2}) &=& \Pr(B_c \mid \mathcal A_d \cap \overline{B_d}) = \frac{\Pr(B_c \cap \overline{B_d} \mid \mathcal A_d) }{  \Pr( \overline{B_d}\mid \mathcal A_d)}  \\
        &\leq& 
        \label{eqn:Bc}
        \frac{\Pr(B_c \mid \mathcal A_d) }{  \Pr( \overline{B_d}\mid \mathcal A_d)} 
        = \frac{\eta L_i^{-1} }{1 - \eta L_i^{-1}} = \frac{\eta}{L_i - \eta}.
    \end{eqnarray}
    Therefore, by~\eqref{eqn:Bc}, we have
    \begin{equation}
    \label{eq:E3-LB}
        \Pr(\mathcal E_3 \mid \mathcal A_d \cap \overline{\mathcal E_2}) \geq \left ( 1 - \frac{\eta}{L_i - \eta} \right )^{ |N_R(v,c) \setminus R_{high}|}.
    \end{equation}

    Next,
    \begin{equation}
    \label{eq:E2-LB}
        \Pr(\overline{\mathcal E_2} \mid \mathcal A_d) = (1 - \eta L_i^{-1} )^{t_i(v,d)}  \geq \keep_i.
    \end{equation}

    Finally, 
    \begin{equation}
    \label{eq:E4}
         \Pr(\mathcal E_4) = \Eq_i(v,c) = \left(1 - \eta L_i^{-1}\right)^{T_i - t_i(v,c)}
    \end{equation}

        Combining~\eqref{eq:rearranged-A234},~\eqref{eq:Ad},
       ~\eqref{eq:E3},~\eqref{eq:E3-LB},~\eqref{eq:E2-LB}, and~\eqref{eq:E4}, $\Pr(\mathcal A_d \cap  \mathcal E_2 \cap \mathcal E_3 \cap \mathcal E_4)$ is at most
        \begin{equation}
        \label{eq:bad-UB}
         \frac{1}{L_i} \left ( (1 - \eta L_i^{-1} )^{ |N_R(v,c) \setminus R_{high}|}  -  \left ( 1 - \frac{\eta}{L_i - \eta} \right )^{ |N_R(v,c) \setminus R_{high}|} \keep_i \right ) \left(1 - \eta L_i^{-1}\right)^{T_i - t_i(v,c)}.
        \end{equation}
        Let $M \coloneqq 1 - \frac{\eta^2}{(L_i - \eta)^2}$, so that 
        $1 - \frac{\eta}{L_i - \eta} = M(1 - \eta L_i^{-1})$.
        Using the fact that $\frac{T_i}{L_i} = r_i \leq r_1$  and $4\eta^2 r_1 < 4 \eta^2 \log \Delta = o(1)$ by Lemma~\ref{lemma: properties}~\ref{item:r}, we have
        \[
            M^{ |N_R(v,c) \setminus R_{high}|} \geq M^{T_i} > \exp \left ( - \frac{2 T_i \eta^2 }{(L_i - \eta)^2}  \right ) > \exp \left ( - \frac{4\eta^2  r_1 }{L_i}  \right )  > 1 - \frac{1}{L_i}.
        \]
        With the above in hand, continuing from~\eqref{eq:bad-UB} and writing $N' = |N_R(v,c) \setminus R_{high}|$, we have
        \begin{align*}
        \Pr(\mathcal A_d \cap  \mathcal E_2 \cap \mathcal E_3 \cap \mathcal E_4) &\leq
        \frac{1}{L_i} (1 - \eta L_i^{-1} )^{ N'} \left ( 1 -  
        M^{N'} 
        \keep_i
        \right ) \left(1 - \eta L_i^{-1}\right)^{T_i - t_i(v,c)} \\
        &\leq \frac{1}{L_i} \left ( 1  -  
        \left (1 - \frac{1}{L_i} \right )  \keep_i  \right ) \left(1 - \eta L_i^{-1}\right)^{ N' + T_i - t_i(u, c)} .
        \end{align*}
        As
        $(1 - \eta L_i^{-1})^{ N' + T_i - t_i(v,c)} \leq (1 - \eta L_i^{-1})^{\alpha'} = Y$ by~\eqref{eqn:Y,T_i}, we have
        \[ 
       \Pr(\mathcal A_d \cap  \mathcal E_2 \cap \mathcal E_3 \cap \mathcal E_4)  \leq    \frac{1}{L_i} \left ( 1  -  
        \left ( 1 - \frac{1}{L_i} \right ) \keep_i  \right ) Y <  \frac{1}{L_i} \left ( 1 - \keep_i + L_i^{-1} \right ) Y .
        \]
        Summing the upper bound above over all $d \in L_i(v)$, we conclude
        \begin{equation*}
        \Pr(\mathcal E_1 \cap  \mathcal E_2 \cap \mathcal E_3 \cap \mathcal E_4 ) \leq  \left ( 1 - \keep_i + L_i^{-1} \right ) Y.
        \end{equation*}
        Combining the above with~\eqref{eq:1st-case-UB} and using linearity of expectation, we have
        \begin{align}
            \E[|S''|] &\leq |S|\left((1-\eta)\,Y + \eta\left ( 1 - \keep_i + L_i^{-1} \right ) Y\right) \nonumber \\
            &\leq |S|Y(1 - \eta\keep_i + \eta L_i^{-1}) \nonumber \\
            &\leq |S| (1 - \eta L_i^{-1} )^{\alpha'} (1 - \eta\keep_i) + |S| L_i^{-1}, \label{eqn:ES' pre}
        \end{align}
        where we use the fact that $Y, \eta \leq 1$.
        Recall that by Lemma~\ref{lemma: properties}~\ref{item:L}, $L_i > \Delta^{\epsilon/2} > 8 \log^{2} \Delta$.
        Furthermore, by definition of $\alpha'$, we have
        \[(1 - \eta L_i^{-1} )^{\alpha'} = (1 - \eta L_i^{-1} )^{\alpha - \frac{\alpha}{\log^{2} \Delta}} \leq (1 - \eta L_i^{-1})^{\alpha} \left (1 + \frac{2 \eta \alpha}{L_i \log^{2} \Delta} \right ).\]
        Therefore, we may further simplify~\eqref{eqn:ES' pre} to
        \begin{align*}
            \E[|S''|] &\leq |S| (1 - \eta L_i^{-1} )^{\alpha} (1 - \eta\keep_i)\left (1 + \frac{2 \eta \alpha}{L_i \log^{2} \Delta} \right ) + \frac{|S|}{8\log^2\Delta} \\
            &\leq |S| (1 - \eta L_i^{-1} )^{\alpha} (1 - \eta\keep_i) + \frac{2 \eta \alpha |S|}{L_i \log^{2} \Delta}+ \frac{|S|}{8\log^2\Delta} \\
            &\leq |S| (1 - \eta L_i^{-1} )^{\alpha} (1 - \eta\keep_i) + \frac{|S|}{4\log^2\Delta},
        \end{align*}
        where the last step follows since
        \[\frac{2 \eta \alpha}{L_i} \leq 
        \frac{2 \eta T_i}{L_i} = \frac{2Kr_i}{\log \Delta} < \frac{1}{8}\]
        by Lemma~\ref{lemma: properties}~\ref{item:r}, condition~\ref{item: bounds on |S|} of Lemma~\ref{lem:great-expectations}, and by definition of $K$.
        This completes the proof of Lemma~\ref{lemma: exp S''}.
\end{proof}

It remains to concentrate $|S''|$ from above, i.e., prove~\eqref{eq: S''}. As mentioned earlier, this implies the conclusion of Lemma~\ref{lem:great-expectations}.
We observe that $|S''| = Z_1 - Z_2$,
where $Z_1$ and $Z_2$ are defined as follows: 
\begin{itemize}
    \item $Z_1$ is the number of vertices $v \in S$ that are not colored during Step~\ref{step: color}, and
    \item $Z_2$ is the number of vertices $v \in S$ 
    that are not colored during Step~\ref{step: color} and
    for which one of the following occurs:
    \begin{itemize}
        \item some $u \in N_R(v,c) \setminus R_{high}$ is activated and assigned $c$, or
        \item $(v,c)$ does not survive its coin flip.
    \end{itemize}
\end{itemize}
 
We show that $Z_1$ and $Z_2$ are concentrated about their means, which implies that $|S''|$ is concentrated from above.
To this end, recall the probability space $(\Omega, \Sigma, \Pr)$ defined on page~\pageref{prob_space}.
In order to concentrate $|S''|$, we aim to apply Lemma~\ref{lem:conc-ineq}.
To this end, we let $\Omega^* \subseteq \Omega$ be the set of exceptional outcomes
in which for some color  $c' \in \bigcup_{v \in S} L_i(v)$, 
more than $ \log^4 \Delta - 1$ vertices in $S$ are activated and assigned $c'$.
We begin by showing $\Pr(\Omega^*)$ is small.
    
\begin{lemma}\label{lemma: omega*}
    $\Pr(\Omega^*) \leq \exp(-\frac 12 \log^4 \Delta)$.
\end{lemma}
\begin{proof}
    We fix $c' \in \bigcup_{v \in S} L_i(v)$, and we bound the probability that 
    more than $ \log^4 \Delta  - 1$ vertices $v \in S$
    are activated and assigned $c'$.
    By Lemma~\ref{lemma: properties}~\ref{item:r}, we have $L_i > \frac 18 T_i$.
    The probability that a given vertex
    $v \in S$ is activated and assigned $c'$ is at most 
    $\eta L_i^{-1} < \frac{8K}{T_i  \log \Delta }$. Since $|S| \leq T_i$,
    the expected number of vertices
    in $S$ activated and assigned $c'$ is less than $\frac{8K}{\log \Delta} < \frac{1}{2 \log \Delta}$.
    Using the Chernoff bound (Lemma~\ref{lem:chernoff}) with $\mu = \frac{1}{2 \log \Delta}$ and $\delta = (2-\epsilon) \log^5 \Delta $,
    the probability that the number of vertices
    activated with color $c'$ is more than $\lfloor \log^4 \Delta \rfloor - 1$ is bounded above by $\exp\left (- \frac 23 \log^4 \Delta\right )$. 
    Taking a union bound over all colors $c' \in \bigcup_{v \in S} L(v)$, we have
    \[\Pr(\Omega^*) < |S|k \exp\left(- \frac 23 \log^4 \Delta\right) < \exp \left (-\frac 12 \log^4 \Delta \right ),\]
    as desired.
\end{proof}

With the above in hand, let us first concentrate $Z_1$.

\begin{lemma}\label{lemma: Z_1 conc}
    $\Pr\left (|Z_1 - \E[Z_1]| > \dfrac{|S|}{8 \log^{2} \Delta} \right  ) \leq \dfrac 12 \exp\left(- \log^3 \Delta\right)$.
\end{lemma}

\begin{proof}\stepcounter{ForClaims} \renewcommand{\theForClaims}{\ref{lemma: Z_1 conc}}
    Recall that $Z_1$ denotes the number of vertices in $S$ that remain uncolored after Algorithm~\ref{algorithm: nibble}.
    We may write $Z_1$ as a sum of binary random variables $Y_v$ for $v \in S$, where $Y_v = 1$ if and only if $v$ is uncolored after Algorithm~\ref{algorithm: nibble}.
    We will concentrate $Z_1$ through a sequence of claims.

    \begin{claim}
        $Y_v$ is $2$-verifiable.
    \end{claim}
    \begin{claimproof}
        Consider an event $Y_v = 1$.
        We define a verifier $R_v$ for each outcome $\omega \in \Omega \setminus \Omega^*$ satisfying $Y_v(\omega) = 1$ as follows:
        \begin{equation}
        \label{eqn:Re}
        R_v(\omega) = \begin{cases}
           \{v \}  & \text{if $\Omega_v = 0$, i.e.~$v$ is not activated;}   \\
          \{v,u \} & \text{if } \omega_v = \omega_u = c' \text{ for some $u \in N(v)$ and color $c'$.}
        \end{cases}
        \end{equation}
        We also let 
        $R_v(\omega) = \emptyset$ whenever $Y_v(\omega) = 0$.
        In the second case, when $v$ and a neighbor $u \in N(v)$ are both activated and assigned $c'$, $u$ is chosen arbitrarily from all vertices $u' \in N(v)$ for which $\omega_{u'} = \omega_v$.
        Since the fact that $v$ is uncolored after Algorithm~\ref{algorithm: nibble} is certified by either (1) the non-activation of $v$, or (2) the assignment of a common color $c'$ to $v$ and some neighbor $u \in N(v)$, 
        the function $R_v(\cdot)$ satisfies the conditions of Definition~\ref{defn: r verifiable}.
    \end{claimproof}

    \begin{claim}\label{claim: Z_1 obs}
        $Z_1$ is $(2,\log^4 \Delta)$-observable with respect to the verifier function $R_v$.
    \end{claim}
    \begin{claimproof}
    
    We fix an outcome $\omega \in \Omega \setminus \Omega^*$.
    First, we observe that each vertex-color pair $(v,c')$ belongs to no set $R_u(\omega)$ for $u \neq v$.
    It follows that each vertex-color pair belongs to at most one witness set $R_v(\omega)$.
    Therefore, to prove $(2,\log^4 \Delta)$-observability,
    it suffices to show that the following holds:
    for each $u \in V(G)$, $u \in R_v(\omega)$ for at most $\log^4 \Delta$ vertices $v \in S$.

    To this end, consider
    an arbitrary vertex $u \in V(G)$. 
    If $u$ is not activated, then $u$ cannot certify that $Y_v = 1$ for any $v \neq u$. Therefore, 
    in this case,
    $u \in R_v(\omega)$ only if $u  = v$.
    If $u$ is activated and assigned some color $c'$, then $u$ 
    may belong to $R_u(\omega)$.
    Furthermore, for each $v \in S \setminus \{u\}$, $u$ can certify that $Y_v = 1$ only if $v$ is activated and assigned $c'$. 
    Therefore, $u \in R_v(\omega)$ only if $u = v$ or $\omega_v = c'$.
    As $\omega \in \Omega \setminus \Omega^*$, the number of vertices $v \in S$ that are activated and assigned $c'$ is at most $ \log^4 \Delta - 1$; therefore, $u \in R_v(\omega)$ for at most $\log^4 \Delta $ vertices $v \in S$, completing the proof.
    \end{claimproof}

    With the above claims in hand, we aim to apply Lemma~\ref{lem:conc-ineq} with $\tau =  \frac{|S|}{8 \log^{2} \Delta}$, $r = 2$, and $d = \log^4 \Delta$.
    To this end, we check that~\eqref{eqn:conc-tau-sqrt} holds.
    We observe that 
    $|S| \gg d$, and hence $\sqrt{|S| d} \geq d$.
    Thus, since $Z_1 \leq |S| \leq \Delta$,
    we have 
    \begin{eqnarray}
    \notag
        96 \sqrt{rd \E[Z_1]} + 128 rd + 8 \Pr(\Omega^*)\sup(Z_1) &=& O\left(\sqrt{|S|d}\right) + 8 \exp \left (- \frac 12 \log^4 \Delta\right ) |S|  \\
        \label{eqn:tau-sqrt}
        &=& O\left(\sqrt{|S| d }\right).
    \end{eqnarray}
    By condition~\ref{item: bounds on |S|} of Lemma~\ref{lem:great-expectations}, $|S| \geq \log^{16} \Delta$
    and so it follows that
    $\tau \gg \sqrt{|S| d}$.
    In particular,~\eqref{eqn:conc-tau-sqrt} holds, and we may apply Lemma~\ref{lem:conc-ineq}.
    
    As $Z_1 \leq |S|$, we have $8rd(4\E[Z_1] + \tau) < |S| \log^8 \Delta$. Therefore,
    \begin{align*}
        \Pr \left (|Z_1 - \E[Z_1]| > \frac{|S|}{8 \log^{2} \Delta} \right ) &\leq 4 \exp \left (- \frac{ |S|^{2} }{64 |S| \log^{12} \Delta} \right ) + 4 \exp \left (-\frac 12 \log^4 \Delta \right ) \\
        &\leq \frac 12 \exp\left(- \log^3 \Delta\right),
    \end{align*}
    completing the proof.
\end{proof}

We employ a similar strategy to concentrate $Z_2$.

\begin{lemma}\label{lemma: Z_2 conc}
    $\Pr\left (|Z_2 - \E[Z_2]| > \dfrac{|S|}{8 \log^{2} \Delta} \right  ) < \dfrac 12 \exp\left(- \log^3 \Delta\right)$.
\end{lemma}

\begin{proof}\stepcounter{ForClaims} \renewcommand{\theForClaims}{\ref{lemma: Z_2 conc}}
    Recall that $Z_2$ is the number of vertices $v \in S$ that remain uncolored after Algorithm~\ref{algorithm: nibble},
    and for which one of the following occurs:
    \begin{itemize}
        \item some $u \in N_R(v,c) \setminus R_{high}$ is activated and assigned the color $c$, or
        \item $(v,c)$ does not survive its coin flip.
    \end{itemize}
    The variable $Z_2$ is a sum of binary random variables $Y_v$ for $v \in S$, where $Y_v = 1$ if and only if $v$ remains uncolored after Algorithm~\ref{algorithm: nibble} and one of the two events above occurs.
    We will concentrate $Z_2$ through a sequence of claims.

    \begin{claim}
        $Y_v$ is $3$-verifiable.
    \end{claim}
    \begin{claimproof}
        For each $v \in S$ and $\omega \in \Omega \setminus \Omega^*$ satisfying $Y_v(\omega) = 1$, we define
        $R_v(\omega)$ as in~\eqref{eqn:Re}.
        We also define, when $Y_v(\omega) = 1$,
        \[
        R'_v(\omega) = \begin{cases}
            \{(v,c)\} & \omega_{(v,c)} = 0 \\
            \{u\} & \omega_{(v,c)} = 1 \text{ and $\omega_u = c$ for some $u \in N_R(v,c) \setminus R_{high}$}.
        \end{cases}
        \]
        We let $R'_v(\omega) = \emptyset$ when $Y_v(\omega) = 0$.
        Finally, we define $R_v^*(\omega) = R_v(\omega) \cup R'_v(\omega)$.
        We note that whenever $Y_v(\omega) = 1$,
        the trial or trials specified by $R_v(\omega)$ certify that $v$ is uncolored, and the trial specified by $R'_v(\omega)$ certifies either (1) some $w \in N_R(v,c) \setminus R_{high}$ is activated and assigned $c$ during Algorithm~\ref{algorithm: nibble}, or (2) $(v,c)$ does not survive the equalizing coin flip during Algorithm~\ref{algorithm: nibble}. Therefore, the verifier function $R_v^*$ satisfies the conditions of Definition~\ref{defn: r verifiable}, and $Y_v$ is $3$-verifiable.
    \end{claimproof}
        
    \begin{claim}
        $Z_2$ is $(3,d)$-observable for $d = \frac{2 |S|}{\log^{8} \Delta}$ with respect to the verifier function $R_v^*$.
    \end{claim}

    \begin{claimproof}
        Fix an outcome $\omega \in \Omega \setminus \Omega^*$, 
        and count the number of sets $R^*_v(\omega)$ to which a given element may belong.
        As argued in the proof of Claim~\ref{claim: Z_1 obs}, each vertex $u \in V(G)$ and vertex-color pair $(u,c')$
        belongs to at most $\log^4 \Delta$ sets $R_v(\omega)$.
        Furthermore, clearly a vertex-color pair $(u,c')$ belongs to $R_v'(\omega)$ only if $u=v$.
        Finally, a vertex $u \in V(G)$ belongs to $R'_v(\omega)$ only if $u \in N_R(v,c) \setminus R_{high}$.
        As each $u \in N_R(v,c) \setminus R_{high}$ has at most $h = \frac{|S|}{\log^{8} \Delta} $ neighbors in $S$,
        each vertex and each vertex-color pair belongs to 
        at most 
        $\frac{|S|}{\log^{8} \Delta}$ sets $R_v'(\omega)$, and hence at most $\frac{|S|}{\log^{8} \Delta} + \log^4 \Delta < \frac{2 |S|}{\log^{8} \Delta}$ sets $R_v^*(\omega)$.
        Hence, $Z_2$ is $(3,d)$-observable.
    \end{claimproof}

    With the above claims in hand, we aim to apply Lemma~\ref{lem:conc-ineq} with $\tau =  \frac{|S|}{8 \log^{2} \Delta}$, $r = 3$, and with $d = \frac{2 |S|}{\log^{8} \Delta}$.
    As $|S| \gg \sqrt{|S| d}$, an identical argument as in~\eqref{eqn:tau-sqrt} shows that
    $\tau$ satisfies~\eqref{eqn:conc-tau-sqrt}.
    As $|Z_2| \leq |S|$ and $\tau < |S|$, we have $8rd(4 \E[|Z_2|] + \tau) < \frac{240 |S|^2}{\log^{8} \Delta}$.
    Therefore, 
    \begin{eqnarray}
    \notag
    \Pr \left (|Z_1 - \E[Z_1]| >  \frac{|S|}{8 \log^{2} \Delta} \right )  & \leq &
    4 \exp \left ( - \frac{|S|^2 \log^{8} \Delta } {64 \cdot 240 \log^{4} \Delta  |S|^2 } \right ) + 4\Pr(\Omega^*) \\
    \notag 
    & < & 4 \exp \left ( - \frac{\log^{4} \Delta }{15360} \right ) + 4 \exp\left (- \frac 12 \log^4 \Delta \right ) \\
    \notag
    &< & \frac 12 \exp(-\log^3 \Delta),
    \end{eqnarray}
    completing the proof.
\end{proof}

We are ready to finish the proof of~\eqref{eq: S''}, which implies Lemma~\ref{lem:great-expectations}.
By the triangle inequality, as $|S''|= Z_1 - Z_2$, we have
\[||S''| - \E[|S''| ] | = |Z_1 - Z_2 - \E[Z_1] + \E[Z_2]   | \leq |Z_1 - \E[Z_1]| + |Z_2 - \E[Z_2]|. \]
Therefore, by Lemmas~\ref{lemma: Z_1 conc}~and~\ref{lemma: Z_2 conc}, and writing $\tau =  \frac{|S|}{8 \log^{2} \Delta}$, we have
\begin{eqnarray*}
    \Pr\left (|S''| > \E[|S''|] + \frac{|S|}{4 \log^{2} \Delta}  \right ) &\leq& \Pr \left (Z_1 > \E[Z_1] + \tau \right ) + \Pr \left ( Z_2 < \E[Z_2] - \tau \right ) \\
    &\leq& \Pr \left (|Z_1 - \E[Z_1]|  >  \tau \right ) + \Pr \left ( |Z_2 - \E[Z_2]| > \tau \right ) \\
    & \leq  & \exp\left(-\log^3 \Delta\right).
\end{eqnarray*}
By the above and Lemma~\ref{lemma: exp S''}, we have
\begin{align*}
    \Pr \left ( |S''| > |S| (1 - \eta L_i^{-1} )^{\alpha} (1 - \keep_i) + \frac{|S|}{2 \log^{2} \Delta}   \right) &\leq  \Pr\left (|S''| > \E[|S''|] + \frac{|S|}{4 \log^{2} \Delta}  \right ) \\
    & \leq \exp\left(-\log^3 \Delta\right),
\end{align*}
completing the proof.

\subsection*{Acknowledgements}

We thank Daniel Cranston for helpful comments on an earlier draft of this manuscript.

\printbibliography

@book {MolloyReed,
    AUTHOR = {Molloy, Michael and Reed, Bruce},
     TITLE = {Graph colouring and the probabilistic method},
    SERIES = {Algorithms and Combinatorics},
    VOLUME = {23},
 PUBLISHER = {Springer-Verlag, Berlin},
      YEAR = {2002},
     PAGES = {xiv+326},
      ISBN = {3-540-42139-4},
   MRCLASS = {05-02 (05C15 05C80 60-02 60C05)},
  MRNUMBER = {1869439},
MRREVIEWER = {P.\ Mark\ Kayll},
       DOI = {10.1007/978-3-642-04016-0},
       URL = {https://doi.org/10.1007/978-3-642-04016-0},
}

@inproceedings {Mahdian,
    AUTHOR = {Mahdian, Mohammad},
     TITLE = {The strong chromatic index of {$C_4$}-free graphs},
 BOOKTITLE = {Proceedings of the {N}inth {I}nternational {C}onference
              ``{R}andom {S}tructures and {A}lgorithms'' ({P}oznan, 1999)},
   JOURNAL = {Random Structures Algorithms},
  FJOURNAL = {Random Structures \& Algorithms},
    VOLUME = {17},
      YEAR = {2000},
    NUMBER = {3-4},
     PAGES = {357--375},
      ISSN = {1042-9832,1098-2418},
   MRCLASS = {05C15 (05C80)},
  MRNUMBER = {1801139},
MRREVIEWER = {R.\ H.\ Schelp},
       DOI = {10.1002/1098-2418(200010/12)17:3/4<357::AID-RSA9>3.0.CO;2-Y},
       URL =
              {https://doi.org/10.1002/1098-2418(200010/12)17:3/4<357::AID-RSA9>3.0.CO;2-Y},
}

@article {Bernshteyn,
    AUTHOR = {Bernshteyn, Anton},
     TITLE = {The {J}ohansson-{M}olloy theorem for {DP}-coloring},
   JOURNAL = {Random Structures Algorithms},
  FJOURNAL = {Random Structures \& Algorithms},
    VOLUME = {54},
      YEAR = {2019},
    NUMBER = {4},
     PAGES = {653--664},
      ISSN = {1042-9832,1098-2418},
   MRCLASS = {05C15 (05D40)},
  MRNUMBER = {3957361},
MRREVIEWER = {Niranjan\ Balachandran},
       DOI = {10.1002/rsa.20811},
       URL = {https://doi.org/10.1002/rsa.20811},
}

@article {MolloyTF,
    AUTHOR = {Molloy, Michael},
     TITLE = {The list chromatic number of graphs with small clique number},
   JOURNAL = {J. Combin. Theory Ser. B},
  FJOURNAL = {Journal of Combinatorial Theory. Series B},
    VOLUME = {134},
      YEAR = {2019},
     PAGES = {264--284},
      ISSN = {0095-8956},
   MRCLASS = {05C15 (05C69)},
  MRNUMBER = {3906639},
MRREVIEWER = {Hsin-Hao Lai},
       DOI = {10.1016/j.jctb.2018.06.007},
       _URL = {https://doi-org.proxy.lib.sfu.ca/10.1016/j.jctb.2018.06.007},
}

@article {Johansson,
    AUTHOR = {Anders Johansson},
     TITLE = {Asymptotic choice number for triangle free graphs},
      YEAR = {1996},
}

@article {JohanssonKr,
    AUTHOR = {Anders Johansson},
     TITLE = {The choice number of sparse graphs},
      YEAR = {1996},
}

@book {Mitzenmacher,
    AUTHOR = {Mitzenmacher, Michael and Upfal, Eli},
     TITLE = {Probability and computing},
   EDITION = {Second},
      NOTE = {Randomization and probabilistic techniques in algorithms and
              data analysis},
 PUBLISHER = {Cambridge University Press, Cambridge},
      YEAR = {2017},
     PAGES = {xx+467},
      ISBN = {978-1-107-15488-9},
   MRCLASS = {68-01 (60C05 60G42 60J10 60K25 62H30 68W20 68W40)},
  MRNUMBER = {3674428},
}

@article {Hajnal,
    AUTHOR = {Kierstead, H. A. and Kostochka, A. V.},
     TITLE = {A short proof of the {H}ajnal-{S}zemer\'edi theorem on
              equitable colouring},
   JOURNAL = {Combin. Probab. Comput.},
  FJOURNAL = {Combinatorics, Probability and Computing},
    VOLUME = {17},
      YEAR = {2008},
    NUMBER = {2},
     PAGES = {265--270},
      ISSN = {0963-5483,1469-2163},
   MRCLASS = {05C15},
  MRNUMBER = {2396352},
MRREVIEWER = {Jian-Liang\ Wu},
       DOI = {10.1017/S0963548307008619},
       URL = {https://doi.org/10.1017/S0963548307008619},
}

@inproceedings{KST,
  title={On a problem of Zarankiewicz},
  author={K{\H{o}}v{\'a}ri, P and T S{\'o}s, Vera and Tur{\'a}n, P{\'a}l},
  booktitle={Colloquium Mathematicum},
  volume={3},
  pages={50--57},
  year={1954},
  organization={Polska Akademia Nauk}
}

@article{anderson2023colouring,
  title={Colouring graphs with forbidden bipartite subgraphs},
  author={Anderson, James and Bernshteyn, Anton and Dhawan, Abhishek},
  journal={Combinatorics, Probability and Computing},
  volume={32},
  number={1},
  pages={45--67},
  year={2023},
  publisher={Cambridge University Press}
}

@misc{HurleyPirot,
      title={Colouring locally sparse graphs with the first moment method}, 
      author={François Pirot and Eoin Hurley},
      year={2021},
      eprint={2109.15215},
      archivePrefix={arXiv},
      primaryClass={math.CO},
      url={https://arxiv.org/abs/2109.15215}, 
}

@article {Kang,
    AUTHOR = {Hurley, Eoin and de Joannis de Verclos, R\'emi and Kang, Ross
              J.},
     TITLE = {An improved procedure for colouring graphs of bounded local
              density},
   JOURNAL = {Adv. Comb.},
  FJOURNAL = {Advances in Combinatorics},
      YEAR = {2022},
     PAGES = {Paper No. 7, 33},
      ISSN = {2517-5599},
   MRCLASS = {05C15 (05C35 05C55)},
  MRNUMBER = {4499710},
}

@article {Bonamy,
    AUTHOR = {Bonamy, Marthe and Perrett, Thomas and Postle, Luke},
     TITLE = {Colouring graphs with sparse neighbourhoods: bounds and
              applications},
   JOURNAL = {J. Combin. Theory Ser. B},
  FJOURNAL = {Journal of Combinatorial Theory. Series B},
    VOLUME = {155},
      YEAR = {2022},
     PAGES = {278--317},
      ISSN = {0095-8956,1096-0902},
   MRCLASS = {05C15 (05D40)},
  MRNUMBER = {4392275},
MRREVIEWER = {Niranjan\ Balachandran},
       DOI = {10.1016/j.jctb.2022.01.009},
       URL = {https://doi.org/10.1016/j.jctb.2022.01.009},
}

@article{BruhnJoos,
  title={A stronger bound for the strong chromatic index},
  author={Bruhn, Henning and Joos, Felix},
  journal={Combinatorics, Probability and Computing},
  volume={27},
  number={1},
  pages={21--43},
  year={2018},
  publisher={Cambridge University Press}
}

@article{andersen1992strong,
  title={The strong chromatic index of a cubic graph is at most 10},
  author={Andersen, Lars D{\o}vling},
  journal={Discrete Mathematics},
  volume={108},
  number={1-3},
  pages={231--252},
  year={1992},
  publisher={Elsevier}
}

@article{horak1993induced,
  title={Induced matchings in cubic graphs},
  author={Hor{\'a}k, Peter and Qing, He and Trotter, William T},
  journal={Journal of Graph Theory},
  volume={17},
  number={2},
  pages={151--160},
  year={1993},
  publisher={Wiley Online Library}
}

@article{cranston2006strong,
  title={Strong edge-coloring of graphs with maximum degree 4 using 22 colors},
  author={Cranston, Daniel W},
  journal={Discrete Mathematics},
  volume={306},
  number={21},
  pages={2772--2778},
  year={2006},
  publisher={Elsevier}
}

@article{faudree1990strong,
  title={The strong chromatic index of graphs},
  author={Faudree, Ralph J and Schelp, Richard H and Gy{\'a}rf{\'a}s, Andr{\'a}s and Tuza, Zsolt},
  journal={Ars Combinatoria},
  volume={29},
  pages={205--211},
  year={1990},
  publisher={CHARLES BABBAGE RES CTR PO BOX 272 ST NORBERT POSTAL STATION, WINNIPEG MB~…}
}

@article{kaiser2014distance,
  title={The distance-t chromatic index of graphs},
  author={Kaiser, Tom{\'a}{\v{s}} and Kang, Ross J},
  journal={Combinatorics, Probability and Computing},
  volume={23},
  number={1},
  pages={90--101},
  year={2014},
  publisher={Cambridge University Press}
}

@article{kim1995brooks,
  title={On Brooks' theorem for sparse graphs},
  author={Kim, Jeong Han},
  journal={Combinatorics, Probability and Computing},
  volume={4},
  number={2},
  pages={97--132},
  year={1995},
  publisher={Cambridge University Press}
}

@article{davies2020graph,
  title={Graph structure via local occupancy},
  author={Davies, Ewan and Kang, Ross J and Pirot, Fran{\c{c}}ois and Sereni, Jean-S{\'e}bastien},
  journal={arXiv preprint arXiv:2003.14361},
  year={2020}
}

@article{dhawan2025bounds,
  title={Bounds for the Independence and Chromatic Numbers of Locally Sparse Graphs},
  author={Dhawan, Abhishek},
  journal={Annals of Combinatorics},
  pages={1--28},
  year={2025},
  publisher={Springer},
  doi={https://doi.org/10.1007/s00026-025-00778-7}
}

@article{anderson2025coloring,
  title={Coloring graphs with forbidden almost bipartite subgraphs},
  author={Anderson, James and Bernshteyn, Anton and Dhawan, Abhishek},
  journal={Random Structures \& Algorithms},
  volume={66},
  number={4},
  pages={e70012},
  year={2025},
  publisher={Wiley Online Library}
}

@inproceedings{PS15,
  title={Fast distributed coloring algorithms for triangle-free graphs},
  author={Pettie, Seth and Su, Hsin-Hao},
  booktitle={International Colloquium on Automata, Languages, and Programming},
  pages={681--693},
  year={2013},
  organization={Springer}
}

@ARTICLE{BollobasBound,
	AUTHOR = "Bollob\'{a}s, B.",
	TITLE = "{Chromatic number, girth and maximal degree}",
	JOURNAL = "Discrete Math.",
	date = "1978",
	volume = {24},
	pages = {311--314},
}

@article{Achlioptas,
	author = {D. Achlioptas and A. Coja-Oghlan},
	title = {Algorithmic barriers from phase transitions},
	journaltitle = {IEEE Symposium on Foundations of Computer Science (FOCS)},
	date = {2008},
	pages = {793--802},
	addendum = {Full version: \url{https://arxiv.org/abs/0803.2122}},
}

@article{Zdeborova,
	author = {L. Zdeborov{\'{a}} and F. Krz{\k{a}}ka{\l}a},
	title = {Phase transitions in the coloring of random graphs},
	journaltitle = {Phys. Rev. E},
	date = {2007},
	volume = {76},
	pages = {031131},
}

@article{RV,
	author = {M. Rahman and B. Vir{\'{a}}g},
	title = {Local algorithms for independent sets are half-optimal},
	journaltitle = {Ann. Probab.},
	volume = {45},
	number = {3},
	pages = {1543--1577},
	date = {2017},
}

@article{wein2020optimal,
  title={Optimal low-degree hardness of maximum independent set},
  author={Wein, Alexander S},
  journal={Mathematical Statistics and Learning},
  volume={4},
  number={3},
  pages={221--251},
  year={2022}
}

@article{bandeira2018notes,
  title={Notes on computational-to-statistical gaps: predictions using statistical physics},
  author={Bandeira, Afonso and Perry, Amelia and Wein, Alexander S},
  journal={Portugaliae mathematica},
  volume={75},
  number={2},
  pages={159--186},
  year={2018}
}

@article{gamarnik2022disordered,
  title={Disordered systems insights on computational hardness},
  author={Gamarnik, David and Moore, Cristopher and Zdeborov{\'a}, Lenka},
  journal={Journal of Statistical Mechanics: Theory and Experiment},
  volume={2022},
  number={11},
  pages={114015},
  year={2022},
  publisher={IOP Publishing}
}

@article{gamarnik2025turing,
  title={Turing in the Shadows of Nobel and Abel: An Algorithmic Story Behind Two Recent Prizes},
  author={Gamarnik, David},
  journal={Notices of the American Mathematical Society},
  volume={72},
  number={05},
  year={2025},
  publisher={American Mathematical Society}
}

@article{gamarnik2021overlap,
  title={The overlap gap property: A topological barrier to optimizing over random structures},
  author={Gamarnik, David},
  journal={Proceedings of the National Academy of Sciences},
  volume={118},
  number={41},
  year={2021},
  publisher={National Acad Sciences}
}

@ARTICLE{AKSConjecture,
	AUTHOR = "Alon, N. and Krivelevich, M. and Sudakov, B.",
	TITLE = "{Coloring graphs with sparse neighborhoods}",
	JOURNAL = "J. Comb. Theory",
	series = {B},
	date = "1999",
	volume = {77},
	pages = {73--82},
}

@inproceedings{Nibble,
  title={Graph and hypergraph colouring via nibble methods: A survey},
  author={Kang, Dong Yeap and Kelly, Tom and K{\"u}hn, Daniela and Methuku, Abhishek and Osthus, Deryk},
  booktitle={European Congress of Mathematics},
  pages={771--823},
  year={2023}
}

@article{vizing1976coloring,
  title={Coloring the vertices of a graph in prescribed colors},
  author={V.G. Vizing},
  journal={Diskret. Analiz},
  volume={29},
  number={3},
  pages={10},
  year={1976}
}

@article{erdos1979choosability,
  title={Choosability in graphs},
  author={P. Erd{\H{o}}s and A.L. Rubin and H. Taylor},
  journal={Congr. Numer},
  volume={26},
  number={4},
  pages={125--157},
  year={1979}
}

@article{haxell2001note,
  title={A note on vertex list colouring},
  author={Haxell, P.E.},
  journal={Comb. Probab. Comput.},
  volume={10},
  number={4},
  pages={345--347},
  date={2001},
}

@article {Hoory,
    AUTHOR = {Hoory, Shlomo and Linial, Nathan and Wigderson, Avi},
     TITLE = {Expander graphs and their applications},
   JOURNAL = {Bull. Amer. Math. Soc. (N.S.)},
  FJOURNAL = {American Mathematical Society. Bulletin. New Series},
    VOLUME = {43},
      YEAR = {2006},
    NUMBER = {4},
     PAGES = {439--561},
      ISSN = {0273-0979,1088-9485},
   MRCLASS = {68Q15 (00-02 05C25 05C80 60G50 68Q17 68R10)},
  MRNUMBER = {2247919},
MRREVIEWER = {Mark\ R.\ Jerrum},
       DOI = {10.1090/S0273-0979-06-01126-8},
       URL = {https://doi-org.proxy2.library.illinois.edu/10.1090/S0273-0979-06-01126-8},
}

@misc{BDMW,
      title={Toward Vu's conjecture}, 
      author={Peter Bradshaw and Abhishek Dhawan and Abhishek Methuku and Michael C. Wigal},
      year={2025},
      eprint={2508.16818},
      archivePrefix={arXiv},
      primaryClass={math.CO},
      url={https://arxiv.org/abs/2508.16818}, 
}

@misc{delcourt2022finding,
  title={Finding an almost perfect matching in a hypergraph avoiding forbidden submatchings},
  author={Delcourt, Michelle and Postle, Luke},
  year={2022},
  eprint={2204.08981},
  archivePrefix={arXiv},
  primaryClass={math.CO},
  url={https://arxiv.org/abs/2204.08981}, 
}

@misc{li2022chromatic,
  title={The chromatic number of triangle-free hypergraphs},
  author={Li, Lina and Postle, Luke},
  year={2022},
  eprint={2202.02839},
  archivePrefix={arXiv},
  primaryClass={math.CO},
  url={https://arxiv.org/abs/2202.02839}, 
}

@article{cranston2023coloring,
  title={Coloring, List Coloring, and Painting Squares of Graphs (and Other Related Problems)},
  author={Cranston, Daniel W},
  journal={The Electronic Journal of Combinatorics},
  pages={DS25--Apr},
  year={2023}
}

@article{huang2018strong,
  title={Strong Chromatic Index of Graphs With Maximum Degree Four},
  author={Huang, Mingfang and Santana, Michael and Yu, Gexin},
  journal={The Electronic Journal of Combinatorics},
  volume={25},
  number={3},
  pages={3--31},
  year={2018}
}

@article{vu2002general,
  title={A general upper bound on the list chromatic number of locally sparse graphs},
  author={Vu, Van H},
  journal={Combinatorics, Probability and Computing},
  volume={11},
  number={1},
  pages={103--111},
  year={2002},
  publisher={Cambridge University Press}
}

@article{bondy1974cycles,
  title={Cycles of even length in graphs},
  author={Bondy, John A and Simonovits, Mikl{\'o}s},
  journal={Journal of Combinatorial Theory, Series B},
  volume={16},
  number={2},
  pages={97--105},
  year={1974},
  publisher={Elsevier}
}

\end{document}